\documentclass[10.5pt]{article}

\usepackage{leftidx}
\usepackage{amsmath}
\usepackage{amscd}
\usepackage{amssymb}
\usepackage{rotating}
\usepackage{amsfonts}
\usepackage{amsthm}
\usepackage{verbatim}
\usepackage{stmaryrd}
\usepackage[top=1in, bottom=1in, left=1.25in, right=1.25in]{geometry}

\usepackage{titlesec}
       \titleformat{\chapter}[display]
             {\normalfont\Large\bfseries}{\thechapter}{11pt}{\Large}
       \titleformat{\section}
             {\normalfont\large\bfseries}{\thesection}{10.5pt}{\large}
       \titlespacing*{\chapter}{0pt}{0pt}{15pt} %left, beforesep, aftersep, right
       \titlespacing*{\section}{0pt}{3.5ex plus 1ex minus .2ex}{2.3ex plus .2ex}

    \usepackage[titletoc]{appendix}

\input xy
\xyoption{all}

\input xy
\xyoption{all}

\allowdisplaybreaks[1]

\newcommand{\ev}{\mathrm{ev}}

\newcommand{\cM}{{\mathcal M}}

\newcommand{\Mbar}{\overline{\cM}}
\newcommand{\Mtilde}{\widetilde{\cM}}

\newcommand{\pqed}{\hfill\qedsymbol\\}
\newcommand{\nn}{\nonumber}
\newcommand{\Res}{\mathrm{Res}}
\newcommand{\DRT}{\mathfrak{DRT}}
\newcommand{\RDRT}{\mathfrak{RDRT}}
\newcommand{\CDRT}{\mathfrak{CDRT}}
\newcommand{\DOL}{\mathfrak{DOL}}
\newcommand{\DS}{\mathfrak{DS}}
\newcommand{\VC}{\overline{\mathrm{Cont}}}
\newcommand{\Aut}{\mathrm{Aut}}
\newcommand{\Cont}{\mathrm{Cont}}
\newcommand{\Ver}{\mathrm{Ver}}
\newcommand{\Edg}{\mathrm{Edg}}
\newcommand{\val}{\mathrm{val}}
\newcommand{\vir}{\mathrm{vir}}
\newcommand{\Nb}{\mathrm{Nb}}
\newcommand{\tM}{\widetilde{\mathcal{M}}}

\newtheorem{theorem}{Theorem}[section]
\newtheorem{theorem/definition}{Theorem/Definition}[section]

\newtheorem{proposition}{Proposition}[section]
\newtheorem{lemma}{Lemma}[section]

\newtheorem{corollary}{Corollary}[section]
\newtheorem{Conjecture}{Conjecture}

\theoremstyle{remark}
\newtheorem{remark}{Remark}[section]

\theoremstyle{definition}
 
\newtheorem{definition}{Definition}[section]

\newcommand{\be}{\begin{equation}}
\newcommand{\ee}{\end{equation}}
\newcommand{\bea}{\begin{eqnarray}}
\newcommand{\eea}{\end{eqnarray}}
\newcommand{\ben}{\begin{eqnarray*}}
\newcommand{\een}{\end{eqnarray*}}
\newcommand{\bet}{\begin{equation}
\begin{split}}
\newcommand{\eet}{\end{split}
\end{equation}}

\begin{document}
\title
{Localized Standard Versus Reduced Formula and  Genus One Gromov-Witten Invariants of Local Calabi-Yau Manifolds}
\author{\normalsize Xiaowen Hu}
\date{}
\maketitle

\begin{abstract}
For local Calabi-Yau manifolds which are total spaces of vector bundle over algebraic GKM manifolds, we propose a formal definition of reduced Genus one Gromov-Witten invariants, by assigning contributions from the refined decorated rooted trees. We show that this definition satisfies a localized version of the standard versus reduced formula, whose global version in the compact cases is due to A. Zinger. As an application we prove the conjecture in a previous article on the genus one Gromov-Witten invariants of local Calabi-Yau manifolds which are total spaces of concave splitting vector bundles over projective spaces. In particular, we prove the mirror formulae for genus one Gromov-Witten invariants of $K_{\mathbb{P}^{2}}$ and $K_{\mathbb{P}^{3}}$, conjectured by Klemm, Zaslow and Pandharipande. In the appendix we derive the modularity of genus one Gromov-Witten invariants for the local $\mathbb{P}^{2}$ as a consequence of the results on  Ramanujan's cubic transformation. Inspired by the localized standard versus reduced formula, we show that the ordinary genus one Gromov-Witten invariants of Calabi-Yau hypersurfaces in projective spaces can be computed by  virtual localization and quantum hyperplane property after the contribution of a genus one vertex is replaced by a modified one.

\end{abstract}

\maketitle

\section{Introduction}
The computations of the Gromov-Witten invariants of Calabi-Yau manifolds play an important role in enumerative geometry and mirror symmetry. In genus zero, for Calabi-Yau complete intersections in projective spaces, we can use the \emph{quantum hyperplane property} to write the integration of the virtual fundamental class as  a \emph{twisted Gromov-Witten invariants} of the ambient projective spaces, see \cite{CK}, \cite{Givental1}, \cite{LLY} and the references therein. \\

In genus one, this approach does not work in a straightforward way, since the quantum hyperplane property does not hold. In \cite{Zinger2}, \cite{ZingerSvR}, A. Zinger defined the reduced genus one Gromov-Witten invariants for symplectic manifolds and found a relation between the reduced and the ordinary Gromov-Witten invariants (standard versus reduced formulae, see \cite[theorem 1A, theorem 1B]{ZingerSvR}).
Furthermore, J. Li and A. Zinger showed that the reduced genus one Gromov-Witten invariants satisfy the hyperplane property for complete intersections in projective spaces (\cite{LZ}). So by the standard versus reduced formula, we can reduced the computation of the genus one Gromov-Witten invariants of a complete intersection $X$ in $\mathbb{P}^{n-1}$ to the computation of genus zero Gromov-Witten invariants of $X$, and integrations of some classes on $\Mbar_{1,k}^{0}(\mathbb{P}^{n-1},d)$, the main component of $\Mbar_{1,k}(\mathbb{P}^{n-1},d)$. The latter integrations can be computed by equivariant localizations on a natural desingularization of $\Mbar_{1,k}^{0}(\mathbb{P}^{n-1},d)$ (\cite{VZ}). Zinger completed the computations for the genus one Gromov-Witten invariants of  Calabi-Yau hypersurfaces in $\mathbb{P}^{n-1}$ by a clever use of properties of symmetric functions and the residue theorem on $S^{2}$ (\cite{Zinger1}). Following this approach, A. Popa computed the genus one  Gromov-Witten invariants of  Calabi-Yau complete intersections in $\mathbb{P}^{n-1}$ (\cite{Popa1}). To generalize this method to complete intersections in more general spaces (such as toric varieties, flag varieties, ...), there are at least two difficulties:
\begin{enumerate}
  \item{} The desingularizations in \cite{VZ} or \cite{HL} can be extended to products of projective spaces in a straightforward way, and may also be extended to toric varieties with some efforts. For Grassmannians and more general flag varieties we need some new ideas.
  \item{} The combinatorial computations in \cite{Zinger1} and \cite{Popa1} rely heavily on the $S_{n}$-symmetry of the toric geometry of $\mathbb{P}^{n-1}$. For more general spaces we have less symmetries so we can not directly make use of the properties of the symmetric rational functions.
\end{enumerate}

The above discussions concern the so called \emph{compact} cases. For a \emph{local Calabi-Yau manifold} $X$, which (in a narrow sense) means the total space of an equivariant concave vector bundle $E$ over a compact algebraic GKM (Goresky-Kottwitz-MacPherson) manifold $Y$ with $c_{1}(X)=0$, one can use virtual localization to compute the genus one Gromov-Witten invariants to any degree; the computations reduce to a purely combinatorial issue. However it is not easy to obtain a closed formula (even via the \emph{mirror map}). A natural idea is to follow Zinger's approach in the compact cases. For example, for $E=K_{\mathbb{P}^{n-1}}$,
\begin{comment}one can define the reduced Gromov-Witten invariants by integrating the obstruction bundle $R^{1}\pi_{*}f^{*}E$ against the main component $\Mbar_{1,0}^{0}(\mathbb{P}^{n-1},d)$.Then  \end{comment}
one can try to prove a result similar to \cite{LZ}, to extend the standard versus reduced formula (SvR for short) to the local case, and to make the localization computations over the desigularization of $\Mbar_{1,0}^{0}(\mathbb{P}^{n-1},d)$. This is just the approach proposed in \cite{Hu}. For example, we compute
\ben
\int_{[\Mbar^{0}_{1,0}(\mathbb{P}^{1},d)]^{\vir}}e\big(R^{1}\pi_{*}f^{*}(\mathcal{O}(-1)\oplus\mathcal{O}(-1))\big)
\een
by linearize $\mathcal{O}(-1)\oplus\mathcal{O}(-1)$ as in  \cite[section 3.1]{Hu}. Only the one-edge graphs
$$ \Gamma_{ij}=\xy
(0,0); (10,0), **@{-};
(0,0)*+{\circ};(10,0)*+{\bullet};(5,2)*+{d};
(0,3)*+{i};(10,3)*+{j};
\endxy,
$$
have nonzero contributions, where $(i,j)=(1,2)$ or $(2,1)$.  The contribution of $\Gamma_{12}$ is
\ben
&&\int_{\widetilde{\mathcal{M}}_{1,1}}\frac{\prod_{k=1}^{2}\Big((-\lambda+\alpha_{k}-\alpha_{1})
\prod_{a=1}^{d-1}(\alpha_{k}-\alpha_{2}+a\cdot\frac{\alpha_{2}-\alpha_{1}}{d})\Big)}{\frac{\alpha_{1}-\alpha_{2}}{d}-\tilde{\psi}}\\
&&\cdot\frac{\frac{\alpha_{2}-\alpha_{1}}{d}(\frac{\alpha_{2}-\alpha_{1}}{d}+\alpha_{1}-\alpha_{2})}
{d\cdot \big(\frac{d!}{d^{d}}\big)^{2}(\alpha_{1}-\alpha_{2})^{d}(\alpha_{2}-\alpha_{1})^{d}}\\
&=&\frac{d-1}{24d^{2}},
\een
where for the notations we refer the reader to \cite{Zinger3}. So we have
\bea\label{2014001}
\int_{[\Mbar^{0}_{1,0}(\mathbb{P}^{1},d)]^{\vir}}e\big(R^{1}\pi_{*}f^{*}(\mathcal{O}(-1)\oplus\mathcal{O}(-1))\big)=\frac{d-1}{12d^{2}},
\eea
which is  \cite[(2.50)]{PZ}. It is well known (see e.g., \cite{GP}) that the genus zero and genus one Gromov-Witten invariants of the resolved conifold $X=\mathcal{O}(-1)\oplus\mathcal{O}(-1)\rightarrow \mathbb{P}^{1}$ are
\ben
N_{0,d}^{X}=\frac{1}{d^{3}}, & N_{1,d}^{X}=\frac{1}{12d}.
\een
So the reduced genus one Gromov-Witten invariants of $X$ defined and computed as above do not satisfy the SvR for compact Calabi-Yau threefold
\bea\label{2014002}
N_{1,d}=N_{1,d}^{0}+\frac{1}{12}N_{0,d},
\eea
as proved in \cite{Zinger2}. But we expect  (\ref{2014002}) and more generally the SvR of Zinger to be valid for a suitable defined reduced genus one Gromov-Witten invariants for the local Calabi-Yau manifolds, and when it holds the combinatorial computations might goes as in \cite{Zinger1} to give a mirror formula.\\

 The novel point in this paper is to define the reduced genus one Gromov-Witten invariants for the \emph{algebraic} GKM  \emph{manifolds} via the localization data. More precisely, let $X$ be a local Calabi-Yau manifold, for every \emph{decorated one loop graph} and every \emph{decorated rooted tree} $\Gamma$, we associate formally a contribution $\Cont_{\Gamma}(N_{1,d}^{0;X})$, and define the \emph{formal reduced genus one Gromov-Witten invariants} of $X$ by
\bea\label{168}
N_{1,d}^{0;X}:=\sum_{\Gamma\in \DOL_{\emptyset}^{d}}\Cont_{\Gamma}(N_{1,d}^{0;X})+\sum_{\widetilde{\Gamma}\in \RDRT_{\emptyset}^{d}/\sim}\Cont_{\widetilde{\Gamma}}(N_{1,d}^{0;X}),
\eea
where $\DOL_{\emptyset}^{d}$ and $\RDRT_{\emptyset}^{d}$ represent the set of decorated one loop graphs and the set of decorated rooted trees respectively.
 The precise definition is given in section 2.3, definition \ref{26}.\\

  The key in this definition is a formal definition (\ref{24}) of the localization contribution of the obstruction bundel $e(\mathcal{U}_{1}^{\prime})$, as a substitution for the localization data of $e\big(R^{1}\pi_{*}f^{*}(\cdot)\big)$. This can be viewed as a formal analog for the exact sequence in \cite[theorem 1.2(2)]{VZ}. By this definition, for the resolved conifold, the contribution of $\Gamma_{ij}$ is
\ben
&&\int_{\widetilde{\mathcal{M}}_{1,1}}\frac{\prod_{k=1}^{2}\Big((\frac{\alpha_{2}-\alpha_{1}}{d}+\alpha_{k}-\alpha_{1})
\prod_{a=1}^{d-1}(\alpha_{k}-\alpha_{2}+a\cdot\frac{\alpha_{2}-\alpha_{1}}{d})\Big)}{\frac{\alpha_{1}-\alpha_{2}}{d}-\tilde{\psi}}\\
&&\cdot\frac{\frac{\alpha_{2}-\alpha_{1}}{d}(\frac{\alpha_{2}-\alpha_{1}}{d}+\alpha_{1}-\alpha_{2})}
{d\cdot \big(\frac{d!}{d^{d}}\big)^{2}(\alpha_{1}-\alpha_{2})^{d}(\alpha_{2}-\alpha_{1})^{d}}\\
&=&\frac{d^{2}-1}{24d^{3}}.
\een
Then
\ben
N_{1,d}^{X;0}=\frac{d^{2}-1}{12d^{3}}
\een
and (\ref{2014002}) holds.\\

 More generally, one can define formal reduced genus one Gromov-Witten invariants with insertions, and  for more general $X$, not necessarily Calabi-Yau ones.
 The first main theorem of this article is
\begin{theorem}\label{169}
Let $X$ be  the total space of an equivariant concave vector bundle $E$ over a compact algebraic GKM manifold $Y$, 
and $\mu_{1},\cdots,\mu_{|J|}\in H_{\mathbb{T}}^{*}(Y)$.
For every decorated rooted tree $\Gamma\in \DRT_{J}^{d}$, we have
\bea\label{170}
&&\Cont_{\Gamma}(\langle\mu_{1},\cdots,\mu_{|J|}\rangle_{1,J,d}^{X})\nn\\
&=&\Cont_{\Gamma}(\langle\mu_{1},\cdots,\mu_{|J|}\rangle_{1,J,d}^{0;X})
+\sum_{m\geq 1}\sum_{J^{\prime}\subset J}\Big[\frac{(-1)^{m+|J^{\prime}|}m^{|J^{\prime}|}(m-1)!}{24}\nn\\
&&
\sum_{p=0}^{n+l-1-|J^{\prime}|-2m}
\Cont_{\Gamma}\Big(\langle \eta_{p}\mathbf{c}_{n+l-1-|J^{\prime}|-2m-p}(TX);\mu_{1},\cdots,\mu_{|J|}\rangle_{(m,J-J^{\prime},d)}^{X}\Big)\Big],
\eea
where $\mathbf{c}_{k}$ denotes the equivariant Chern class.
\end{theorem}
We refer the reader to section 3 for the precise meaning of this theorem. We call (\ref{170}) the localized standard versus reduced formula (LSvR for short). As a corollary of (\ref{168}) and (\ref{170}), we have

\begin{corollary}\label{171}
In the set-up of theorem \ref{169}, 
\bea\label{172}
&&\langle\mu_{1},\cdots,\mu_{|J|}\rangle_{1,J,d}^{X}\nn\\
&=&\langle\mu_{1},\cdots,\mu_{|J|}\rangle_{1,J,d}^{0;X}
+\sum_{m\geq 1}\sum_{J^{\prime}\subset J}\Big[\frac{(-1)^{m+|J^{\prime}|}m^{|J^{\prime}|}(m-1)!}{24}\nn\\
&&\sum_{p=0}^{n+l-1-|J^{\prime}|-2m}\langle\eta_{p}\mathbf{c}_{n+l-1-|J^{\prime}|-2m-p}(TX);\mu_{1},\cdots,\mu_{|J|}
\rangle_{(m,J-J^{\prime},d)}^{X}\Big].
\eea
\end{corollary}

Let $\gamma_{1},\cdots,\gamma_{|J|}\in H^{*}(Y)$, and suppose they have  equivariant liftings $\mu_{1},\cdots,\mu_{|J|}\in H_{\mathbb{T}}^{*}(Y)$
(equivariant liftings always exist when $Y$ is \emph{equivariantly formal} \cite{GKM}). If the degrees of $\gamma_{1},\cdots,\gamma_{|J|}$ satisfy the dimension constraint of the non-equivariant genus one Gromov-Witten invariants,  $\langle\mu_{1},\cdots,\mu_{|J|}\rangle_{1,J,d}^{0;X}$ is a complex number, i.e.,  involving no equivariant parameters, so we can define the reduced genus one Gromov-Witten invariant of $X$ by
\bea\label{nonequi}
\langle\gamma_{1},\cdots,\gamma_{|J|}\rangle_{1,J,d}^{0;X}:=\langle\mu_{1},\cdots,\mu_{|J|}\rangle_{1,J,d}^{0;X}.
\eea
Then we have a non-equivariant version of (\ref{172})
\bea\label{172nonequi}
&&\langle\gamma_{1},\cdots,\gamma_{|J|}\rangle_{1,J,d}^{X}\nn\\
&=&\langle\gamma_{1},\cdots,\gamma_{|J|}\rangle_{1,J,d}^{0;X}
+\sum_{m\geq 1}\sum_{J^{\prime}\subset J}\Big[\frac{(-1)^{m+|J^{\prime}|}m^{|J^{\prime}|}(m-1)!}{24}\nn\\
&&\sum_{p=0}^{n+l-1-|J^{\prime}|-2m}\langle\eta_{p}c_{n+l-1-|J^{\prime}|-2m-p}(TX);\gamma_{1},\cdots,\gamma_{|J|}
\rangle_{(m,J-J^{\prime},d)}^{X}\Big].
\eea
Consequently, the definition (\ref{nonequi}) is independent of the choices of the equivariant liftings and the $\mathbb{T}$-actions on $Y$ and $E$.
Needless to say, it is very desirable to seek a geometric way to define the reduced genus one Gromov-Witten invariants which coincide with our formal one, and prove (\ref{172nonequi}) geometrically.\\

Note that Zinger's SvR has two (equivalent) versions, one involves the $\eta_{p}$-classes, the other the $\tilde{\eta}_{p}$-classes. Zinger's computation for the genus one GW invariants of CY hypersurfaces does not involve the $\tilde{\eta}_{p}$-classes. But our computation for local invariants does need the LSvR for $\tilde{\eta}_{p}$-classes, because our definition of reduced genus one local Gromov-Witten invariants is formal, from which we cannot prove the divisor equation in the usual way; but we can deduce the divisor equation for reduced invariants from corollary \ref{171} and the divisor equation for the ordinary local Gromov-Witten invariants and for the invariants involve $\tilde{\eta}_{p}$-classes (lemma \ref{191} and \ref{192}).\\

We prove the LSvR for arbitrary equivariant \emph{concave} vector bundles over algebraic GKM manifolds. This combinatorial  approach has  an advantage that for local Calabi-Yau manifolds it bypass the first difficulty discussed above on the desingularization of the moduli spaces of stable maps of genus one.\\

The corollary \ref{171} enables us to compute genus one Gromov-Witten invariants of
\bea\label{40}
X=\mathrm{Tot}\big(E=\bigoplus_{k=1}^{l}\mathcal{O}_{\mathbb{P}^{n-1}}(-a_{i})\rightarrow \mathbb{P}^{n-1}\big),
\eea
with $a_{k}>0$ for $1\leq k\leq l$ and $\sum_{k=1}^{l}a_{k}=n$; the crucial point is that while the computation of the lefthand side of (\ref{172}) via the virtual localization does not directly lead to a mirror formula, we can compute the righthand side of (\ref{172}) to obtain a mirror formula of the lefthand side, via Zinger's method in \cite{Zinger1}. The genus one mirror formulae for $K_{\mathbb{P}^{2}}$ and $K_{\mathbb{P}^{3}}$ have been conjectured in \cite{KZ} (see also \cite{ABK}) and \cite{KP} respectively. In \cite{Hu}, based on some observations on Zinger's formulae in compact cases, we made a conjecture (generalizing the conjecture for $K_{\mathbb{P}^{2}}$ and $K_{\mathbb{P}^{3}}$) on the genus one mirror formulae for $X$ of the form (\ref{40}), which is now the second main theorem of this article. Let
\bea\label{272}
R(w,t)=e^{wt}\sum_{d\geq 0}e^{dt}\frac{\prod_{k=1}^{l}\prod_{s=0}^{a_{k}d-1}(-a_{k}w-s)}{\prod_{s=1}^{d}(w+s)^{n}},
\eea
and for $q\ge p\geq 0$, let
\bea\label{299}
I_{p,q}(t)=\frac{d}{dt}\Big(\frac{I_{p-1,q}(t)}{I_{p-1,p-1}(t)}\Big).
\eea
Denote $I_{p}(t)=I_{p,p}(t)$ for $p\geq 0$. Let
\bea\label{300}
T=I_{0,1}(t).
\eea
\begin{theorem}\label{267}(=\textrm{Theorem \ref{297}})
\begin{multline}\label{265}
\sum_{d=1}^{\infty}e^{dT}N_{1,d}^{X}=\frac{n}{48}\Big(n-1-2\sum_{k=1}^{l}\frac{1}{a_{k}}\Big)\big(T-t\big)\\
-\left\{\begin{array}{lll}
\frac{n+l}{48}\log (1-\prod_{k=1}^{l}(-a_{k})^{a_{k}}e^{t})
+\sum_{p=l}^{\frac{n+l-2}{2}}\frac{(n+l-2p)^{2}}{8}\log I_{p}(e^{t}), & \mathrm{if}& 2\mid(n+l);\\
\frac{n+l-3}{48}\log (1-\prod_{k=1}^{l}(-a_{k})^{a_{k}}e^{t})
+\sum_{p=l}^{\frac{n+l-3}{2}}\frac{(n+l-2p)^{2}-1}{8}\log I_{p}(e^{t}), & \mathrm{if}& 2\nmid(n+l).\\
\end{array}\right.
\end{multline}
\end{theorem}
Since $I_{p}(e^{t})=0$ for $0\leq p\leq l-1$, this theorem is equivalent to  \cite[conjecture 1]{Hu}.\\

 For general concave equivariant vector bundles over algebraic GKM manifolds, the corollary \ref{171} enables us to obtain analogs to proposition \ref{215} and proposition \ref{174}. To obtain a mirror formula, however, one needs additional techniques to overcome the second difficulty discussed above.\\

  For $X=K_{\mathbb{P}^{2}}$, the genus one Gromov-Witten potential can be written as a modular form in suitable modular coordinate on the the modular curve for $\Gamma(3)$. The derivation of this result in \cite{ABK} is a mixture of rigorous mathematics and mirror-symmetry-arguments. To make things clear, we derive this fact from the results on the Ramanujan's cubic transformation, which should have been well-known to experts.\\

It is interesting to note the interplay between the computations of Gromov-Witten invariants of global ($=$compact) and of local Calabi-Yau manifolds.
In principle, the latter should be easier, and the compuation for the local CYs helps us to understand the global CYs. For example, in \cite[section 6]{KP}, they made use of the localization computation for $K_{\mathbb{P}^{3}}$ to fix the universal behavior at the conifold and thus obtain the genus one Gromov-Witten potential for the CY hypersurface in $\mathbb{P}^{5}$ via the B-model. Conversely, we used the result on the genus one Gromov-Witten potential for the compact CY complete intersections of \cite{Zinger1} and \cite{Popa1} to fix the universal behavior at the conifolds, and further observed that formulation of the group of terms such as $\sum_{p=l}^{\frac{n+l-2}{2}}\frac{(n+l-2p)^{2}}{8}\log I_{p}(e^{t})$ in (\ref{265}) should be \emph{ubiquitous},  thus made the conjecture in \cite{Hu}. Now we have known that the standard versus reduced formula in the local case holds in a \emph{refined} way, what does it \emph{feedback} to the global theory? We investigate this topic in section 3.6 and make some interesting observations and a conjecture (see conjecture \ref{189} and the remark following it). Briefly speaking, the term in the second row of (\ref{182}) gives the \emph{defect} of the \emph{naive} quantum hyperplane theorem in genus one after virtual localization. This gives a hint to answer the question of Givental at the end of \cite{Givental}.\\

  This article  is organized as follows. In section 2 we first fix the terminology for localization computations. We recall the ordinary virtual localization in genus one, and give the localization data for equivariant integrations over $\Mbar_{(m,J)}(Y,d)$. Then we define the  formal reduced genus one local Gromov-Witten invariants. In section 3 we prove the LSvR, from simple cases to the general cases. In section 3.6 we discuss the modified virtual localization for compact Calabi-Yau manifolds. In section 4 we use Givental's result on genus zero mirror symmetry to compute the difference between the standard and the formal reduced Gromov-Witten invariants of (\ref{40}). In section 5 we compute the formal reduced Gromov-Witten invariants of (\ref{40}), following Zinger's method and using results from \cite{Popa1} and \cite{Popa2}. In the appendix we see how to deduce the modularity of $\mathcal{F}_{1}$ for $K_{\mathbb{P}^{2}}$ from Ramanujan's cubic transformation theory.\\

\noindent\textbf{Notations}:
\begin{enumerate}
  \item{} We use $[x^{k}]\Big(f(x, y_1,y_2,\cdots)\Big)$ to represent the coefficient of $x^{k}$ in the Laurent expansion of $f(x)$ at $x=0$.
  \item{} For a vector bundle $E$, let $e(E)$ be its Euler class, and for an equivariant vector bundle $E$, let $\mathbf{e}(E)$ be its equivariant Euler class. Similarly, we denote by $c_{p}(E)$ the $p$-th Chern class of $E$, and $\mathbf{c}_{p}(E)$  the $p$-th equivariant Chern class of $E$.
  \item{} Denote the set of positive integers (resp., non-negative integers ) by $\mathbb{Z}^{>0}$ (resp., $\mathbb{Z}^{\geq 0}$).
  \item{} In section 4 and section 5, we frequently feel convienient to write $e^{T}=Q$, and $e^{t}=q$. This $q$ should not be confused with the $q$ appeared in the subscript of $I_{p,q}(e^{t})$.
  \item{} When we discuss the refined decorated rooted trees, we try our best to follow the terminology in \cite{VZ}, \cite{Zinger1}. However, because we need the Greek letters $\mu$ and $\eta$ at other places, we instead denote the maps $\mu$ and $\eta$ in \cite{Zinger1} by their English analogs $\mathfrak{m}$ and $\mathfrak{e}$.
  \item{} Our notation for the invariant $\langle \eta_{p}\mu_{0};\mu_{1},\cdots,\mu_{|J|}\rangle_{(m,J-J^{\prime},d)}$ is slightly different from \cite{ZingerSvR}, see (\ref{68}).
\end{enumerate}

\emph{Acknoledgement} The author thanks Prof. Jian Zhou for his interest in this work. He thanks Huazhong Ke for a lot of discussions on this topic, and thanks him and Xiaobo Zhuang for reading a rough draft of this article and giving suggestions before they started to the USA. He thanks Prof. Aleksey Zinger for his comments on an earlier version of this article. He also thanks Huijun Fan, Xiaobo Liu, Yunfeng Jiang, Zhilan Wang and Jie Zhou for helpful discussions.

\section{Fixed Loci and localization contributions}
Let $X$ be the total space of a  vector bundle $\pi:E \rightarrow Y$ where $Y$ is a smooth projective variety over $\mathbb{C}$.  We assume that the vector bundle $E$ is concave, which means that for every non-constant map from a curve\footnote{For the purpose of this article we can slightly weaken this condition by restricting the arithmetic genus of $C$ to 0 and 1.} $f:C\rightarrow Y$ we have $H^{0}(C,f^{*}E)=0$. The total Chern class of $X$ is understood as
\bea\label{36}
c(TX)=\pi^{*}\Big(c(E)c(TY)\Big).
\eea
When $c_{1}(E)+c_{1}(TY)=0$, we call $X$ a \emph{local Calabi-Yau} manifold.
Since $E$ is concave, every non-constant map from a curve to $X$ actually map the curve into $Y$,  and when we compute various Gromov-Witten-type invariants of $X$, we are actually working on the moduli spaces related to $Y$, so it is convenient to take $c(TX)$ as a cohomology class (or equivariant cohomology class when we are in an equivariant world) over $Y$, by an abuse of notation. In the explicit computations in section 4 and section 5, we consider the cases $X$ of the form (\ref{40}).\\

Now suppose $Y$ is an \emph{algebraic} GKM \emph{manifold} \footnote{Algebraic GKM manifolds are called \emph{balloon manifolds} in \cite{LLY}.}
of dimension $n-1$, which means  the following data (see \cite{GKM},\cite{LLY}):
 \begin{itemize}
 \item{}There is a $\mathbb{T}=(\mathbb{C}^{*})^{k}$-action on $Y$ such that there are $N$ isolated fixed points, named $P_{1},\cdots,P_{N}$.
 \item{}There are finitely many invariant line  ($\cong \mathbb{P}^{1}$) connecting the fixed points. For each fixed point $P_{i}$, where $1\leq i\leq N$, there are $n-1$ invariant lines ($\cong \mathbb{P}^{1}$) connecting $P_{i}$ to other fixed points. We assume that the weights of these $n$ invariant lines at $P_{i}$ are \emph{pairwise} independent.  For every pair of distinct fixed points there is at most one invariant line connecting them\footnote{This assumption is not necessary, but will  make the presentations less complicated. }; if there exists one for $P_{i}$ and $P_{j}$, we denote it by $\overline{P_{i}P_{j}}(=\overline{P_{j}P_{i}})$, and say that $P_{i}$ and $P_{j}$ are neighboring to each other. Let $\Nb(P_{i})$ be the set of fixed points that are neighboring to $P_{i}$, and for $P_{j}\in \Nb(P_{i})$, we denote the weights of $\overline{P_{i}P_{j}}$ at $P_{i}$ by $\alpha_{i,j}$; we have $\alpha_{i,j}=-\alpha_{j,i}$.

 \end{itemize}
     Denote the equivariant cohomology ring with rational coefficients of $Y$ by $H_{\mathbb{T}}^{*}(Y)$. The equivariant cohomology ring of  a point is denoted by $\mathbb{Q}[\alpha_{1},\cdots,\alpha_{k}]$, and its quotient field by $\mathbb{Q}_{\alpha}$. Thus $\alpha_{i,j}$ are linear combinations of $\alpha_{1},\cdots,\alpha_{k}$ for $1\leq i\leq N$, $1\leq j\leq n$. For every $P_{i}$, there is an associated restriction map
\bea
|_{P_{i}}:H_{\mathbb{T}}^{*}(Y)\rightarrow \mathbb{Q}[\alpha_{1},\cdots,\alpha_{k}].
\eea
Suppose $E$ is a equivariant concave vector bundle of rank $l$ over $Y$, with a $\mathbb{T}$-linearization such that the weights of $E$ at $P_{i}$ are $\varepsilon_{i,1},\cdots,\varepsilon_{i,l}$.\\

Following \cite{BM}, let
\bea
H_{2}^{+}(Y)=\{\beta\in\mathrm{Hom}(\mathrm{Pic}(Y),\mathbb{Z}):\beta(L)\geq 0  \hspace{0.2cm}\mathrm{for} \hspace{0.2cm} \forall  \hspace{0.2cm} \mathrm{ample}\hspace{0.2cm} L\}.
\eea
For $d\in H_{2}^{+}(Y)$ , on the moduli stack of stable maps $\Mbar_{g,k}(Y,d)$, let $\ev_{i}$ be the $i$-th evaluation map, $\pi:\Mbar_{g,k+1}(Y,d)\rightarrow \Mbar_{g,k}(Y,d)$ be the universal curve and $f:\Mbar_{g,k+1}(Y,d)\rightarrow Y$ the universal stable map. Let
\bea\label{70}
\mathcal{U}_{g}=R^{1}\pi_{*}f^{*}E,
\eea
and for $k\geq 1$ let
\bea\label{71}
\mathcal{U}_{g}^{\prime}=\ev_{1}^{*}(E)\oplus R^{1}\pi_{*}f^{*}E,
\eea
which are both vector bundles over $\Mbar_{g,k}(Y,d)$. The \emph{genus $g$ Gromov-Witten invariants} for  $X$ of the form (\ref{40}) with (primary) insertions $\mu_{1},\cdots,\mu_{k}\in H^{*}(Y)$ are given by
\bea\label{165}
 \langle\mu_{1},\cdots,\mu_{k}\rangle_{g,k,d}^{X}=\Big(\bigwedge_{j=1}^{k}\ev_{j}^{*}\mu_{j}\Big)\wedge e(\mathcal{U}_{g})\cap [\Mbar_{g,k}(Y,d)]^{\vir}.
\eea
In particular, when $g=1$ and $k=0$, the genus one Gromov-Witten invariants  of $X$ are  given by
\bea\label{166}
N_{1,d}^{X}= e(\mathcal{U}_{1})\cap [\Mbar_{1,0}(Y,d)]^{\vir}.
\eea

The $\mathbb{T}$-action on $Y$ naturally induces a $\mathbb{T}$-action on the moduli stack of stable maps $\Mbar_{g,k}(Y,d)$ and some other related moduli spaces. The linearization of $E$ naturally induces linearizations of $\mathcal{U}_{g}$ and $\mathcal{U}_{g}^{\prime}$.  The \emph{equivariant genus $g$ Gromov-Witten invariants} of $X=\mathrm{Tot}(E\rightarrow Y)$ with (primary) insertions $\mu_{1},\cdots,\mu_{k}\in H_{\mathbb{T}}^{*}(Y)$ are given by
\bea\label{167}
 \langle\mu_{1},\cdots,\mu_{k}\rangle_{g,k,d}^{X}=\Big(\bigwedge_{j=1}^{k}\ev_{j}^{*}\mu_{j}\Big)\wedge \mathbf{e}(\mathcal{U}_{g})\cap [\Mbar_{g,k}(Y,d)]_{\mathbb{T}}^{\vir},
\eea
where $[\Mbar_{g,k}(Y,d)]_{\mathbb{T}}^{\vir}$ is the equivariant virtual fundamental class.\\

The $\mathbb{T}$-action on $Y$ also induces a $\mathbb{T}$-action on  some other related moduli spaces as we will see.  By \cite{AB} and \cite{GP}, the equivariant integration (against the fundamental cycle when the moduli space is smooth or the virtual moduli cycle when we have a perfect obstruction theory) of the equivariant cohomologogy classes on these spaces can be computed by \emph{(virtual) localization}, i.e., every fixed locus contributes to the integration, and summing the contributions we obtain the integration. When the integration we are computing is understood, we call the contribution coming from a fixed locus the \emph{localization contribution} of this fixed locus (or of the graph which indexes this fixed locus). \\

 In the following part of this section, we describe the fixed loci of three types of moduli spaces $\Mbar_{1,k}(Y,d)$, $\Mbar_{(m,J)}(Y,d)$, the \emph{formal fixed loci} as an analog to the truly existing fixed loci of $\tM_{1,k}(\mathbb{P}^{n-1},d)$,  and the corresponding (\emph{formal} in section 2.3) localization  contributions of several Gromov-Witten-type invariants. For convenience, we prefer to use $J$ to represent the set of marked points.

\subsection{Fixed loci on $\Mbar_{1,J}(Y,d)$ and  localization contributions}
The fixed loci on $\Mbar_{g,J}(Y,d)$ and  their localization contributions to the equivariant Gromov-Witten invariants of $X$ are well-known, and the reader may refer to, e.g., \cite{GP}, \cite{Kontsevich}. In this subsection we recall the results which are necessary for us and fix the notations.\\

The fixed loci on  $\Mbar_{g,J}(Y,d)$ are indexed by \emph{decorated graphs}. We recall the terminology for decorated graphs in \cite{Zinger1}. A decorated graph (for $Y$) is a tuple $\Gamma=(\Ver,\Edg;\mathfrak{g},\mathfrak{m},\mathfrak{d},\mathfrak{e})$, where
\begin{itemize}
\item{}$(\Ver,\Edg)$ is a graph. More precisely, $\Ver$ is a finite set, $\Edg$ is a finite set of maps, from a finite set $\mathrm{Dom}(\Edg)$ to the set of two-element subsets of $\Ver$. For a vertex $v\in\Ver$ and an edge $e\in\Edg$, if $v$ lies in the image of $e$, we call $v\in e$ by an abuse of notation. Also, if the image of an edge $e\in\Edg$ is $\{v_{1},v_{2}\}$, we call $e=\{v_{1},v_{2}\}$ by an abuse of notation; we should keep in mind that in general there may be more than one edges with the same image. However in this article we mainly discuss the trees, so no confusions arises. The edges containing the vertex $v$ is denoted by $\Edg(v)$, i.e., $\Edg(v)=\{e\in\Edg:v\in e\}$. In addition, we assume that the graph $(\Ver,\Edg)$ is connected in the usual sense. Thus the genus of $(\Ver,\Edg)$ is $\mathfrak{g}(\Ver,\Edg)=1-|\Ver|+|\Edg|$. We use $\Edg^{\Gamma}(v)$ instead of $\Edg(v)$ when we want to emphasize the underlying graph $\Gamma$.
\item{}The map
\ben
\mathfrak{g}:\Ver\rightarrow \mathbb{Z}^{\geq 0}
\een
indicates the genus of the contracted component of the domain curve that a vertex represents.
\item{}The map
\ben
\mathfrak{m}:\Ver\rightarrow [N]
\een
indicates the fixed point which the contracted component maps to. We demand that if $\{v_{1},v_{2}\}\in\Edg$, then
\ben
\mathfrak{m}(v_{1})\neq \mathfrak{m}(v_{2}).
\een
\item{}Denote the free semigroup $\sum (\mathbb{Z}^{\geq 0}\cdot\overline{P_{i}P_{j}})$ by  $\mathbb{B}(Y)$, where the sum runns over the invariant lines of $Y$. The map
\ben
\mathfrak{d}:\Edg\rightarrow \mathbb{B}(Y)
\een
for an edge $e=\{v_{1},v_{2}\}\in\Edg$ takes values in $\mathbb{Z}^{>0}\cdot\overline{P_{\mathfrak{m}(v_{1})}P_{\mathfrak{m}(v_{2})}}$,
indicating the degree of the invariant line that an edge represents. For convenience in later use, let $d(e)\in \mathbb{Z}^{>0}$ such that $\mathfrak{d}(e)=d(e)\cdot\overline{P_{\mathfrak{m}(v_{1})}P_{\mathfrak{m}(v_{2})}}$ when $e=\{v_{1},v_{2}\}$.
\item{}The \emph{label map}
\ben
\mathfrak{e}:J\rightarrow \Ver
\een
indicates on which contracted component a marked point lie.
\item{} The genus of a decorated graph $\Gamma$ is
\ben
\mathfrak{g}(\Gamma)=\mathfrak{g}(\Ver,\Edg)+\sum_{v\in \Ver}\mathfrak{g}(v)=1-|\Ver|+|\Edg|+\sum_{v\in \Ver}\mathfrak{g}(v).
\een
\item{}Let
\ben
\mathfrak{d}(\Gamma)=\sum_{e\in\Edg}\mathfrak{d}(e)\in \mathbb{B}(Y).
\een
There is a canonical map $\mathrm{d}:\mathbb{B}(Y)\rightarrow H_{2}^{+}(Y)$. For $\Gamma$ to represent a fixed locus on $\Mbar_{1,J}(Y,d)$, we need $\mathrm{d}\circ\mathfrak{d}(\Gamma)=d$; we call $\mathrm{d}\circ\mathfrak{d}(\Gamma)$ the degree of $\Gamma$.
\item{}The valence of a vertex $v\in\Ver$ is
\bea
\val(v)=|\Edg(v)|+|\mathfrak{e}^{-1}(v)|.
\eea
\item{}There is a natural projection map $\pi$ from the set of decorated graphs to the set of graphs, mapping $\Gamma$ to $(\Ver,\Edg)$. The automorphism group of $(\Ver,\Edg)$ acts naturally on the set $\pi^{-1}\big((\Ver,\Edg)\big)$, and the stable subgroup associated to $\Gamma$ is called the automorphism group of  $\Gamma$, denoted by $\Aut(\Gamma)$.

\end{itemize}

For $\Mbar_{1,J}(Y,d)$ there are two types of decorated graphs, the \emph{decorated one-loop graphs} and the \emph{decorated rooted trees}. On a decorated one-loop graph every vertex has genus zero. On a decorated rooted tree every vertex except the root has genus zero,  and the root has genus one. So we drop the map $\mathfrak{g}$ in the presentations of  decorated one-loop graphs and decorated rooted trees. We denote the set of decorated one-loop graphs (resp., decorated rooted trees ) of degree $d$ and with the set of marked points $J$ by $\DOL_{J}^{d}(Y)$ (resp., $\DRT_{J}^{d}(Y)$). In the proper context we always have a fixed $Y$, so we drop the notation for $Y$ and simply write $\DOL_{J}^{d}$ (resp., $\DRT_{J}^{d}$).\\

  Let us first consider the contributions from decorated one-loop graphs. 
For $\mu_{1},\cdots,\mu_{|J|}\in H_{\mathbb{T}}^{*}(Y)$, the localization contribution of a decorated one-loop graph $\Gamma$ to
 \bea\label{67}
 \langle\mu_{1},\cdots,\mu_{|J|}\rangle_{1,J,d}^{X}
 \eea
  can be written as a product of  contributions from edges and from vertices together with  a factor coming from the automorphism group $\Aut(\Gamma)$, i.e.,
 \bea\label{11}
 &&\Cont_{\Gamma}\big(\langle\mu_{1},\cdots,\mu_{|J|}\rangle_{1,J,d}^{X}\big)=\frac{1}{|\Aut(\Gamma)|}\nn\\
 &&\cdot\prod_{v\in \Ver}\Cont_{\Gamma;v}\big(\langle\mu_{1},\cdots,\mu_{|J|}\rangle_{1,J,d}^{X}\big)\prod_{e\in \Edg} \Cont_{\Gamma;e}\big(\langle\mu_{1},\cdots,\mu_{|J|}\rangle_{1,J,d}^{X}\big).
 \eea
 where the contribution of an vertex $v\in \Ver$ with $\mathfrak{m}(v)=i$ is
\bea\label{13}
&&\Cont_{\Gamma;v}\big(\langle\mu_{1},\cdots,\mu_{|J|}\rangle_{1,J,d}^{X}\big)=\prod_{j\in \mathfrak{e}^{-1}(v)}\mu_{j}\big|_{P_{i}}\nn\\
&&\cdot\Big(\prod_{k=1}^{l}\varepsilon_{i,k}\prod_{P_{j}\in \Nb(P_{i})}\alpha_{i,j}\Big)^{|\Edg(v)|-1}
\int_{\Mbar_{0,\val(v)}}\frac{1}{\prod_{e\in \Edg(v)}\Big(\frac{\alpha_{v,e}}{d(e)}-\psi_{(v,e)}\Big)},
\eea
where $\alpha_{(v,e)}=\alpha_{i,j}$ if $e=\{v,v^{\prime}\}$ with $\mathfrak{m}(v^{\prime})=j$, and $\psi_{(v,e)}$ is the $\psi$-class associated to the marked point on $\Mbar_{0,\val(v)}$ corresponding to the edge $e$.
 The explicit form of $\Cont_{\Gamma;e}\big(\langle\mu_{1},\cdots,\mu_{|J|}\rangle_{1,J,d}^{X}\big)$ can be computed by the \emph{holomorphic Lefschetz formula} (\cite{AS}), see e.g., \cite{YZ}; we will not spell out the general formula for this since we don't need it. For $X$ of the form (\ref{40}), the explicit form for $\Cont_{\Gamma;e}\big(\langle\mu_{1},\cdots,\mu_{|J|}\rangle_{1,J,d}^{X}\big)$ is (\ref{12}).

Note that we always adopt the convention that, the formal integrals over $\Mbar_{0,1}$ and $\Mbar_{0,2}$ are understood as extending the range of $n$ in the following identity to $r\geq 1$:
\bea\label{21}
\int_{\Mbar_{0,r}}\frac{1}{\prod_{i=1}^{r}(w_{i}-\psi_{i})}=\frac{1}{\prod_{i=1}^{r}w_{i}}\big(\sum_{i=1}^{r}\frac{1}{w_{i}}\big)^{r-3}.
\eea

  Next we consider the contributions from decorated rooted trees.  For a decorated rooted tree $\Gamma\in\DRT_{J}^{d}$ the root $v_{0}$  represents a genus one subcurve which is contracted by the stable map. The localization contribution of $\Gamma$ to (\ref{67}) can also be written as a product
\bea\label{31}
 &&\Cont_{\Gamma}\big(\langle\mu_{1},\cdots,\mu_{|J|}\rangle_{1,J,d}^{X}\big)=\frac{1}{|\Aut(\Gamma)|}\nn\\
 &&\cdot\prod_{v\in \Ver} \Cont_{\Gamma;v}\big(\langle\mu_{1},\cdots,\mu_{|J|}\rangle_{1,J,d}^{X}\big)\prod_{e\in \Edg} \Cont_{\Gamma;e}\big(\langle\mu_{1},\cdots,\mu_{|J|}\rangle_{1,J,d}^{X}\big),
 \eea
  where for an edge $e$ and a vertex $v\in \Ver\backslash\{v_{0}\}$, the contribution is the same as those in (\ref{11})  respectively, while the contribution of $v_{0}$ with $\mathfrak{m}(v_{0})=i$ is
\bea\label{14}
&&\Cont_{\Gamma;v_{0}}\big(\langle\mu_{1},\cdots,\mu_{|J|}\rangle_{1,J,d}^{X}\big)=\prod_{j\in \mathfrak{e}^{-1}(v)}\mu_{j}\big|_{P_{i}}\nn\\
&&\cdot\Big(\prod_{k=1}^{l}\varepsilon_{i,k}\prod_{P_{j}\in \Nb(P_{i})}\alpha_{i,j}\Big)^{|\Edg(v_{0})|-1}
\int_{\Mbar_{1,\val(v_{0})}}\frac{\prod_{P_{j}\in \Nb(P_{i})}\Lambda_{1}^{\vee}(\alpha_{i,j})\prod_{k=1}^{l}\Lambda_{1}^{\vee}(\varepsilon_{i,k})}{\prod_{e\in \Edg(v_{0})}\Big(\frac{\alpha_{v_{0},e}}{d(e)}-\psi_{(v_{0},e)}\Big)}.\nn\\
\eea
Summing the contributions, we have
\bea\label{34}
\langle\mu_{1},\cdots,\mu_{|J|}\rangle_{1,J,d}^{X}=\sum_{\Gamma\in \DOL_{J}^{d}}\Cont_{\Gamma}\big(\langle\mu_{1},\cdots,\mu_{|J|}\rangle_{1,J,d}^{X}\big)
+\sum_{\Gamma\in \DRT_{J}^{d}}\Cont_{\Gamma}\big(\langle\mu_{1},\cdots,\mu_{|J|}\rangle_{1,J,d}^{X}\big).\nn\\
\eea

\subsection{Fixed loci on $\Mbar_{(m,J)}(Y,d)$ and  localization contributions}
In this subsection, we recall the definition of \cite{ZingerSvR} for $\Mbar_{(m,J)}(Y,d)$ and some natural cohomology classes on it. Then we describe the fixed loci on $\Mbar_{(m,J)}(Y,d)$ and their localization contributions.\\

Let $Y$ be a smooth projective variety. For $\mathbf{d}=(d_{1},\cdots,d_{m})\in (H_{2}^{+}(Y)-\{0\})^{m}$ and finite sets $J_{1},\cdots,J_{m}$, define $\Mbar_{(m;J_{1},\cdots,J_{m})}\big(Y,\mathbf{d}\big)$ by the cartesian diagram
\bea\label{57}
\xymatrix{
  \Mbar_{(m;J_{1},\cdots,J_{m})}(Y,\mathbf{d})\ar[d]^{\ev_{0}} \ar[r]
                & \prod_{s=1}^{m}\Mbar_{0,\{0_{s}\}\sqcup J_{s}}(Y,d_{s}) \ar[d]_{\ev_{0_{1}}\times\cdots\times\ev_{0_{m}}}  \\
  Y \ar[r]^{\triangle_{Y}}
                &  Y^{m}           }.
\eea
Define
\bea\label{72}
[\Mbar_{(m;J_{1},\cdots,J_{m})}(Y,\mathbf{d})]^{\vir}=\triangle_{Y}^{!}\Bigg(\prod_{s=1}^{m}[\Mbar_{0,\{0_{s}\}\sqcup J_{s}}(Y,d_{s})]^{\vir}\Bigg),
\eea
where $\triangle_{Y}^{!}$ is the Gysin map.
\begin{remark}\label{87}
For flag varieties $Y=G/P$, since $\ev_{0_{s}}:\Mbar_{0,\{0_{s}\}\sqcup J_{s}}(Y,d_{s})\rightarrow Y$ is a smooth morphism and
$\Mbar_{0,1}(Y,d_{s})$ is smooth for every $1\leq s\leq m$, $\Mbar_{(m,J)}(Y,\mathbf{d})$ is smooth as well. Thus
$[\Mbar_{(m;J_{1},\cdots,J_{m})}(Y,\mathbf{d})]^{\vir}=[\Mbar_{(m;J_{1},\cdots,J_{m})}(Y,\mathbf{d})]$.
\end{remark}
\vspace{0.2cm}

For $d\in H_{2}^{+}(Y)-\{0\}$, let
\bea\label{58}
\Mbar_{(m,J)}(Y,d)=\coprod_{J_{1}\sqcup\cdots\sqcup J_{m}=J}\coprod_{\stackrel{d_{1}+\cdots+d_{m}=d}{d_{1},\cdots,d_{m}>0}}\Mbar_{(m;J_{1},\cdots,J_{m})}(Y,\mathbf{d}),
\eea
and
\bea\label{73}
[\Mbar_{(m,J)}(Y,d)]^{\vir}=\coprod_{J_{1}\sqcup\cdots\sqcup J_{m}=J}\coprod_{\stackrel{d_{1}+\cdots+d_{m}=d}{d_{1},\cdots,d_{m}>0}}[\Mbar_{(m;J_{1},\cdots,J_{m})}(Y,\mathbf{d})]^{\vir}.
\eea
There are natural projection maps
\bea\label{66}
\pi_{s}:\Mbar_{(m;J_{1},\cdots,J_{m})}(Y,\mathbf{d})\rightarrow \Mbar_{0,\{0_{s}\}\sqcup J_{s}}(Y,d_{s})
\eea
for $1\leq s \leq m$. Let $\psi_{0_{s}}$ be the Euler class of the cotangent line bundle on $\Mbar_{0,0_{s}\sqcup J_{s}}(Y,d_{s})$ associated to the marked point $0_{s}$. We define $\eta_{p}\in H^{2p}\big(\Mbar_{(m;J_{1},\cdots,J_{m})}(Y,\mathbf{d})\big)$ by the generating function
\bea\label{62}
\sum_{p=0}^{\infty}z^{p}\eta_{p}=\prod_{s=1}^{m}\frac{1}{1-z\pi_{s}^{*}\psi_{0_{s}}}.
\eea
Varying $d_{1}+\cdots+d_{m}=\mathbf{d}$ and $J_{1}\sqcup\cdots\sqcup J_{m}=J$, we obtain the class $\eta_{p}$ over $\Mbar_{(m,J)}(\mathbb{P}^{n-1},d)$.\\

For every $j\in J$ there is an evaluation map $\ev_{j}:\Mbar_{(m,J)}(Y,d)\rightarrow Y$ in an obvious way. There is also an evaluation map $\ev_{0}:\Mbar_{(m,J)}(Y,d)\rightarrow Y$ associated to the common \emph{0-th} marked point. For $J^{\prime}\subset J$, let
\bea\label{68}
\langle \eta_{p}\mu_{0};\mu_{1},\cdots,\mu_{|J|}\rangle_{(m,J-J^{\prime},d)}^{Y}:=
\frac{1}{m!}\eta_{p}\ev_{0}^{*}\big(\mu_{0}\prod_{j\in J^{\prime}}\mu_{j}\big)\prod_{j\in J-J^{\prime}}\ev_{j}^{*}(\mu_{j})\cap [\Mbar_{(m,J-J^{\prime})}(Y,d)]^{\vir}.
\eea

For $X=\mathrm{Tot}(E\rightarrow Y)$ where $E$ is a concave vector bundle over $Y$, to define invariants for $X$ similar to (\ref{68}), we need only to replace $Y$ by $X$ in the above definitions, and note that $\Mbar_{g,J}(X,d)=\Mbar_{g,J}(Y,d)$ and $[\Mbar_{g,J}(X,d)]^{\vir}=\mathcal{U}_{g}\cap [\Mbar_{g,J}(Y,d)]^{\vir}$. More concretely, we define
\bea\label{74}
\langle \eta_{p}\mu_{0};\mu_{1},\cdots,\mu_{|J|}\rangle_{(m,J-J^{\prime},d)}^{X}
&:=&\frac{1}{m!}\Bigg(\big(\ev_{0}^{*}e(E)\big)^{m-1}\prod_{i=1}^{m}\pi_{i}^{*}(\mathcal{U}_{0})\cdot\eta_{p}\ev_{0}^{*}\big(\mu_{0}\prod_{j\in J^{\prime}}\mu_{j}\big)\nn\\
&&\cdot\prod_{j\in J-J^{\prime}}\ev_{j}^{*}(\mu_{j})\Bigg)\cap [\Mbar_{(m,J-J^{\prime})}(Y,d)]^{\vir}.
\eea

Now let $Y$ be an algebraic GKM manifold and $X=\mathrm{Tot}(E\rightarrow Y)$ is the total space of a concave equivariant vector bundle $E$. To describe the fixed loci on $\Mbar_{(m,J)}(Y,d)$, we need to introduce some notions. For a  set $S$, let $\mathcal{P}(S)$ be its power set. The set of  \emph{$m$-colored partitions} of a finite set $S$ is defined to be
\ben
\mathcal{A}_{m}(S)=\{I\in \mathcal{P}(S)^{m}:I=(I_{(1)},\cdots,I_{(m)}), I_{(1)}\sqcup\cdots I_{(m)}=S, |I_{(i)}|>0, \hspace{0.2cm} \textrm{for}\hspace{0.2cm}1\leq i\leq m \},
\een
 the set of  \emph{nonnegative $m$-colored partitions} of a finite set $S$ is defined to be
\ben
\mathcal{A}_{m}^{0}(S)=\{I\in \mathcal{P}(S)^{m}:I=(I_{(1)},\cdots,I_{(m)}), I_{(1)}\sqcup\cdots I_{(m)}=S, |I_{(i)}|\geq 0, \hspace{0.2cm} \textrm{for}\hspace{0.2cm}1\leq i\leq m \},
\een
and the set of  \emph{$m$-colored partitions} of a pair of finite sets $(S,J)$ is defined to be
\ben
&&\mathcal{A}_{m}(S,J)=\mathcal{A}_{m}(S)\times\mathcal{A}_{m}^{0}(J).
\een
A \emph{$m$-colored decorated rooted tree} is a pair $\Gamma^{\mathfrak{c}}=(\Gamma,I,K)$, where $\Gamma$ is a decorated rooted tree with a root $v_{0}$, and $(I,K)\in \mathcal{A}_{m}(\Edg(v_{0}),\mathfrak{e}^{-1}(v_{0}))$. The notions of the degree of $\Gamma^{\mathfrak{c}}$ and the valence of a vertex  is inherited from those of $\Gamma$. The set of $m$-colored decorated rooted trees of degree $d$ and with the set of marked points $J$ is denoted by $m\CDRT_{J}^{d}$.\\

There is a canonical projection map $\pi_{\CDRT}:m\CDRT_{J}^{d}\rightarrow \DRT_{J}^{d}$ with $\pi_{\CDRT}(\Gamma, I,K)=\Gamma$. The automorphism group $\Aut(\Gamma)$ acts on $\pi_{\CDRT}^{-1}(\Gamma)$ in a natural way, the stable subgroup of $\Gamma^{\mathfrak{c}}$ is called the  automorphism  group of $\Gamma^{\mathfrak{c}}$. The fixed loci on $\Mbar_{(m, J)}(Y,d)$ are indexed by $\Gamma^{\mathfrak{c}}\in m\CDRT_{J}^{d} $ in an obvious way. Two $m$-colored decorated rooted trees $\Gamma_{1}^{\mathfrak{c}}$ and $\Gamma_{2}^{\mathfrak{c}}$ index the same fixed locus if and only if $\pi_{\CDRT}(\Gamma_{1}^{\mathfrak{c}})=\pi_{\CDRT}(\Gamma_{2}^{\mathfrak{c}})=\Gamma$ for some $\Gamma$ and $\Gamma_{1}^{\mathfrak{c}}=g.\Gamma_{2}^{\mathfrak{c}}$ for some $g\in \Aut(\Gamma)$.\\

Let $\mu_{0}, \mu_{1},\cdots,\mu_{|J|}\in H_{\mathbb{T}}^{*}(Y)$. We denote equivariant version of $\psi$-classes over $\Mbar_{0,J}(Y,d)$ and $\tilde{\eta}$-classes over $\Mbar_{(m,J)}(Y,d)$ by the same symbols in the equivariant integration; by the context, no confusion should arise. 
Now we use virtual localization to compute (\ref{74}) as a summing over graphs.
For $\Gamma^{\mathfrak{c}}=(\Gamma,I,K)\in m\CDRT_{J^{\prime}}^{d}$, the localization contribution of $\Gamma^{\mathfrak{c}}$  to $\langle\eta_{p}\mathbf{c}_{q}(TX);\mu_{1},\cdots,\mu_{|J|}\rangle_{(m,J^{\prime},d)}^{X}$ can be written as
 \begin{multline}\label{15}
 \Cont_{\Gamma^{\mathfrak{c}}}(\langle\eta_{p}\mathbf{c}_{q}(TX);\mu_{1},\cdots,\mu_{|J|}\rangle_{(m,J^{\prime},d)}^{X})
 =\frac{1}{m!}\frac{1}{|\Aut(\Gamma^{\mathfrak{c}})|}\\
 \cdot\prod_{v\in \Ver} \Cont_{\Gamma^{\mathfrak{c}};v}\big(\langle\eta_{p}\mathbf{c}_{q}(TX);\mu_{1},\cdots,\mu_{|J|}\rangle_{(m,J^{\prime},d)}^{X}\big) \\
 \cdot\prod_{e\in \Edg} \Cont_{\Gamma^{\mathfrak{c}};e}\big(\langle\eta_{p}\mathbf{c}_{q}(TX);\mu_{1},\cdots,\mu_{|J|}\rangle_{(m,J^{\prime},d)}^{X}\big).
 \end{multline}
For an edge $e$ and a vertex $v\in \Ver\backslash\{v_{0}\}$, the contribution is still the same as (\ref{11})  respectively. Suppose $\mathfrak{m}(v_{0})=i$, then the contribution of $v_{0}$ is
\bea\label{17}
&&\Cont_{\Gamma^{\mathfrak{c}};v_{0}}\big(\langle\eta_{p}\mathbf{c}_{q}(TX);\mu_{1},\cdots,\mu_{|J|}\rangle_{(m,J^{\prime},d)}^{X}\big)
=\prod_{j\in \mathfrak{e}^{-1}(v_{0})}\mu_{j}\big|_{P_{i}}\nn\\
&&\Big(\prod_{k=1}^{l}\varepsilon_{i,k}\prod_{P_{j}\in\Nb(P_{i})}\alpha_{i,j}\Big)^{|\Edg(v_{0})|-1}
\cdot[x^{q}]\Big(\prod_{k=1}^{l}(1+\varepsilon_{i,k}x)\prod_{P_{j}\in\Nb(P_{i})}(1+\alpha_{i,j}x)\Big)\nn\\
&&\cdot\int_{\Mbar_{0,I_{(1)}\sqcup K_{(1)}\sqcup \{0_{1}\}}\times\cdots\times\Mbar_{0,I_{(m)}\sqcup K_{(m)}\sqcup \{0_{m}\}}}
[x^{p}]\Bigg(\frac{1}{\prod_{e\in\Edg(v_{0})}(\frac{\alpha_{v_{0},e}}{d(e)}-\psi_{(v_{0},e)}) \prod_{s=1}^{m}(1-x\psi_{0_{s}})}\Bigg).\nn\\
\eea
We briefly explain (\ref{17}).
\begin{itemize}
\item{}The marked point $0_{s}$ for $1\leq s\leq m$ comes from the common node represented by the root $v_{0}$, which becomes a marked point when we split the domain curve with respect to the $m$ colors.
\item{}The operator $[x^{q}]$ extracts the equivariant $q$-th Chern class of $X$ restricted to the fixed point $P_{i}$. The operator $[x^{p}]$ extracts all the $p$-th monomials of $\psi_{0_{s}}$ for $1\leq s\leq m$, and the sum is $\eta_{p}$ restricted to this fixed locus, by the definition of $\eta_{p}$.
\item{}By the argument parallel to the proof of the cutting edge axiom  in \cite{Behrend}, it is easy to show that there is a natural perfect obstruction theory on $\Mbar_{(m,J^{\prime})}(Y,d)$, and the corresponding virtual fundamental cycle is the same as (\ref{73}). The localization contribution is easily   read out from this perfect obstruction theory. In particular, when $Y$ is a flag variety, by remark \ref{87} it is straightforward to obtain (\ref{17}). Note that the deformation of the domain curves should be \emph{color-preserved}, so the \emph{node-smoothing} contribution in the usual virtual localization
\ben
\int_{\Mbar_{0,\val(v)}}\frac{\cdots}{\Big(\frac{\alpha_{v,e}}{d(e)}-\psi_{(v,e)}\Big)\cdots}
\een
should be replaced by a color-preserved version, which is of the form in (\ref{17}).
\end{itemize}

For later use, we need to write the invariant $\langle\eta_{p}\mathbf{c}_{q}(TX);\mu_{1},\cdots,\mu_{|J|}\rangle_{(m,J^{\prime},d)}^{X}$ as summing over the $\DRT_{J}^{d}$. First note that for evary $\Gamma_{1}\in\DRT_{J_{1}}^{d}$ and $J_{2}\supset J_{1}$, we can attach additional $|J_{2}|-|J_{1}|$ marked points to the root $v_{0}$ of  $\Gamma_{1}$, thus obtain $\Gamma_{2}\in\DRT_{J_{2}}^{d}$. In this way, we get an injective map
\ben
\rho_{J_{1},J_{2}}:\DRT_{J_{1}}^{d}\rightarrow \DRT_{J_{2}}^{d},
\een
such that $\rho_{J_{1},J_{2}}(\Gamma_{1})=\Gamma_{2}$.\\

Then we have
\bea\label{18}
&&\langle\eta_{p}\mathbf{c}_{q}(TX);\mu_{1},\cdots,\mu_{|J|}\rangle_{(m,J^{\prime},d)}^{X}\nn\\
&&=\sum_{\Gamma\in \DRT_{J}^{d}}\Cont_{\Gamma}\big(\langle\eta_{p}\mathbf{c}_{q}(TX);\mu_{1},\cdots,\mu_{|J|}\rangle_{(m,J^{\prime},d)}^{X}\big),
\eea
where
\bea\label{19}
&&\Cont_{\Gamma}\big(\langle\eta_{p}\mathbf{c}_{q}(TX);\mu_{1},\cdots,\mu_{|J|}\rangle_{(m,J^{\prime},d)}^{X}\big)\nn\\
&=&\frac{1}{m!}\frac{1}{|\Aut(\Gamma)|}
\sum_{\Gamma^{\mathfrak{c}}\in\pi_{\CDRT}^{-1}\circ\rho_{J^{\prime},J}^{-1}(\Gamma)}\prod_{v\in \Ver} \Cont_{\Gamma^{\mathfrak{c}};v}\big(\langle\eta_{p}\mathbf{c}_{q}(TX);\mu_{1},\cdots,\mu_{|J|}\rangle_{(m,J^{\prime},d)}^{X}\big)\nn\\
&&\cdot\prod_{e\in \Edg} \Cont_{\Gamma^{\mathfrak{c}};e}\big(\langle\eta_{p}\mathbf{c}_{q}(TX);\mu_{1},\cdots,\mu_{|J|}\rangle_{(m,J^{\prime},d)}^{X}\big).
 \eea
So if $\Gamma$ is not in the image of $\rho_{J^{\prime},J}$, then $\Cont_{\Gamma}\big(\langle\eta_{p}\mathbf{c}_{q}(TX);\mu_{1},\cdots,\mu_{|J|}\rangle_{(m,J^{\prime},d)}^{X}\big)=0$.\\

Note that the automorphism factor in (\ref{19}) is $1/|\Aut(\Gamma)|$ (we have $\Aut(\Gamma)\cong\Aut(\rho_{J^{\prime},J}^{-1}\Gamma)$ when $\Gamma$ is in the image of $\rho_{J^{\prime},J}$), not $1/|\Aut(\Gamma^{\mathfrak{c}})|$, because $g_{1}.\Gamma_{\mathfrak{c}}$ and $g_{2}.\Gamma_{\mathfrak{c}}$ index the same fixed locus when $g_{1}$ and $g_{2}$ are in the same coset of $\Aut(\Gamma_{\mathfrak{c}})$ in $\Aut(\Gamma)$.\\

\subsection{Formal fixed loci, formal localization contributions and reduced genus one Gromov-Witten invariants}
For $\mathbb{P}^{n-1}$, the fixed loci on $\tM_{1,J}(\mathbb{P}^{n-1},d)$ are described in \cite{VZ}. There are two types of fixed loci. The first type is indexed by the decorated one-loop graphs. The fixed locus indexed by a one-loop graph is exactly the same as that in section 2.1, and for a hypersurface $X$ in $\mathbb{P}^{n-1}$, its localization contribution to $\langle\mu_{1},\cdots,\mu_{|J|}\rangle_{1,J,d}^{0;X}$ is the same as the localization contributions to the usual genus one Gromov-Witten invariants. The other type of fixed loci and their localization contributions are described in \cite{VZ} (see also \cite{Zinger1} in the cases $|J|=0$ or $1$). \\

For an algebraic GKM manifold $Y$, a finite set $J$ and $d\in H_{2}^{+}(Y)$,  as an analog to $\mathbb{P}^{n-1}$, we assign two types of fixed loci . The \emph{formal fixed loci of the first type} for $(Y,J,d)$ are indexed by $\DOL_{J}^{d}(Y)$, and their \emph{formal localization contributions} to the reduced genus one Gromov-Witten invariants $\langle \mu_{1},\cdots,\mu_{|J|}\rangle_{1,J,d}^{0;X}$ are the same as (\ref{11}).\\

To define the \emph{formal fixed loci of the second type}, we need to recall the definition of refined decorated rooted trees (\cite{VZ}, \cite{Zinger1}) \footnote{We highly recommend the reader who is not familiar with the definition of refined decorated rooted trees and some other related notions to refer to \cite[ section 1.4]{VZ} for the case $Y=\mathbb{P}^{n-1}$.}. A \emph{refined decorated rooted tree} is a tuple
\bea
\widetilde{\Gamma}=(\Ver,\Edg,v_{0};\Ver_{+},\Ver_{0};\mathfrak{m},\mathfrak{d},\mathfrak{e}),
\eea
where $(\Ver,\Edg,v_{0})$ is a rooted tree and\footnote{In \cite{Zinger1} the domain of the map $\mathfrak{m}$ is $\Ver-\Ver_{0}$. In this article we think it is more convenient to extend $\mathfrak{m}$ to $\Ver$.}

\begin{description}
  \item{(i)} $\Ver_{+}, \Ver_{0}\subset \Ver-\{v_{0}\}$, $\Ver_{+}\neq \emptyset$, $\Ver_{+}\cap\Ver_{0}=\emptyset$, $\{v_{0},v\}\in\Edg$ for $\forall v\in \Ver_{+}\cup\Ver_{0}$.
  \item{(ii)} $\Edg_{+}:=\{\{v_{0},v\}:v\in\Ver_{+}\}$, $\Edg_{0}:=\{\{v_{0},v\}:v\in\Ver_{0}\}$.
  \item{(iii)} There are three maps
  \ben
  \mathfrak{m}:\Ver\rightarrow [N],&\mathfrak{d}:\Edg-\Edg_{0}\rightarrow \mathbb{B}(Y),&\mathfrak{e}:J\rightarrow\Ver.
  \een
  The map
  \ben
\mathfrak{d}:\Edg-\Edg_{0}\rightarrow \mathbb{B}(Y)
\een
for an edge $e=\{v_{1},v_{2}\}\in\Edg-\Edg_{0}$ takes values in $\mathbb{Z}^{>0}\cdot\overline{P_{\mathfrak{m}(v_{1})}P_{\mathfrak{m}(v_{2})}}$. Let $d(e)\in \mathbb{Z}^{>0}$ such that $\mathfrak{d}(e)=d(e)\cdot\overline{P_{\mathfrak{m}(v_{1})}P_{\mathfrak{m}(v_{2})}}$ when $e=\{v_{1},v_{2}\}$.
Let
\ben
\mathfrak{d}(\widetilde{\Gamma})=\sum_{e\in\Edg-\Edg_{0}}\mathfrak{d}(e)\in \mathbb{B}(Y).
\een
 For $\widetilde{\Gamma}$ to represent a formal fixed locus for $(Y,J,d)$, we need $\mathrm{d}\circ\mathfrak{d}(\widetilde{\Gamma})=d$; we call $\mathrm{d}\circ\mathfrak{d}(\widetilde{\Gamma})$ the degree of $\widetilde{\Gamma}$.
  \item{(iv)} If $v_{1}\in\Ver_{+}$, $v_{2}\in\Ver-\Ver_{0}$ and $\{v_{0},v_{2}\}\in\Edg$, then
  \bea\label{75}
  \mathfrak{d}(\{v_{0},v_{1}\})=\mathfrak{d}(\{v_{0},v_{2}\})
  \Leftrightarrow v_{2}\in \Ver_{+}.
  \eea
  Note that $\mathfrak{d}(\{v_{0},v_{1}\})=\mathfrak{d}(\{v_{0},v_{2}\})$ if and only if $d(\{v_{0},v_{1}\})=d(\{v_{0},v_{2}\})$ and $\mathfrak{m}(v_{1})=\mathfrak{m}(v_{2})$.
  \item{(v)} If $\{v_{1},v_{2}\}\in\Edg$, then $\mathfrak{m}(v_{1})=\mathfrak{m}(v_{2})$ if and only if $v_{1}=v_{0}$ and
    $v_{2}\in\Ver_{0}$, or $v_{2}=v_{0}$ and $v_{1}\in\Ver_{0}$.
  \item{(vi)} If $v_{1}\in\Ver_{0}$ then $|\Edg(v_{1})|\geq 2$ and $|\val(v_{1})|=|\Edg(v_{1})|+|\mathfrak{e}^{-1}(v_{1})|\geq 3$.
\end{description}

 \begin{remark}
In the definition of the refined decorated rooted trees in this article we do not include the condition  $\sum_{e\in\Edg_{+}}d(e)\geq 2$. However, as the proof of lemma \ref{79} shows, the localization contribution of a refined decorated rooted tree which does not satisfies this condition is zero. So one can make the choice to include this condition or not; without this condition the summing over graphs becomes slightly easier.
\end{remark}

 We denote the set of refined decorated rooted trees for $(Y,J,d)$ by $\RDRT_{J}^{d}(Y)$. As before, when we are discussing a fixed $Y$ which is clear from the context, we drop the notation $Y$. \\

 Let $\widetilde{\Gamma}\in \RDRT_{J}^{d}(Y)$. For every $e\in\Edg(v_{0})$, there is an associated \emph{strand} $\mathcal{Z}_{\widetilde{\Gamma}_{e}}$, which is a decorated tree; we refer the reader to  \cite[section 1.4]{Zinger1} for the definition. We need also the stacks $\tM_{1,(I,J)}$ for finite sets $I$ and $J$, which are blow-ups of $\Mbar_{1, I\sqcup J}$; we refer the reader to \cite{ZingerSvR} for the definition of these spaces, the line bundle $\mathbb{L}$, the cohomology classes $\tilde{\psi}=c_{1}(\mathbb{L})$ and $\psi_{j}$, $j\in J$, and their integrations over $\tM_{1,(I,J)}$. The integration that we need is
 \bea\label{89}
 \int_{\tM_{1,(I,J)}}\tilde{\psi}^{|I|+|J|}=\frac{|I|^{|J|}\cdot(|I|-1)!}{24}.
 \eea
 The \emph{formal fixed locus} associated to $\widetilde{\Gamma}$ is defined to be
 \bea\label{88}
 \mathcal{Z}_{\widetilde{\Gamma}}=\tM_{1,(\Edg(v_{0}),\mathfrak{e}^{-1}(v_{0}))}
 \times \mathbb{P}^{|\Ver_{+}|-1}\times\prod_{e\in\Edg(v_{0})}\mathcal{Z}_{\widetilde{\Gamma}_{e}}.
 \eea
%However, two distinct elements of $\RDRT_{J}^{d}(Y)$ may correspond to the same formal fixed locus.
 For a refined decorated rooted tree $\widetilde{\Gamma}=(\Ver,\Edg,v_{0};\Ver_{+},\Ver_{0};\mathfrak{m},\mathfrak{d},\mathfrak{e})$, we can naturally associate a decorated rooted tree $\Gamma=(\Ver_{\Gamma},\Edg_{\Gamma},v_{0};\mathfrak{m}_{\Gamma},\mathfrak{d}_{\Gamma},\mathfrak{e}_{\Gamma})$ such that
\begin{itemize}
  \item{} $\Ver_{\Gamma}=\Ver-\Ver_{0}$.
  \item{} $\Edg_{\Gamma}=(\Edg-\Edg_{0})\sqcup \{\{v_{0},v\}:v\in \Edg(v_{1})-\{v_{0}\}$ for some $v_{1}\in\Ver_{0}\}$.
  \item{}  $\mathfrak{m}_{\Gamma}: \Ver_{\Gamma}\rightarrow [N]$
  is the the restriction the map $\mathfrak{m}$ to $\Ver-\Ver_{0}$.
  \item $\mathfrak{d}_{\Gamma}:\Edg_{\Gamma}\rightarrow\mathbb{B}(Y)$ is a map with $\mathfrak{d}_{\Gamma}(e)=\mathfrak{d}(e)$ for $e\in\Edg-\Edg_{0}$, and $\mathfrak{d}(\{v_{0},v\})=\mathfrak{d}(\{v_{1},v\})$ where $v_{1}\in\Ver_{0}\}$ and $v\in \Edg(v_{1})-\{v_{0}\}$.
  \item{} $\mathfrak{e}_{\Gamma}:J\rightarrow\Ver_{\Gamma}$ is a map with $\mathfrak{e}_{\Gamma}(j)=\mathfrak{e}(j)$ for $j\in \mathfrak{e}^{-1}(\Ver-\Ver_{0})$ and $\mathfrak{e}_{\Gamma}(j)=v_{0}$ for $j\in \mathfrak{e}^{-1}(\Ver_{0})$.
\end{itemize}
By the assumptions in the definition of refined decorated rooted trees, especially (v), we see that $\Gamma(\Ver_{\Gamma},\Edg_{\Gamma},v_{0};\mathfrak{m}_{\Gamma},\mathfrak{d}_{\Gamma},\mathfrak{e}_{\Gamma})$ is a decorated rooted tree, which we call the \emph{underlying decorated rooted tree} of $\widetilde{\Gamma}$.
There is a canonical projection map $\pi_{\RDRT}:\RDRT_{J}^{d}\rightarrow \DRT_{J}^{d}$ which send a refined decorated rooted tree to its underlying decorated rooted tree.\\

\textbf{An important modification}: for the clearness in counting graphs in the future, we now make a slight modification of the definition of $\RDRT_{J}^{d}$. For a refined decorated rooted tree $\widetilde{\Gamma}$ and the corresponding decorated rooted tree $\Gamma=\pi_{\RDRT}(\widetilde{\Gamma})$, the edges in $\Edg^{\widetilde{\Gamma}}(v_{0})$ induces a partition of $\Edg^{\Gamma}(v_{0})$. If $e\in \Edg^{\widetilde{\Gamma}}(v_{0})-\Edg_{0}$, then let $I_{e}=\{e\}$; If $e=\{v_{0},v\}\in \Edg_{0}$, then let $I_{e}=\Edg(v)-\{e\}$. Thus $\{I_{e}\}_{e\in \Edg^{\widetilde{\Gamma}}(v_{0})}$ canonically corresponds to an  \emph{unordered partition} of $\Edg^{\Gamma}(v_{0})$. Similarly, for every $e=\{v_{0},v\}\in \Edg^{\widetilde{\Gamma}}(v_{0})$, let $J_{e}=\mathfrak{e}^{-1}(v)$, together with $\mathfrak{e}^{-1}(v_{0})$, we obtain an \emph{unordered nonnegative partition} of $J$. Now we \emph{impose} the condition that the two partitions are \emph{ordered}. More precisely:
\begin{definition}
A refined decorated rooted tree is a tuple
$\widetilde{\Gamma}=(\Ver,\Edg,v_{0};\Ver_{+},\Ver_{0};\mathfrak{m},\mathfrak{d},\mathfrak{e}, I,K)$, such that $\widetilde{\Gamma}=(\Ver,\Edg,v_{0};\Ver_{+},\Ver_{0};\mathfrak{m},\mathfrak{d},\mathfrak{e})$ satisfies the conditions (i)-(vi) above, and $I\in \mathcal{A}_{m}\big(\Edg^{\Gamma}(v_{0})\big)$ and $J\in\mathcal{A}_{m+1}^{0}(J)$ are \emph{colored partitions} which are compatible with the unordered partitions associated to $\widetilde{\Gamma}=(\Ver,\Edg,v_{0};\Ver_{+},\Ver_{0};\mathfrak{m},\mathfrak{d},\mathfrak{e})$, where $m=|\Edg^{\widetilde{\Gamma}}(v_{0})|$.
\end{definition}
From now on, we adopt this definition of refined decorated rooted trees. We use $\RDRT_{J}^{d}$ to represent the set of refined decorated rooted trees in this sense. Then the automorphism group $\Aut(\Gamma)$ acts on $\pi_{\RDRT}^{-1}(\Gamma)$ in a natural way, and the stable subgroup of $\widetilde{\Gamma}$ is called the  automorphism  group of $\widetilde{\Gamma}$. Two refined decorated rooted trees $\widetilde{\Gamma}_{1}$ and $\widetilde{\Gamma}_{2}$ index the same fixed locus if and only if they have the same underlying decorated rooted tree $\Gamma$ and  $\widetilde{\Gamma}_{1}=g.\widetilde{\Gamma}_{2}$ for some $g\in \Aut(\Gamma)$, and when this happens we say $\widetilde{\Gamma}_{1}\sim\widetilde{\Gamma}_{2}$.\\

Now we define the \emph{formal localization contribution} of $\widetilde{\Gamma}$. Let $\pi_{e}:\mathcal{Z}_{\widetilde{\Gamma}}\rightarrow \mathcal{Z}_{\widetilde{\Gamma}_{e}}$ be the natural projection map for $e\in\Edg(v_{0})$. On $\mathcal{Z}_{\widetilde{\Gamma}}$ there is a universal tangent line bundle associated to the attachment at $v_{0}$, and we denote it by $L_{e}$.
Let $L_{\widetilde{\Gamma}}=\pi_{e}^{*}L_{e}$ for $e\in \Edg_{+}$; by (\ref{75}) $\pi_{e}^{*}L_{e}$ as an equivariant line bundle is independent of the choice of $e\in \Edg_{+}$. Also, let
\bea
 F_{\widetilde{\Gamma};B}^{c}=\bigoplus_{e\in\Edg(v_{0})-\Edg_{+}}\pi_{e}^{*}L_{e}.
\eea
Let $\gamma$ be the tautological line bundle on $\mathbb{P}^{|\Ver_{+}|-1}$, and $c_{1}(\gamma)=-H$. We use the same symbol for the pullbacks of $\gamma$ and $\mathbb{L}$ to $\widetilde{\mathcal{Z}}_{\widetilde{\Gamma}}$ via the natural projection maps.\\

For each $e\in\Edg(v_{0})$,
$\mathcal{Z}_{\widetilde{\Gamma}_{e}}$ is a fixed locus in an appropriate moduli space of  genus zero stable maps into $Y$, thus the virtual normal bundle  $\mathcal{N}\widetilde{\mathcal{Z}}_{\widetilde{\Gamma}_{e}}$ and the vector bundle $\mathcal{U}_{0}^{\prime}$ corresponding to $X=\mathrm{Tot}(E\rightarrow Y)$ are well-defined.\\

After these preparation, we define $\mathbf{e}(\mathcal{N}\widetilde{\mathcal{Z}}_{\widetilde{\Gamma}})$ via
\bea\label{23}
\frac{\mathbf{e}(\mathcal{N}\widetilde{\mathcal{Z}}_{\widetilde{\Gamma}})}{\mathbf{e}(T_{\mathfrak{m}(v_{0})}Y)}
:=\prod_{e\in \Edg(v_{0})}\frac{\mathbf{e}(\mathcal{N}\mathcal{Z}_{\widetilde{\Gamma}_{e}})}{\mathbf{e}(T_{\mathfrak{m}(v_{0})}Y)}
 \cdot\frac{\mathbf{e}(L_{\widetilde{\Gamma}}^{*}\otimes F_{\widetilde{\Gamma};B}^{c}
\otimes\gamma^{*})
 \mathbf{e}(\mathbb{L}\otimes L_{\widetilde{\Gamma}}\otimes\gamma)}
{\mathbf{e}\big(L_{\widetilde{\Gamma}}^{*}\otimes T_{\mathfrak{m}(v_{0})}Y
\otimes \gamma^{*}\big)}
\eea
and define $\mathbf{e}(\mathcal{U}_{1}^{\prime})$ and $ \mathbf{e}(\mathcal{U}_{1})$ by
\bea\label{24}
\mathbf{e}(\mathcal{U}_{1}^{\prime})= \mathbf{e}(\mathcal{U}_{1})\cdot \mathbf{e} (E)\big|_{P_{\mathfrak{m}(v_{0})}}
:= \prod_{e\in\Edg(v_{0})}\pi_{e}^{*}\mathbf{e}(\mathcal{U}_{0}^{\prime})
\cdot \mathbf{e}(L_{\widetilde{\Gamma}}^{*}\otimes E_{\mu(v_{0})}\otimes \gamma^{*}).
\eea

 We \emph{define} the formal localization contribution of $\widetilde{\Gamma}=\big(\Ver,\Edg,v_{0};\Ver_{+},\Ver_{0};\mathfrak{m},\mathfrak{d},\mathfrak{e}\big)$ to
 $$\langle \mu_{1},\cdots,\mu_{|J|}\rangle_{1,J,d}^{0;X}$$ by
 \bea\label{22}
 \Cont_{\widetilde{\Gamma}}\big(\langle \mu_{1},\cdots,\mu_{|J|}\rangle_{1,J,d}^{0;X}\big)&:=&
 \frac{1}{\Aut(\widetilde{\Gamma})}
 \prod_{j\in J}\mu_{j}\big|_{P_{\mathfrak{m}\circ\mathfrak{e}(j)}}
 \cdot\int_{\widetilde{\mathcal{Z}}_{\widetilde{\Gamma}}}
 \frac{\mathbf{e}(\mathcal{U}_{1})}{\mathbf{e}(\mathcal{N}\widetilde{\mathcal{Z}}_{\widetilde{\Gamma}})}\nn\\
 &=&\frac{1}{\Aut(\widetilde{\Gamma})}\prod_{j\in J}\mu_{j}\big|_{P_{\mathfrak{m}\circ\mathfrak{e}(j)}}
 \cdot
 \int_{\widetilde{\mathcal{Z}}_{\widetilde{\Gamma}}}
 \frac{\mathbf{e}(\mathcal{U}_{1}^{\prime})}{\mathbf{e}(E)\big|_{P_{\mathfrak{m}(v_{0})}}
 \mathbf{e}(\mathcal{N}\widetilde{\mathcal{Z}}_{\widetilde{\Gamma}})}.
 \eea

\begin{definition}\label{26}
Let $Y$ be an algebraic GKM manifold and $E$ a concave equivariant vector bundle over $Y$.
For the local space $X=\mathrm{Tot}(E\rightarrow Y)$ and $\mu_{1},\cdots,\mu_{J}\in H_{\mathbb{T}}^{*}(Y)$ we define
\bea\label{76}
\langle \mu_{1},\cdots,\mu_{|J|}\rangle_{1,J,d}^{0;X}&:=&\sum_{\Gamma\in \DOL_{J}^{d}}\Cont_{\Gamma}\big(\langle \mu_{1},\cdots,\mu_{|J|}\rangle_{1,J,d}^{0;X}\big)\nn\\
&&+\sum_{\widetilde{\Gamma}\in \RDRT_{J}^{d}/\sim}\Cont_{\widetilde{\Gamma}}\big(\langle \mu_{1},\cdots,\mu_{|J|}\rangle_{1,J,d}^{0;X}\big).
\eea
In particular, when $X$ is a local Calabi-Yau space, the \emph{reduced genus one degree $d$ Gromov-Witten invariants} of $X$ is defined by
\bea\label{27}
N_{1,d}^{0;X}:=\sum_{\Gamma\in \DOL_{\emptyset}^{d}}\Cont_{\Gamma}(N_{1,d}^{0;X})+\sum_{\widetilde{\Gamma}\in \RDRT_{\emptyset}^{d}/\sim}\Cont_{\widetilde{\Gamma}}(N_{1,d}^{0;X}),
\eea
where $\Cont_{\Gamma}(N_{1,d}^{0;X})=\Cont_{\Gamma}(\langle \cdot\rangle_{1,\emptyset,d}^{0;X})$ for $\Gamma\in\DOL_{\emptyset}^{d}$ and
 $\Cont_{\widetilde{\Gamma}}(N_{1,d}^{0;X})=\Cont_{\widetilde{\Gamma}}(\langle \cdot\rangle_{1,\emptyset,d}^{0;X})$ for $\widetilde{\Gamma}\in \RDRT_{\emptyset}^{d}$.
\end{definition}

By this definition, when $X$ is a local Calabi-Yau space, $N_{1,d}^{0;X}$ is \emph{a priori} an element of $\mathbb{Q}_{\alpha}$. We will see that as a corollary of the LSvR, we have in fact $N_{1,d}^{0;X}\in \mathbb{Q}$. \\

Now we write (\ref{22}) as a product of contributions of edges and vertices. Since we will finally write the summing
\ben
\sum_{\widetilde{\Gamma}\in \RDRT_{J}^{d}/\sim}\Cont_{\widetilde{\Gamma}}\big(\langle \mu_{1},\cdots,\mu_{|J|}\rangle_{1,J,d}^{0;X}\big)
\een
into a summing over $\DRT_{J}^{d}$, we need to gather the contributions from the vertices in $\Edg_{0}$ and put it into the contribution of the root $v_{0}$. For this, let us introduce some notations.\\

Suppose $\mathfrak{m}(v_{0})=i$. For $e=\{v_{0},v\}\in \Edg(v_{0})\backslash (\Edg_{0}\cup\Edg_{+})$, let $I_{e}=\{e\}$, and let $\omega_{e}=\frac{\alpha_{i,\mathfrak{m}(v)}}{d(e)}$. For $e=\{v_{0},v\}\in \Edg_{0}$, let $I_{e}=\Edg(v)\backslash\{e\}$,
and for $f=\{v,v^{\prime}\}\in\Edg(v)\backslash\{e\}$, let $\omega_{f}=\frac{\alpha_{i,\mathfrak{m}(v^{\prime})}}{d(f)}$.
Moreover, let $\omega_{+}=\frac{\alpha_{i,\mathfrak{m}(v)}}{d(e)}$ for any $e=\{v_{0},v\}\in\Edg_{+}$, which is well-defined because
$\alpha_{\mathfrak{m}(v)}$ and $d(e)$ are independent of the choice of $e\in\Edg_{+}$; thus $\omega_{+}$ is the equivariant Euler class of $L_{\widetilde{\Gamma}}$. Let
\bea\label{37}
\overline{|\Edg(v_{0})|}=|\Edg(v_{0})\backslash \Edg_{0}|+\sum_{v\in \Edg_{0}}\big(|\Edg(v)|-1\big).
\eea

Equivalently, $\overline{|\Edg(v_{0})|}$  is equal to $|\Edg^{\Gamma}(v_{0})|$, where $\Gamma=\pi_{\RDRT}(\widetilde{\Gamma})$.\\

 Then we define
\begin{multline}\label{20}
\Cont_{\widetilde{\Gamma};v_{0}}\big(\langle \mu_{1},\cdots,\mu_{|J|}\rangle_{1,J,d}^{0;X}\big):=
\prod_{j\in \mathfrak{e}^{-1}(\{v_{0}\cup\Ver_{0}\})}\mu_{j}\big|_{P_{\mathfrak{m}(v_{0})}}
\cdot\Big(\prod_{k=1}^{l}\varepsilon_{i,k}\prod_{P_{j}\in \Nb(P_{i})}\alpha_{i,j}\Big)^{\overline{|\Edg(v_{0})|}-1}\\
\cdot\int_{\tM_{1,(\Edg(v_{0}),\mathfrak{e}^{-1}(v_{0}))}\times \mathbb{P}^{|\Edg_{+}|-1}}\Bigg(\frac{\prod_{k=1}^{l}(-\omega_{+}+\varepsilon_{i,k}+H)
\prod_{P_{j}\in \Nb(P_{i})}(-\omega_{+}+\alpha_{i,j}+H)}{\omega_{+}-\tilde{\psi}-H}\\
\cdot
\prod_{e\in \Edg(v_{0})-\Edg_{+}-\Edg_{0}}\frac{1}{\omega_{e}-\omega_{+}+H}
\prod_{v\in \Ver_{0}}\\
\int_{\Mbar_{0,I_{\{v_{0},v\}}}\sqcup \{0_{s}\}\sqcup \mathfrak{e}^{-1}(v)}\frac{1}{(-\omega_{+}-\psi_{0_{s}}+H)\prod_{f\in I_{e}}(\omega_{f}-\psi_{f})}\Bigg).
\end{multline}

For edges in $\Edg\backslash\Edg_{0}$ and vertices in $\Ver\backslash(\{v_{0}\}\cup\Ver_{0})$ , the contributions is defined as the same as those in (\ref{11}) respectively. Then by (\ref{23}), (\ref{24}), and (\ref{22}), it is not hard to see that, for $\widetilde{\Gamma}\in \RDRT_{J}^{d}$,
\bea\label{28}
&&\Cont_{\widetilde{\Gamma}}\big(\langle \mu_{1},\cdots,\mu_{|J|}\rangle_{1,J,d}^{0;X}\big)=\frac{1}{\Aut(\widetilde{\Gamma})}\nn\\
&&\cdot\prod_{v\in \Ver\backslash\Ver_{0}}\Cont_{\widetilde{\Gamma};v}\big(\langle \mu_{1},\cdots,\mu_{|J|}\rangle_{1,J,d}^{0;X}\big)
\prod_{e\in \Edg\backslash\Edg_{0}}\Cont_{\widetilde{\Gamma};e}\big(\langle \mu_{1},\cdots,\mu_{|J|}\rangle_{1,J,d}^{0;X}\big).
\eea
There is a canonical bijection between $\Ver\backslash\Ver_{0}$ and $\Ver\big(\pi_{\RDRT}(\widetilde{\Gamma})\big)$, and between $\Edg\backslash\Edg_{0}$ and $\Edg\big(\pi_{\RDRT}(\widetilde{\Gamma})\big)$. Using this bijection,
from (\ref{28}) we finally obtain
\bea\label{29}
\langle \mu_{1},\cdots,\mu_{|J|}\rangle_{1,J,d}^{0;X}=\sum_{\Gamma\in \DOL_{J}^{d}}\Cont_{\Gamma}\big(\langle \mu_{1},\cdots,\mu_{|J|}\rangle_{1,J,d}^{0;X}\big)+\sum_{\Gamma\in \DRT_{J}^{d}}\Cont_{\Gamma}\big(\langle \mu_{1},\cdots,\mu_{|J|}\rangle_{1,J,d}^{0;X}\big),
\eea
where for $\Gamma\in \DRT_{J}^{d}$,
\bea\label{30}
&&\Cont_{\Gamma}\big(\langle \mu_{1},\cdots,\mu_{|J|}\rangle_{1,J,d}^{0;X}\big)=\frac{1}{\Aut(\Gamma)}\nn\\
&&\cdot\sum_{\widetilde{\Gamma}\in\pi_{\RDRT}^{-1}(\Gamma)}
\prod_{v\in \Ver(\Gamma)}\Cont_{\widetilde{\Gamma};v}\big(\langle \mu_{1},\cdots,\mu_{|J|}\rangle_{1,J,d}^{0;X}\big)
\prod_{e\in \Edg(\Gamma)}\Cont_{\widetilde{\Gamma};e}\big(\langle \mu_{1},\cdots,\mu_{|J|}\rangle_{1,J,d}^{0;X}\big).
\eea

\begin{remark}
When $Y=\mathbb{P}^{n-1}$, our formal fixed locus associated to a refined decorated rooted tree may be different from that in \cite{VZ}, but the factor $\mathbf{e}(\mathcal{N}\widetilde{\mathcal{Z}}_{\widetilde{\Gamma}})$ in the localization contribution is the same, as remarked in 
\cite[footnote 16]{Zinger1}.
\end{remark}

\section{Localized standard versus reduced formula}

\subsection{}
Let $Y$ be an  algebraic GKM manifold of dimension $n-1$, $E$ a concave equivariant vector bundle of rank $l$ over $Y$, and $X=\mathrm{Tot}(E\rightarrow Y)$.

Now we state the first main theorem of this article.
\begin{theorem}\label{32}
Let $\mu_{1},\cdots,\mu_{|J|}\in H_{\mathbb{T}}^{*}(Y)$.
For every decorated rooted tree $\Gamma\in \DRT_{J}^{d}$, we have the \emph{LSvR}
\bea\label{64}
&&\Cont_{\Gamma}(\langle\mu_{1},\cdots,\mu_{|J|}\rangle_{1,J,d}^{X})\nn\\
&&=\Cont_{\Gamma}(\langle\mu_{1},\cdots,\mu_{|J|}\rangle_{1,J,d}^{0;X})
+\sum_{m\geq 1}\sum_{J^{\prime}\subset J}\Big[\frac{(-1)^{m+|J^{\prime}|}m^{|J^{\prime}|}(m-1)!}{24}\nn\\
&&\sum_{p=0}^{n+l-1-|J^{\prime}|-2m}
\Cont_{\Gamma}\Big(\langle \eta_{p}c_{n+l-1-|J^{\prime}|-2m-p}(TX);\mu_{1},\cdots,\mu_{|J|}\rangle_{(m,J-J^{\prime},d)}^{X}\Big)\Big].
\eea
In particular, when $X$ is a local Calabi-Yau space, for every $\Gamma\in\DRT_{\emptyset}^{d}$ we have
\begin{multline}\label{33}
\Cont_{\Gamma}(N_{1,d}^{X})=\Cont_{\Gamma}(N_{1,d}^{0;X})+\frac{1}{24}\sum_{m\geq 1}\Big[(-1)^{m}(m-1)!\\
\sum_{p=0}^{n+l-1-2m}\Cont_{\Gamma}\big(\langle\eta_{p}c_{n+l-1-2m-p}(TX);\rangle_{(m,\emptyset,d)}^{X}\big)\Big].
\end{multline}
\end{theorem}
Note that by definition, for $\Gamma\in\DOL_{J}^{d}$ we have
\bea\label{38}
\Cont_{\Gamma}\big(\langle\mu_{1},\cdots,\mu_{|J|}\rangle_{1,J,d}^{X}\big)
=\Cont_{\Gamma}\big(\langle\mu_{1},\cdots,\mu_{|J|}\rangle_{1,J,d}^{0;X}\big)
\eea
thus by (\ref{64}) and (\ref{34}), (\ref{18}), (\ref{27}) we have
\begin{corollary}\label{9}
\begin{multline}\label{78}
\langle\mu_{1},\cdots,\mu_{|J|}\rangle_{1,J,d}^{X}
=\langle\mu_{1},\cdots,\mu_{|J|}\rangle_{1,J,d}^{0;X}+\sum_{m\geq 1}\sum_{J^{\prime}\subset J}\Big[\frac{(-1)^{m+|J^{\prime}|}m^{|J^{\prime}|}(m-1)!}{24}\\
\sum_{p=0}^{n+l-1-|J^{\prime}|-2m}\langle\eta_{p}
\mathbf{c}_{n+l-1-|J^{\prime}|-2m-p}(TX);\mu_{1},\cdots,\mu_{|J|}\rangle_{(m,J-J^{\prime},d)}^{X}\Big].
\end{multline}
In particular, when $X$ is a local Calabi-Yau space,
\bea\label{39}
N_{1,d}^{0;X}=N_{1,d}^{0;X}+\frac{1}{24}\sum_{m\geq 1}(-1)^{m}(m-1)!
\sum_{p=0}^{n+l-1-2m}\langle\eta_{p}c_{n+l-1-2m-p}(TX);\rangle_{(m,\emptyset,d)}^{X}.
\eea
\end{corollary}

Note that when $X$ is a local Calabi-Yau space, in our definition of $N_{1,d}^{0;X}$  we have fixed a choice of the linearization of $E$. A consequence of corollary \ref{9} is that $N_{1,d}^{0;X}$ is independent of the choice of linearization of $E$.\\

The proof of theorem \ref{32} will occupy sections 3.1-3.4.
First note that, in (\ref{31}), (\ref{19}) and (\ref{30}), the factors
\ben
\frac{1}{\Aut(\Gamma)}\prod_{v\in \Ver\backslash\{v_{0}\}}\Cont_{\Gamma;v}(\cdot)\prod_{e\in \Edg}\Cont_{\Gamma;e}(\cdot)
\een
are common; in fact by definition, for $v\in\Ver\backslash\{v_{0}\}$, $\Cont_{\Gamma;v}(\cdot)$ are equal for the three types of invariants in (\ref{33}), and so are  $\Cont_{\Gamma;e}(\cdot)$ for $e\in\Edg$. So it suffices to show
\bea\label{41}
&&\Cont_{\Gamma;v_{0}}\big(\langle\mu_{1},\cdots,\mu_{|J|}\rangle_{1,J,d}^{X}\big)=
\sum_{\widetilde{\Gamma}\in\pi_{\RDRT}^{-1}(\Gamma)}\Cont_{\widetilde{\Gamma};v_{0}}\big(\langle\mu_{1},\cdots,\mu_{|J|}\rangle_{1,J,d}^{0;X}\big)\nn\\
&&+\sum_{m\geq 1}\sum_{J^{\prime}\subset J}\frac{(-1)^{m+|J^{\prime}|}m^{|J^{\prime}|}}{24m}
\sum_{p=0}^{n+l-1-|J^{\prime}|-2m}\sum_{\Gamma^{\mathfrak{c}}\in\pi_{\CDRT}^{-1}\circ\rho_{J-J^{\prime},J}^{-1}(\Gamma)}\nn\\
&&\Cont_{\Gamma^{\mathfrak{c}};v_{0}}
\big(\langle\tilde{\eta}_{p}\mathbf{c}_{n+l-1-|J^{\prime}|-2m-p}(TX);\mu_{1},\cdots,\mu_{|J|}\rangle_{(m,J-J^{\prime},d)}^{X}\big).
\eea
This formula concerns only the root $v_{0}$ together with $\mathfrak{m}(v_{0})$, $\Edg(v_{0})$ together with their degrees, and the vertices $v$ neighbored to $v_{0}$ together with their labels $\mathfrak{m}(v)$. So we only need to show (\ref{41}) for \emph{decorated stars}. Here by a \emph{decorated star}, we mean a decorated rooted tree with $\val(v)=1$ for all $v\in \Ver\backslash\{v_{0}\}$. The set of decorated stars of degree $d$ with the set of marked points $J$ is denoted by $\DS_{J}^{d}$; we have $J=\mathfrak{e}^{-1}(v_{0})$.
Note further that the three (groups of) terms in (\ref{41}) has a common factor (recall (\ref{37}) and the statement following it)
\ben
\prod_{j\in\mathfrak{e}^{-1}(v_{0})}\mu_{j}\big|_{P_{\mathfrak{m}(v_{0})}}\cdot\Big(\prod_{P_{j}\in\Nb(P_{i})}
\alpha_{i,j}\prod_{k=1}^{l}\varepsilon_{i,k}\Big)^{|\Edg(v_{0})|-1}.
\een
Therefore, for $\Gamma\in\DS_{J}^{d}$ with $\mu(v_{0})=i$, let
\bea\label{42}
\VC_{\Gamma}^{X}=\int_{\Mbar_{1,\Edg(v_{0})\sqcup \mathfrak{e}^{-1}(v_{0})}}\frac{\prod_{k=1}^{l}\Lambda_{1}^{\vee}(\varepsilon_{i,k})
\prod_{P_{j}\in\Nb(P_{i})}\Lambda_{1}^{\vee}(\alpha_{i,j})
}{\prod_{e\in \Edg(v_{0})}\Big(\frac{\alpha_{v_{0},e}}{d(e)}-\psi_{(v_{0},e)}\Big)},
\eea

\bea\label{44}
\VC_{\Gamma}^{0;X}
&=&
\sum_{\widetilde{\Gamma}\in\pi_{\RDRT}^{-1}(\Gamma)}\int_{\tM_{1,(\Edg(v_{0},\mathfrak{e}^{-1}(v_{0}))}
\times \mathbb{P}^{|\Edg_{+}|-1}}\nn\\
&&\Bigg(\frac{\prod_{k=1}^{l}(-\omega_{+}+\varepsilon_{i,k}+H)\prod_{P_{j}\in\Nb(P_{i})}
(-\omega_{+}+\alpha_{i,j}+H)}{\omega_{+}-\tilde{\psi}-H}\nn\\
&&\cdot
\prod_{e\in \Edg(v_{0})-\Edg_{+}-\Edg_{0}}\frac{1}{\omega_{e}-\omega_{+}+H}
\prod_{v\in \Ver_{0}}\nn\\
&&\int_{\Mbar_{0,I_{\{v_{0},v\}}}\sqcup \{0_{s}\}\sqcup \mathfrak{e}^{-1}(v)}\frac{1}{(-\omega_{+}-\psi_{0_{s}}+H)\prod_{f\in I_{e}}(\omega_{f}-\psi_{f})}\Bigg),
\eea

\bea\label{43}
&&\VC_{\Gamma,J-J^{\prime}}^{(m,p,q),X}=
[x^{q}]\Big(\prod_{k=1}^{l}(1+\varepsilon_{i,k}x)\prod_{P_{j}\in\Nb(P_{i})}(1+x(\alpha_{i,j}))\Big)\nn\\
&&\cdot\sum_{\Gamma^{\mathfrak{c}}\in\pi_{\CDRT}^{-1}\circ\rho_{J-J^{\prime},J}^{-1}(\Gamma)}\int_{\Mbar_{0,I_{(1)}\sqcup K_{(1)}\sqcup \{0_{1}\}}\times\cdots\times\Mbar_{0,I_{(m)}\sqcup K_{(m)}\sqcup \{0_{m}\}}}\nn\\
&&[x^{p}]\Bigg(\frac{1}{\prod_{e\in\Edg(v_{0})}(\frac{\alpha_{v_{0},e}}{d(e)}-\psi_{(v_{0},e)}) \prod_{s=1}^{m}(1-x\psi_{0_{s}})}\Bigg).
\eea

Then to show (\ref{41}) it suffices to show for every $\Gamma\in\DS_{J}^{d}$,
\bea\label{45}
\VC_{\Gamma}^{X}=\VC_{\Gamma}^{0;X}
+\sum_{m=1}^{2m\leq n+l-1}\sum_{J^{\prime}\subset J}\frac{(-1)^{m+|J^{\prime}|}m^{|J^{\prime}|}}{24m}
\sum_{p=0}^{n+l-1-|J^{\prime}|-2m}\VC_{\Gamma,J-J^{\prime}}^{(m,p,n+l-1-|J^{\prime}|-2m-p),X}.\nn\\
\eea

  Suppose $\Gamma$ is a decorated star whose root is labeled by $i$ with $|\Edg(v_{0})|=r$,  the vertices other than $v_{0}$ are labeled by $j_{1},\cdots,j_{r}$, $\mathfrak{e}^{-1}(v_{0})=J$, and the degree of the corresponding edges are $d_{1},\cdots,d_{r}$. Let $\omega_{s}=\frac{\alpha_{i,j_{s}}}{d_{s}}$; since $Y$ is an  algebraic GKM manifold,  $\omega_{s_{1}}=\omega_{s_{2}}$ if and only if $j_{s_{1}}=j_{s_{2}}$ and $d_{s_{1}}=d_{s_{2}}$. Let $\vec{\omega}=(\omega_{1},\cdots,\omega_{r})$, and we use $\Gamma_{i;\vec{\omega};J}$ to represent this decorated star. When $J=\emptyset$, we denote $\Gamma_{i;\vec{\omega}}=\Gamma_{i;\vec{\omega};J}$ for short. When $d_{s}=1$ for $1\leq s\leq r$ and $j_{s}$ are pairwisely distinct for $1\leq s\leq r$, we call $\Gamma_{i;\vec{\omega};J}$ a \emph{simply decorated star}.\\

  In section 3.2, we prove (\ref{45}) for simply decorated stars with $J=\emptyset$. In section 3.3, we prove (\ref{45}) for all decorated stars with $J=\emptyset$. In section 3.4 we prove for $J\neq \emptyset$. In section 3.5 we state another form of LSvR and sketch a proof for it. \\

\subsection{Simply decorated stars}
Let $\Gamma_{i;\vec{\omega}}\in\DS_{\emptyset}^{d}$ be a simply decorated star with $J=\emptyset$.
\begin{lemma}
\bea\label{10}
\int_{\Mbar_{1,r}}\frac{\lambda_{1}}
{\prod_{k=1}^{r}(w_{k}-\psi_{k})}=
\frac{1}
{24\prod_{k=1}^{r}w_{k}}\Big(\sum_{k=1}^{r}\frac{1}{w_{k}}\Big)^{r-1}.
\eea
\end{lemma}
\begin{proof}: This follows straightforwardly from \cite[proposition 3.1]{Givental} and (\ref{21}),
or from the $\lambda_{g}$-conjecture for $g=1$.
\end{proof}

\begin{lemma}\label{79}
\bea\label{80}
\VC_{\Gamma_{i;\vec{\omega}}}^{0;X}=0.
\eea
\end{lemma}
\begin{proof}: Since $\Gamma_{i;\vec{\omega}}$ is simply decorated,  by the definition of refined decorated rooted trees, for every $\widetilde{\Gamma}\in\pi_{\RDRT}^{-1}(\Gamma_{i;\vec{\omega}})$ we have $|\Edg_{+}|=1$ and $\omega_{+}=\alpha_{i,j}$ for some $P_{j}\in\Nb(P_{i})$. Thus in the righthand-side (\ref{44}), the integrand in the second row has a factor $H$, but $\mathbb{P}^{|\Edg_{+}|-1}$ is a point, so the conclusion follows.
\end{proof}

By (\ref{10}), we have
\bea\label{46}
\VC_{\Gamma_{i;\vec{\omega}}}^{X}&=&\int_{\Mbar_{1,r}}\frac{ \prod_{k=1}^{l}\Lambda_{1}^{\vee}(\varepsilon_{i,k})\cdot\prod_{P_{j}\in\Nb(P_{i})}
\Lambda_{1}^{\vee}(\alpha_{i,j})}
{\prod_{k=1}^{r}(\alpha_{i,j_{k}}-\psi_{k})}\nn\\
&=&\prod_{k=1}^{l}\varepsilon_{i,k}\prod_{P_{j}\in\Nb(P_{i})}\alpha_{i,j}\int_{\Mbar_{1,r}}\frac{1}
{\prod_{k=1}^{r}(\alpha_{i,j_{k}}-\psi_{k})}\nn\\
&&-[x]\Big(\prod_{k=1}^{l}(x+\varepsilon_{i,k})\prod_{P_{j}\in\Nb(P_{i})}(x+\alpha_{i,j})\Big)\cdot\int_{\Mbar_{1,r}}\frac{\lambda_{1}}
{\prod_{k=1}^{r}(\alpha_{i,j_{k}}-\psi_{k})}\nn\\
&=&\prod_{k=1}^{l}\varepsilon_{i,k}\prod_{P_{j}\in\Nb(P_{i})}\alpha_{i,j}\int_{\Mbar_{1,r}}\frac{1}
{\prod_{k=1}^{r}(\alpha_{i,j_{k}}-\psi_{k})}\nn\\
&&-[x]\Big(\prod_{k=1}^{l}(x+\varepsilon_{i,k})\prod_{P_{j}\in\Nb(P_{i})}(x+\alpha_{i,j})\Big)\cdot
\frac{1}
{24\prod_{k=1}^{r}\alpha_{i,j_{k}}}\Big(\sum_{k=1}^{r}\frac{1}{\alpha_{i,j_{k}}}\Big)^{r-1}.\nn\\
\eea

On the other hand,
\bea\label{47}
&&\frac{(-1)^{m}}{24m}\sum_{p=0}^{n+l-1-2m}\VC_{\Gamma_{i;\vec{\omega}}}^{(m,p,n+l-1-2m-p),X}\nn\\
&=&\frac{(-1)^{m}}{24m}\sum_{I\in \mathcal{A}_{m}([r]) }
\int_{\Mbar_{0,I_{(1)}\sqcup \{0_{1}\}}\times\cdots\times\Mbar_{0,I_{(m)}\sqcup \{0_{m}\}}}\nn\\
&&[x^{n+l-1-2m}]\Bigg(\frac{\prod_{k=1}^{l}(1+\varepsilon_{i,k}x)\prod_{P_{j}\in\Nb(P_{i})}(1+\alpha_{i,j}x)}
{\prod_{k=1}^{r}(\alpha_{i,j_{k}}-\psi_{k})\cdot \prod_{s=1}^{m}(1-x\psi_{0_{s}})}\Bigg)\nn\\
&\stackrel{(*)}{=}&\frac{(-1)^{m}}{24m}\sum_{I\in \mathcal{A}_{m}([r]) }
\int_{\Mbar_{0,I_{(1)}\sqcup \{0_{1}\}}\times\cdots\times\Mbar_{0,I_{(m)}\sqcup \{0_{m}\}}}
[x^{m}]\Bigg(\frac{\prod_{k=1}^{l}(x+\varepsilon_{i,k})
\prod_{P_{j}\in\Nb(P_{i})}(x+\alpha_{i,j})}{\prod_{k=1}^{r}(\alpha_{i,j_{k}}-\psi_{k})\cdot \prod_{s=1}^{m}(x-\psi_{0_{s}})}\Bigg)\nn\\
&=&\frac{(-1)^{m}}{24m}\sum_{I\in \mathcal{A}_{m}([r]) }[x^{m}]\Bigg(\prod_{k=1}^{l}(x+\varepsilon_{i,k})\prod_{P_{j}\in\Nb(P_{i})}(x+\alpha_{i,j})\nn\\
&&\cdot
\int_{\Mbar_{0,I_{(1)}\sqcup \{0_{1}\}}\times\cdots\times\Mbar_{0,I_{(m)}\sqcup \{0_{m}\}}}
\frac{1}{\prod_{k=1}^{r}(\alpha_{i,j_{k}}-\psi_{k})\cdot \prod_{s=1}^{m}(x-\psi_{0_{s}})}\Bigg)\nn\\
&=&\frac{(-1)^{m}}{24m}\sum_{I\in \mathcal{A}_{m}([r]) }[x^{m}]\Bigg(\frac{\prod_{k=1}^{l}(x+\varepsilon_{i,k})\prod_{P_{j}\in\Nb(P_{i})}(x+\alpha_{i,j})}
{x^{m}\prod_{k=1}^{r}\alpha_{i,j_{k}}}
\cdot\prod_{s=1}^{m}\Big(\frac{1}{x}+\sum_{j\in I_{(s)}}\frac{1}{\alpha_{i,j}}\Big)^{|I_{(s)}|-2}\Bigg).\nn\\
\eea
Some attention should be paid to the equality $(*)$. For a fixed $I\in \mathcal{A}_{m}([r])$, without loss of generality we may assume $I_{(s)}=\{j_{s}\}$ for $1\leq s\leq r_{1}$ and $|I_{(s)}|\geq 2$ for $r_{1}+1\leq s\leq r$. Thus
\bea\label{65}
&&\int_{\Mbar_{0,I_{(1)}\sqcup \{0_{1}\}}\times\cdots\times\Mbar_{0,I_{(m)}\sqcup \{0_{m}\}}}
[x^{n+l-1-2m}]\Bigg(\frac{\prod_{k=1}^{l}(1+\varepsilon_{i,k}x)\prod_{P_{j}\in\Nb(P_{i})}(1+\alpha_{i,j}x)}
{\prod_{k=1}^{r}(\alpha_{i,j_{k}}-\psi_{k})\cdot \prod_{s=1}^{m}(1-x\psi_{0_{s}})}\Bigg)\nn\\
&=&[x^{n+l-1-2m}]\Bigg( \prod_{s=1}^{r_{1}}\Big(1+\alpha_{i,j_{k}}x\Big)^{-1}\nn\\
&&\cdot\int_{\Mbar_{0,I_{(r_{1}+1)}\sqcup \{0_{r_{1}+1}\}}\times\cdots\times\Mbar_{0,I_{(m)}\sqcup \{0_{m}\}}}
\frac{\prod_{k=1}^{l}(1+\varepsilon_{i,k}x)\prod_{P_{j}\in\Nb(P_{i})}(1+\alpha_{i,j}x)}
{\prod_{k=r_{1}+1}^{r}(\alpha_{i,j_{k}}-\psi_{k})\cdot \prod_{s=r_{1}+1}^{m}(1-x\psi_{0_{s}})}\Bigg).
\eea
 By the assumption that $\alpha_{i,_k}$ for $1\leq k\leq r$ are pairwisely distinct, the factor $\prod_{s=1}^{r_{1}}(1+\alpha_{i,j_k}x)$ in the denominator divides the product $\prod_{P_{j}\in\Nb(P_{i})}(1+\alpha_{i,j}x)$ in the numerator. Thus the righthand-side of (\ref{65}) as a rational function of $x$ has only poles at $x=0$ and $x=\infty$, so the equality $(*)$ follows. For the general decorated stars, the corresponding expression to the righthand-side of (\ref{65}) has other poles, from which the contributions of refined decorated stars arise. \\

 Now from the following proposition and by lemma \ref{79} we see that (\ref{45}) holds for simply decorated stars with $J=\emptyset$.
\begin{proposition}\label{49}
Let $w_{1},\cdots,w_{r}$ be independent variables. Then we have
\bea\label{50}
&&\sum_{m=1}^{r}\frac{(-1)^{m}}{24m}\sum_{I\in \mathcal{A}_{m}([r]) }[x^{m}]\Bigg(\frac{\prod_{k=1}^{l}(x+\varepsilon_{i,k})\prod_{P_{j}\in\Nb(P_{i})}(x+\alpha_{i,j)}}
{x^{m}\prod_{k=1}^{r}w_{k}}
\cdot\prod_{s=1}^{m}\Big(\frac{1}{x}+\sum_{j\in I_{(s)}}\frac{1}{w_{j}}\Big)^{|I_{(s)}|-2}\Bigg)\nn\\
&=&\prod_{k=1}^{l}\varepsilon_{i,k}\prod_{P_{j}\in\Nb(P_{i})}\alpha_{i,j}\int_{\Mbar_{1,r}}\frac{1}
{\prod_{k=1}^{r}(w_{k}-\psi_{k})}\nn\\
&&-[x]\Big(\prod_{k=1}^{l}(x+\varepsilon_{i,k})\prod_{P_{j}\in\Nb(P_{i})}(x+\alpha_{i,j})\Big)\cdot
\frac{1}
{24\prod_{k=1}^{r}w_{k}}\Big(\sum_{k=1}^{r}\frac{1}{w_{k}}\Big)^{r-1}.
\eea
\end{proposition}
The proposition follows immediately from the following lemma.

\begin{lemma}\label{181}
We have
\bea\label{59}
\sum_{m=1}^{r}\frac{(-1)^{m}}{m}\sum_{I\in \mathcal{A}_{m}([r]) }[x^{r-1}]\Bigg(\prod_{s=1}^{m}\Big(1+\sum_{j\in I_{(s)}}xw_{j}\Big)^{|I_{(s)}|-2}\Bigg)=-\Big(\sum_{k=1}^{r}w_{k}\Big)^{r-1},
\eea

\bea\label{60}
&&[x^{r}]\Bigg(\sum_{m=1}^{r}\frac{(-1)^{m}}{24m}\sum_{I\in \mathcal{A}_{m}([r]) }\prod_{s=1}^{m}\Big(1+\sum_{j\in I_{(s)}}xw_{j}\Big)^{|I_{(s)}|-2}\Bigg)\nn\\
&=&\int_{\Mbar_{1,r}}\frac{1}
{\prod_{k=1}^{r}(1-w_{k}\psi_{k})},
\eea
and for  $2\leq p\leq r$,
\bea\label{61}
\sum_{m=1}^{r}\frac{(-1)^{m}}{m}\sum_{I\in \mathcal{A}_{m}([r]) }[x^{r-p}]\Bigg(\prod_{s=1}^{m}\Big(1+\sum_{j\in I_{(s)}}xw_{j}\Big)^{|I_{(s)}|-2}\Bigg)=0.
\eea
\end{lemma}
\begin{proof}
Let
\bea\label{183}
H_{r}(w_{1},\cdots,w_{r})=\sum_{m=1}^{r}\frac{(-1)^{m}}{24m}\sum_{I\in \mathcal{A}_{m}([r]) }\Bigg(\prod_{s=1}^{m}\Big(1+\sum_{j\in I_{(s)}}w_{j}\Big)^{|I_{(s)}|-2}\Bigg),
\eea
and
\ben
F_{r}(w_{1},\cdots,w_{r})=\int_{\Mbar_{1,r}}\frac{1}
{\prod_{k=1}^{r}(1-w_{k}\psi_{k})}.
\een
Consider the Taylor expansion of $H_{r}(w_{1},\cdots,w_{r})$ at $w_{1}=\cdots=w_{r}=0$, and for $p\in \mathbb{Z}^{\geq 0}$ let $H_{r,p}(w_{1},\cdots,w_{r})$ be the degree $p$ part of this expansion, which is a symmetric polynomial in $w_{1},\cdots, w_{r}$ of degree $p$. Then what we need to prove is
\bea\label{2}
H_{r,r-1}(w_{1},\cdots,w_{r})=-\frac{1}{24}\Big(\sum_{k=1}^{r}w_{k}\Big)^{r-1},
\eea
\bea\label{3}
H_{r,r}(w_{1},\cdots,w_{r})=F_{r}(w_{1},\cdots,w_{r}),
\eea
and for $r\in \mathbb{Z}^{>0}$, and $0\leq p\leq r-2$
\bea\label{1}
H_{r,p}(w_{1},\cdots,w_{r})=0.
\eea
We prove these identities by induction on $r$. The $r=1$ case is trivial. For $r>1$, by the definition of $H_{r}$, it is not hard to see that for a fixed $I\in \mathcal{A}_{m}([r-1])$ the coefficient of
\ben
\prod_{s=1}^{m}\Big(1+\sum_{j\in I_{(s)}}w_{j}\Big)^{|I_{(s)}|-2}
\een
in  $H_{r}(w_{1},\cdots,w_{r-1},0)$ is
\ben
\frac{(-1)^{m}}{24m}\sum_{s=1}^{m}\Big(1+\sum_{j\in I_{(s)}}w_{j}\Big)+\frac{(-1)^{m+1}}{24(m+1)}\cdot (m+1)=\frac{(-1)^{m}}{24m}\sum_{k=1}^{r-1}w_{k}.
\een
Therefore
\bea\label{4}
H_{r}(w_{1},\cdots,w_{r-1},0)=\sum_{k=1}^{r-1}w_{k}\cdot H_{r-1}(w_{1},\cdots,w_{r-1}).
\eea
Note that a symmetric polynomial $f(w_{1},\cdots,w_{r})$  of degree less than $r$ is uniquely determined by $f(w_{1},\cdots,w_{r-1},0)$, so  by induction we obtain (\ref{1}) and (\ref{2}).
For (\ref{3}),  note that by the string equation, (\ref{4}) holds for $F_{r}$, i.e.,  we have
\ben
F_{r}(w_{1},\cdots,w_{r-1},0)=\sum_{k=1}^{r-1}w_{k}\cdot F_{r-1}(w_{1},\cdots,w_{r-1}).
\een
Note also that a symmetric polynomial $f(w_{1},\cdots,w_{r})$  of degree $r$ is determined by $f(w_{1},\cdots,w_{r-1},0)$ up to  the coefficient of  $w_{1}\cdots w_{r}$. So by induction, to show that (\ref{3}) holds for $r$, it suffices to show that the coefficients of $w_{1}\cdots w_{r}$ in $H_{r,r}(w_{1},\cdots,w_{r})$ and in $F_{r}(w_{1},\cdots,w_{r})$ are equal.
But by the definition of $H_{r}$, it is easy to see that the coefficient of $w_{1}\cdots w_{r}$ in it (this monomial appears only when $m=r$ and $|I_{(s)}|=1$ for $1\leq s\leq r$) is
\ben
\frac{(r-1)!}{24},
\een
which is equal to $\int_{\Mbar_{1,r}}\psi_{1}\cdots\psi_{r}$, by the dilaton equation.
\end{proof}

\begin{remark}
Although the simply decorated stars are the most simple cases, they provide a prototype of the LSvR. As we will see, the contribution from the reduced invariants are \emph{correction terms} when the assumption of simply decorated stars is not satisfied.
The formula (\ref{60}) can be viewed as a \emph{combinatorial solution} to the $n$-point function of \cite{LX} in genus one. It is interesting to find a higher genera analog of this formula, which may shed light on the computation of Gromov-Witten invariants in higher genera.
\end{remark}

\subsection{General decorated stars}
Let Let $\Gamma_{i;\vec{\omega}}\in\DRT_{\emptyset}^{d}$ be a decorated star where $\vec{\omega}=(\omega_{1},\cdots,\omega_{r})$. Similar to (\ref{46}), we have
\bea\label{90}
\VC_{\Gamma_{i;\vec{\omega}}}^{X}
&=&\prod_{k=1}^{l}\varepsilon_{i,k}\prod_{P_{j}\in\Nb(P_{i})}\alpha_{i,j}\int_{\Mbar_{1,r}}\frac{1}
{\prod_{k=1}^{r}(\omega_{k}-\psi_{k})}\nn\\
&&-[x]\Big(\prod_{k=1}^{l}(x+\varepsilon_{i,k})\prod_{P_{j}\in\Nb(P_{i})}(x+\alpha_{i,j})\Big)\cdot
\frac{1}
{24\prod_{k=1}^{r}\omega_{k}}\Big(\sum_{k=1}^{r}\frac{1}{\omega_{k}}\Big)^{r-1}.\nn\\
\eea

Now we consider contributions of the refined decorated rooted trees whose underlying decorated rooted tree is $\Gamma_{i;\vec{\omega}}$.  Define $\widetilde{\mathcal{A}}_{m}([r];\vec{\omega})$ to be the set of 3-tuples $(I,U,V)\in \mathcal{P}([r])^{m}\times [m]\times[m]$ satisfying\\
(i) Writing $I$ as $I=(I_{(1)},\cdots,I_{(m)})$, then $\bigsqcup_{s=1}^{m}I_{(s)}=[r]$;\\
(ii) $U\cap V=\emptyset$;\\
(iii) $|I_{(s)}|=1$ for $s\in U\cup V$, and $|I_{(s)}|\geq 2$ for $s\not\in U\cup V$;\\
(iv) $\omega_{s}$ are equal to each other for $s\in U$, and $\omega_{i}\neq \omega_{s}$ for $s\in U$ and $i\in V$.\\

For $(I,U,V)\in \widetilde{\mathcal{A}}_{m}([r];\vec{\omega})$, denote the common weight $\omega_{s}$ for $s\in U$ by $\omega_{U}$ and suppose $I_{(s)}=\{i_{s}\}$ for $s\in U\cup V$, $I_{U}=\bigsqcup_{s\in U}I_{(s)}$.\\

By the definition of refined decorated rooted trees, it is not hard to see that there is a natural 1-1 correspondence between $\widetilde{\mathcal{A}}_{m}([r];\vec{\omega})$ and $\pi_{\RDRT}^{-1}(\Gamma_{i;\vec{\omega}})$. Thus by (\ref{44}) and (\ref{89}) we have
\bea\label{52}
\VC_{\Gamma_{i;\vec{\omega}}}^{0;X}=\sum_{m\geq 1}\sum_{(I,U,V)\in \widetilde{\mathcal{A}}_{m}([r];\vec{\omega}) }
\VC_{\Gamma_{i;(I,U,V);\vec{\omega}}}^{0;X},
\eea
where
\bea\label{53}
&&\VC_{\Gamma_{i;(I,U,V);\vec{\omega}}}^{0;X}\nn\\
&=&\frac{1}{m!}\int_{\tM_{1,m}\times \mathbb{P}^{|U|-1}}\Bigg(
\frac{\prod_{k=1}^{l}(-\omega_{U}+\varepsilon_{i,k}+H)\prod_{P_{j}\in\Nb(P_{i})}
(-\omega_{U}+\alpha_{i,j}+H)}{\omega_{U}-\tilde{\psi}-H}\nn\\
&&\cdot\prod_{s\in V}\frac{1}{-\omega_{U}+H+\omega_{i_s}}
\prod_{s\in [m]\backslash U\cup V}\int_{\Mbar_{0,I_{(s)}\sqcup \{0_{s}\}}}\frac{1}{(-\omega_{U}-\psi_{0_{s}}+H)\prod_{j\in I_{(s)}}(\omega_{j}-\psi_{j})}\Bigg)\nn\\
&=&\frac{(-1)^{m+1}(m-1)!}{24m!}\int_{\mathbb{P}^{|U|-1}}
\Bigg(\frac{\prod_{k=1}^{l}(H-\omega_{U}+\varepsilon_{i,k})\prod_{P_{j}\in\Nb(P_{i})}(H-\omega_{U}+\alpha_{i,j})}
{(H-\omega_{U})^{2m+1-|U|}\prod_{k\in [r]\backslash I_{U}}\omega_{k}}\nn\\
&&\cdot\prod_{s\in[m]\backslash U}\Big(\frac{1}{H-\omega_{U}}+\sum_{j\in I_{(s)}}\frac{1}{\omega_{j}}\Big)^{|I_{(s)}|-2}\Bigg)\nn\\
&=&\frac{(-1)^{m+1}}{24m}[y^{|U|-1}]\Bigg(\frac{\prod_{k=1}^{l}(y-\omega_{U}+\varepsilon_{i,k})
\prod_{P_{j}\in\Nb(P_{i})}(y-\omega_{U}+\alpha_{i,j})}
{(y-\omega_{U})^{2m+1-|U|}\prod_{k\in [r]\backslash I_{U}}\omega_{k}}\nn\\
&&\cdot\prod_{s\in[m]\backslash U}\Big(\frac{1}{y-\omega_{U}}+\sum_{j\in I_{(s)}}\frac{1}{\omega_{j}}\Big)^{|I_{(s)}|-2}\Bigg).
\eea

Note that there is a canonical projection
\ben
\widetilde{\pi}:\widetilde{\mathcal{A}}_{m}([r];\vec{\omega})\rightarrow \mathcal{A}_{m}([r])
\een
defined by $\pi(I,U,V)=I$. We are going to write
\bea\label{91}
\frac{(-1)^{m}}{24m}\sum_{p=0}^{n+l-1-2m}\VC_{\Gamma_{i;\vec{\omega}}}^{(m,p,n+l-1-2m-p),X}
 \eea
 as a sum over contributions from $I\in\mathcal{A}_{m}([r])$, and then compare the contribution of $I$ to (\ref{91}) and the sum of the $\VC_{\Gamma_{i;(I,U,V);\vec{\omega}}}^{0;X}$ for  $(I,U,V)\in\widetilde{\pi}^{-1}(I)$ .\\

Let us temporarily fix $I\in \mathcal{A}_{m}([r])$, and  suppose\\
(i) $|I_{s}|=1$ for $s\in \mathfrak{W}$ and $|I_{s}|\geq 2$ for $s\in [m]\backslash \mathfrak{W}$, where $\mathfrak{W}$ is a subset of $[m]$;\\
(ii) $\mathfrak{W}=\bigsqcup_{k\in \mathfrak{K}}U_{k}$ is a partition \emph{according to weight}. Precisely speaking, for every $s\in \mathfrak{W}$, suppose $I_{(s)}=\{i_{s}\}$, then for $i_{a},i_{b}\in I_{\mathfrak{W}}:=\bigsqcup_{s\in W}I_{(s)}$, we have $\omega_{i_a}=\omega_{i_b}$ if and only if $a$ and $b$ belong to the same $U_{k}$ for some $k\in\mathfrak{K}$.\\

We denote $I_{U_{k}}=\bigsqcup_{s\in U_{k}}I_{(s)}$ for $k\in \mathfrak{K}$, and  $I_{\mathfrak{W}}=\sum_{k\in\mathfrak{K}}I_{U_{k}}=\bigsqcup_{s\in W}I_{(s)}$.
Then we have
\bea\label{56}
&&\frac{(-1)^{m}}{24m}\sum_{p=0}^{n+l-1-2m}\VC_{\Gamma_{i;\vec{\omega}}}^{(m,p,n+l-1-2m-p),X}\nn\\
&=&\sum_{I\in\mathcal{A}_{m}([r])}\mathrm{Cont}_{I}\Big(
\frac{(-1)^{m}}{24m}\sum_{p=0}^{n+l-1-2m}\VC_{\Gamma_{i;\vec{\omega}}}^{(m,p,n+l-1-2m-p),X}\Big),\nn\\
\eea
where the contribution of $I\in \mathcal{A}_{m}([r])$ to
\bea\label{54}
\frac{(-1)^{m}}{24m}\sum_{p=0}^{n+l-1-2m}\VC_{\Gamma_{i;\vec{\omega}}}^{(m,p,n+l-1-2m-p),X}
\eea
is
\bea\label{55}
&&\sum_{I\in\mathcal{A}_{m}([r])}\mathrm{Cont}_{I}\Big(
\frac{(-1)^{m}}{24m}\sum_{p=0}^{n+l-1-2m}\VC_{\Gamma_{i;\vec{\omega}}}^{(m,p,n+l-1-2m-p),X}\Big),\nn\\
&=&\frac{(-1)^{m}}{24m}\sum_{p=0}^{n+l-1-2m}
\Bigg[\sum_{\stackrel{\sum_{s\in \mathfrak{W}}p_{i_{s}}=p}{\forall p_{i_{s}}\geq 0}}\prod_{s\in \mathfrak{W}}(-\omega_{i_{s}})^{p_{i_{s}}}\nn\\
&&\cdot
\int_{\prod_{s\in [m]\backslash\mathfrak{W}}\Mbar_{0,I_{(s)}\sqcup \{0_{s}\}}}
[x^{n+l-1-2m-p}]\Bigg(\frac{\prod_{t=1}^{l}(1+\varepsilon_{i,t}x)\prod_{P_{j}\in\Nb(P_{i})}(1+\alpha_{i,j}x)}{\prod_{k\in [r]\backslash I_{\mathfrak{W}}}(\omega_{k}-\psi_{k})\cdot \prod_{s\in[m]\backslash\mathfrak{W}}(1-x\psi_{0_{s}})}\Bigg)\Bigg]\nn\\
&=&\frac{(-1)^{m}}{24m}
\int_{\prod_{s\in [m]\backslash\mathfrak{W}}\Mbar_{0,I_{(s)}\sqcup \{0_{s}\}}}\nn\\
&&[x^{n+l-1-2m}]\Bigg(\frac{\prod_{t=1}^{l}(1+\varepsilon_{i,t}x)\prod_{P_{j}\in\Nb(P_{i})}
(1+\alpha_{i,j}x)}{\prod_{k\in\mathfrak{K}}(1+x\omega_{U_k})^{|U_{k}|}\prod_{k=[r]\backslash I_{\mathfrak{W}}}(\omega_{k}-\psi_{k})\cdot \prod_{s\in[m]\backslash\mathfrak{W}}(1-x\psi_{0_{s}})}\Bigg)\nn\\
&=&\frac{(-1)^{m}}{24m}\Res_{x=0}\Bigg(
\int_{\prod_{s\in [m]\backslash\mathfrak{W}}\Mbar_{0,I_{(s)}\sqcup \{0_{s}\}}}\nn\\
&&
\frac{\prod_{t=1}^{l}(1+\varepsilon_{i,t}x)\prod_{P_{j}\in\Nb(P_{i})}(1+\alpha_{i,j}x)}
{x^{n+l-2m}\prod_{k\in\mathfrak{K}}(1+x\omega_{U_k})^{|U_{k}|}\prod_{k\in [r]\backslash I_{\mathfrak{W}}}(\omega_{k}-\psi_{k})\cdot \prod_{s\in[m]\backslash\mathfrak{W}}(1-x\psi_{0_{s}})}\Bigg)\nn\\
&=&\frac{(-1)^{m}}{24m}\Big(-\Res_{x=\infty}-\sum_{k\in\mathfrak{K}}\Res_{x=-\frac{1}{\omega_{U_{k}}}}\Big)\Bigg(
\int_{\prod_{s\in [m]\backslash\mathfrak{W}}\Mbar_{0,I_{(s)}\sqcup \{0_{s}\}}}\nn\\
&&
\frac{\prod_{t=1}^{l}(1+\varepsilon_{i,t}x)\prod_{P_{j}\in\Nb(P_{i})}(1+\alpha_{i,j}x)}
{x^{n+l-2m}\prod_{k\in\mathfrak{K}}(1+x\omega_{U_k})^{|U_{k}|}\prod_{k\in [r]\backslash I_{\mathfrak{W}}}(\omega_{k}-\psi_{k})\cdot \prod_{s\in[m]\backslash\mathfrak{W}}(1-x\psi_{0_{s}})}\Bigg)\nn\\
&=&\frac{(-1)^{m}}{24m}\Big(\Res_{x=0}+\sum_{k\in\mathfrak{K}}\Res_{x=-\omega_{k}}\Big)\Bigg(
\int_{\prod_{s\in [m]\backslash\mathfrak{W}}\Mbar_{0,I_{(s)}\sqcup \{0_{s}\}}}\nn\\
&&
\frac{\prod_{t=1}^{l}(x+\varepsilon_{i,t})\prod_{P_{j}\in\Nb(P_{i})}(x+\alpha_{i,j})}
{x^{m+1}\prod_{k\in\mathfrak{K}}(x+\omega_{U_k})^{|U_{k}|}\prod_{k\in [r]\backslash I_{\mathfrak{W}}}(\omega_{k}-\psi_{k})\cdot \prod_{s\in[m]\backslash\mathfrak{W}}(x-\psi_{0_{s}})}\Bigg)\nn\\
&=&\frac{(-1)^{m}}{ 24m}
[x^{m}]\Bigg(\int_{\prod_{s\in [m]\backslash\mathfrak{W}}\Mbar_{0,I_{(s)}\sqcup \{0_{s}\}}}\nn\\
&&\frac{\prod_{t=1}^{l}(x+\varepsilon_{i,t})\prod_{P_{j}\in\Nb(P_{i})}(x+\alpha_{i,j})}
{\prod_{k\in\mathfrak{K}}(x+\omega_{U_k})^{|U_{k}|}\prod_{k\in [r]\backslash I_{\mathfrak{W}}}(\omega_{k}-\psi_{k})\cdot \prod_{s\in[m]\backslash\mathfrak{W}}(x-\psi_{0_{s}})}\Bigg)\nn\\
&&+\frac{(-1)^{m}}{ 24m}\sum_{k\in\mathfrak{K}}
[y^{|U_{k}|-1}]\Bigg(\int_{\prod_{s\in [m]\backslash\mathfrak{W}}\Mbar_{0,I_{(s)}\sqcup \{0_{s}\}}}\nn\\
&&\frac{\prod_{t=1}^{l}(y-\omega_{U_{k}}+\varepsilon_{i,t})\prod_{P_{j}\in\Nb(P_{i})}(y-\omega_{U_{k}}+\alpha_{i,j})}
{(y-\omega_{U_{k}})^{m+1}
\prod_{j\in\mathfrak{K}\backslash\{k\}}(y-\omega_{U_{k}}+\omega_{U_j})^{|U_{j}|}\prod_{k\in [r]\backslash I_{\mathfrak{W}}}(\omega_{k}-\psi_{k}) \prod_{s\in[m]\backslash\mathfrak{W}}(y-\omega_{U_{k}}-\psi_{0_{s}})}\Bigg)\nn\\
&=&\frac{(-1)^{m}}{24m}
[x^{m}]\Bigg(\frac{\prod_{t=1}^{l}(x+\varepsilon_{i,t})\prod_{P_{j}\in\Nb(P_{i})}(x+\alpha_{i,j})}
{x^{m}\prod_{k=1}^{r}\omega_{k}}
\cdot\prod_{s=1}^{m}\Big(\frac{1}{x}+\sum_{j\in I_{(s)}}\frac{1}{\omega_{j}}\Big)^{|I_{(s)}|-2}\Bigg)\nn\\
&&+
\frac{(-1)^{m}}{24m}\sum_{k\in\mathfrak{K}}
[y^{|U_{k}|-1}]\Bigg(\frac{\prod_{t=1}^{l}(y-\omega_{U_{k}}-a_{t}\alpha_{i})
\prod_{P_{j}\in\Nb(P_{i})}(y-\omega_{U_{k}}+\alpha_{i,j})}
{(y-\omega_{U_{k}})^{2m+1-|U_{k}|}\prod_{j\in [r]\backslash I_{U_{k}}}\omega_{j}}\nn\\
&&\cdot\prod_{s\in[m]\backslash U_{k}}\Big(\frac{1}{y-\omega_{U_{k}}}+\sum_{j\in I_{(s)}}\frac{1}{\omega_{j}}\Big)^{|I_{(s)}|-2}\Bigg).
\eea
Comparing the last expressions of (\ref{53}) and (\ref{55}), we see
\bea\label{92}
&&\sum_{(I,U,V)\in\widetilde{\pi}^{-1}(I)}\VC_{\Gamma_{i;(I,U,V);\vec{\omega}}}^{0;X}
+\sum_{I\in\mathcal{A}_{m}([r])}\mathrm{Cont}_{I}\Big(
\frac{(-1)^{m}}{24m}\sum_{p=0}^{n+l-1-2m}\VC_{\Gamma_{i;\vec{\omega}}}^{(m,p,n+l-1-2m-p),X}\Big)\nn\\
&=&\frac{(-1)^{m}}{24m}
[x^{m}]\Bigg(\frac{\prod_{t=1}^{l}(x+\varepsilon_{i,t})\prod_{P_{j}\in\Nb(P_{i})}(x+\alpha_{i,j})}
{x^{m}\prod_{k=1}^{r}\omega_{k}}
\cdot\prod_{s=1}^{m}\Big(\frac{1}{x}+\sum_{j\in I_{(s)}}\frac{1}{\omega_{j}}\Big)^{|I_{(s)}|-2}\Bigg).
\eea
Summing over $I\in \mathcal{A}_{m}([r])$ and $m\geq 1$, by proposition \ref{49}, we obtain (\ref{45}).

\subsection{Localized standard versus reduced formula for primary  insertions}
Let $\Gamma_{i;\vec{\omega};J}\in\DS_{J}^{d}$. From the string equation it is easily seen that
\bea\label{96}
\VC_{\Gamma_{i;J;\vec{\omega}}}^{X}
&=&\Big(\sum_{j=1}^{r}\frac{1}{\omega_{j}}\Big)^{|J|}\cdot\Bigg[\prod_{k=1}^{l}\varepsilon_{i,k}\prod_{P_{j}\in\Nb(P_{i})}\alpha_{i,j}\int_{\Mbar_{1,r}}\frac{1}
{\prod_{k=1}^{r}(\omega_{k}-\psi_{k})}\nn\\
&&-[x]\Big(\prod_{k=1}^{l}(x+\varepsilon_{i,k})\prod_{P_{j}\in\Nb(P_{i})}(x+\alpha_{i,j})\Big)\cdot
\frac{1}
{24\prod_{k=1}^{r}\omega_{k}}\Big(\sum_{k=1}^{r}\frac{1}{\omega_{k}}\Big)^{r-1}\Bigg].\nn\\
\eea
Define $\widetilde{\mathcal{A}}_{m}([r],J;\vec{\omega})$ to be the set of 4-tuples $(I,K,U,V)\in \mathcal{P}([r])^{m}\times\mathcal{P}(J)^{m}\times [m]\times[m]$ satisfying\\
(i) Writing $I$ as $I=(I_{(1)},\cdots,I_{(m)})$, then $\bigsqcup_{s=1}^{m}I_{(s)}=[r]$;\\
(ii) Writing $J$ as $K=(K_{(1)},\cdots,K_{(m)})$, then $\bigsqcup_{s=1}^{m}K_{(s)}=J$;\\
(iii) $U\cap V=\emptyset$;\\
(iv) $|I_{(s)}|=1$, $K_{(s)}=\emptyset$ for $s\in U\cup V$, and $|I_{(s)}|+|K_{(s)}|\geq 2$ for $s\not\in U\cup V$;\\
(v) $\omega_{s}$ are equal to each other for $s\in U$, and $\omega_{i}\neq \omega_{s}$ for $s\in U$ and $i\in V$.\\

For $(I,K,U,V)\in \widetilde{\mathcal{A}}_{m}([r],J;\vec{\omega})$, denote the common weight $\omega_{s}$ for $s\in U$ by $\omega_{U}$ and suppose $I_{(s)}=\{i_{s}\}$ for $s\in U\cup V$, $I_{U}=\bigsqcup_{s\in U}I_{(s)}$. There is a natural 1-1 correspondence between $\coprod_{J^{\prime}\subset J}\widetilde{\mathcal{A}}_{m}([r],J-J^{\prime};\vec{\omega})$ and $\pi_{\RDRT}^{-1}(\Gamma_{i;\vec{\omega};J})$; in this correspondence, the subset $J^{\prime}$ corresponds to $\mathfrak{e}^{-1}(v_{0})$ in a refined decorated rooted tree. By (\ref{44}) and (\ref{89}) we have
\bea\label{93}
\VC_{\Gamma_{i;\vec{\omega};J}}^{0;X}=\sum_{J^{\prime}\subset J}\sum_{m\geq 1}\sum_{(I,K,U,V)\in \widetilde{\mathcal{A}}_{m}([r],J-J^{\prime};\vec{\omega})}
\VC_{\Gamma_{i;J^{\prime},(I,K,U,V);\vec{\omega}}}^{0;X},
\eea
where for $J^{\prime}\subset J$ and $(I,K;U,V)\in \widetilde{\mathcal{A}}_{m}([r],J-J^{\prime};\vec{\omega})$,
\bea\label{82}
&&\VC_{\Gamma_{i;J^{\prime},(I,K;U,V);\vec{\omega}}}^{0;X}\nn\\
&=&\frac{1}{m!}\int_{\tM_{1,(m,J^{\prime})}\times \mathbb{P}^{|U|-1}}\Bigg(\frac{\prod_{k=1}^{l}(-\omega_{U}+\varepsilon_{i,k}+H)\prod_{P_{j}\in\Nb(P_{i})}
(-\omega_{U}+\alpha_{i,j}+H)}{\omega_{U}-\tilde{\psi}-H}\nn\\
&&\cdot\prod_{s\in V}\frac{1}{-\omega_{U}+H+\omega_{i_s}}
\prod_{s\in [m]\backslash U\cup V}\int_{\Mbar_{0,I_{(s)}\sqcup K_{(s)}\sqcup \{0_{s}\}}}\frac{1}{(-\omega_{U}-\psi_{0_{s}}+H)\prod_{j\in I_{(s)}}(\omega_{j}-\psi_{j})}\Bigg)\nn\\
&=&\frac{(-1)^{m+|J^{\prime}|+1}m^{|J^{\prime}|}}{24m}
\int_{\mathbb{P}^{|U|-1}}\Bigg(\frac{\prod_{k=1}^{l}(H-\omega_{U}+\varepsilon_{i,k})
\prod_{P_{j}\in\Nb(P_{i})}(H-\omega_{U}+\alpha_{i,j})}
{(H-\omega_{U})^{2m+1+|J^{\prime}|-|U|}\prod_{k\in [r]\backslash I_{U}}\omega_{k}}\nn\\
&&\cdot\prod_{s\in[m]\backslash U}
\Big(\frac{1}{H-\omega_{U}}+\sum_{j\in I_{(s)}}\frac{1}{\omega_{j}}\Big)^{|I_{(s)}|+|K_{(s)}|-2}\Bigg)\nn\\
&=&\frac{(-1)^{m+|J^{\prime}|+1}m^{|J^{\prime}|}}{24m}[y^{|U|-1}]\Bigg(\frac{\prod_{k=1}^{l}(y-\omega_{U}+\varepsilon_{i,k})
\prod_{P_{j}\in\Nb(P_{i})}(y-\omega_{U}+\alpha_{i,j})}
{(y-\omega_{U})^{2m+1+|J^{\prime}|-|U|}\prod_{k\in [r]\backslash I_{U}}\omega_{k}}\nn\\
&&\cdot\prod_{s\in[m]\backslash U}\Big(\frac{1}{y-\omega_{U}}+\sum_{j\in I_{(s)}}\frac{1}{\omega_{j}}\Big)^{|I_{(s)}|+|K_{(s)}|-2}\Bigg).
\eea

There is a canonical projection
\ben
\widetilde{\pi}:\widetilde{\mathcal{A}}_{m}([r],J;\vec{\omega})\rightarrow \mathcal{A}_{m}([r],J)
\een
defined by $\pi(I,K,U,V)=(I,K)$. Let us temporarily fix $J^{\prime}$ and $(I,K)\in \mathcal{A}_{m}([r],J-J^{\prime})$, and  suppose\\
(i) $|I_{s}|=1$ and $K_{(s)}=\emptyset$ for $s\in \mathfrak{W}$ and $|I_{s}|+|K_{(s)}|\geq 2$ for $s\in [m]\backslash \mathfrak{W}$, where $\mathfrak{W}$ is a subset of $[m]$;\\
(ii) $\mathfrak{W}=\bigsqcup_{k\in \mathfrak{K}}U_{k}$ is a partition \emph{according to weight}. Precisely speaking, for every $s\in \mathfrak{W}$, suppose $I_{(s)}=\{i_{s}\}$, then for $i_{a},i_{b}\in I_{\mathfrak{W}}:=\bigsqcup_{s\in W}I_{(s)}$, we have $\omega_{i_a}=\omega_{i_b}$ if and only if $a$ and $b$ belong to the same $U_{k}$ for some $k\in\mathfrak{K}$.\\

We denote $I_{U_{k}}=\bigsqcup_{s\in U_{k}}I_{(s)}$ for $k\in \mathfrak{K}$, and  $I_{\mathfrak{W}}=\sum_{k\in\mathfrak{K}}I_{U_{k}}=\bigsqcup_{s\in W}I_{(s)}$. Similar to (\ref{56})-(\ref{55}), we have
\bea\label{94}
&&\frac{(-1)^{m+|J^{\prime}|}m^{|J^{\prime}|}(m-1)!}{24m!}
\sum_{p=0}^{n+l-1-2m}\VC_{\Gamma_{i;\vec{\omega};J},J-J^{\prime}}^{(m,p,n+l-1-2m-|J^{\prime}|-p),X}\nn\\
&=&\sum_{(I,K)\in\mathcal{A}_{m}([r],J-J^{\prime})}\mathrm{Cont}_{(I,K)}\Big(
\frac{(-1)^{m+|J^{\prime}|}m^{|J^{\prime}|}}{24m}
\sum_{p=0}^{n+l-1-2m-|J^{\prime}|}\VC_{\Gamma_{i;\vec{\omega};J},J-J^{\prime}}^{(m,p,n+l-1-2m-|J^{\prime}|-p),X}\Big),\nn\\
\eea
where
\bea\label{81}
&&\mathrm{Cont}_{(I,K)}\Big(
\frac{(-1)^{m+|J^{\prime}|}m^{|J^{\prime}|}}{24m}
\sum_{p=0}^{n+l-1-2m-|J^{\prime}|}\VC_{\Gamma_{i;\vec{\omega};J},J-J^{\prime}}^{(m,p,n+l-1-2m-|J^{\prime}|-p),X}\Big)\nn\\
&=&\frac{(-1)^{m+|J^{\prime}|}m^{|J^{\prime}|}}{24m}\sum_{p=0}^{n+l-1-2m-|J^{\prime}|}
\Bigg[\sum_{\stackrel{\sum_{s\in \mathfrak{W}}p_{i_{s}}=p}{\forall p_{i_{s}}\geq 0}}\prod_{s\in \mathfrak{W}}(-\omega_{i_{s}})^{p_{i_{s}}}\nn\\
&&\cdot
\int_{\prod_{s\in [m]\backslash\mathfrak{W}}\Mbar_{0,I_{(s)}\sqcup K_{(s)}\sqcup \{0_{s}\}}}
[x^{n+l-1-2m-|J^{\prime}|-p}]\Bigg(\frac{\prod_{t=1}^{l}(1+\varepsilon_{i,t}x)\prod_{P_{j}\in\Nb(P_{i})}(1+\alpha_{i,j}x)}{\prod_{k\in [r]\backslash I_{\mathfrak{W}}}(\omega_{k}-\psi_{k})\cdot \prod_{s\in[m]\backslash\mathfrak{W}}(1-x\psi_{0_{s}})}\Bigg)\Bigg]\nn\\
&=&\frac{(-1)^{m+|J^{\prime}|}m^{|J^{\prime}|}}{24m}
[x^{m+|J^{\prime}|}]\Bigg(\frac{\prod_{t=1}^{l}(x+\varepsilon_{i,t})\prod_{P_{j}\in\Nb(P_{i})}(x+\alpha_{i,j})}
{x^{m}\prod_{k=1}^{r}\omega_{k}}\nn\\
&&\cdot\prod_{s=1}^{m}\Big(\frac{1}{x}+\sum_{j\in I_{(s)}}\frac{1}{\omega_{j}}\Big)^{|I_{(s)}|+|K_{(s)}|-2}\Bigg)\nn\\
&&+
\frac{(-1)^{m+|J^{\prime}|}m^{|J^{\prime}|}}{24m}\sum_{k\in\mathfrak{K}}
[y^{|U_{k}|-1}]\Bigg(\frac{\prod_{t=1}^{l}(y-\omega_{U_{k}}+\varepsilon_{i,t})
\prod_{P_{j}\in\Nb(P_{i})}(y-\omega_{U_{k}}+\alpha_{i,j})}
{(y-\omega_{U_{k}})^{2m+1+|J^{\prime}|-|U_{k}|}\prod_{j\in [r]\backslash I_{U_{k}}}\omega_{j}}\nn\\
&&\cdot\prod_{s\in[m]\backslash U_{k}}\Big(\frac{1}{y-\omega_{U_{k}}}+\sum_{j\in I_{(s)}}\frac{1}{\omega_{j}}\Big)^{|I_{(s)}|+|K_{(s)}|-2}\Bigg).
\eea
So
\bea\label{95}
&&\sum_{(I,K,U,V)\in\widetilde{\pi}^{-1}(I,K)}\VC_{\Gamma_{i;(I,K,U,V);\vec{\omega}}}^{0;X}\nn\\
&&+\sum_{(I,K)\in\mathcal{A}_{m}([r],J-J^{\prime})}\mathrm{Cont}_{(I,K)}\Big(
\frac{(-1)^{m+|J^{\prime}|}m^{|J^{\prime}|}}{24m}
\sum_{p=0}^{n+l-1-2m-|J^{\prime}|}\VC_{\Gamma_{i;\vec{\omega};J},J-J^{\prime}}^{(m,p,n+l-1-2m-|J^{\prime}|-p),X}\Big)\nn\\
&=&\frac{(-1)^{m+|J^{\prime}|}m^{|J^{\prime}|}}{24m}
[x^{m+|J^{\prime}|}]\Bigg(\frac{\prod_{t=1}^{l}(x+\varepsilon_{i,t})\prod_{P_{j}\in\Nb(P_{i})}(x+\alpha_{i,j})}
{x^{m}\prod_{k=1}^{r}\omega_{k}}
\cdot\prod_{s=1}^{m}\Big(\frac{1}{x}+\sum_{j\in I_{(s)}}\frac{1}{\omega_{j}}\Big)^{|I_{(s)}|-2}\Bigg),\nn\\
\eea
Fixing $I\in \mathcal{A}_{m}([r])$, summing (\ref{95}) over $J^{\prime}$ and $K\in\mathcal{A}_{m}^{0}(J-J^{\prime})$ we obtain
\bea\label{83}
&&\sum_{J^{\prime}\subset J}\frac{(-1)^{m+|J^{\prime}|}m^{|J^{\prime}|}}{24m}\sum_{I\in \mathcal{A}_{m}([r]) }[x^{m+|J^{\prime}|}]\Bigg\{\frac{\prod_{k=1}^{l}(x+\varepsilon_{i,k})\prod_{P_{j}\in\Nb(P_{i})}(x+\alpha_{i,j})}
{x^{m}\prod_{k=1}^{r}\omega_{k}}\nn\\
&&\cdot\Big(\frac{m}{x}+\sum_{j=1}^{r}\frac{1}{\omega_{j}}\Big)^{|J|-|J^{\prime}|}\prod_{s=1}^{m}\Big(\frac{1}{x}+\sum_{j\in I_{(s)}}\frac{1}{\omega_{j}}\Big)^{|I_{(s)}|-2}\Bigg\}\nn\\
&=&\frac{(-1)^{m}}{24m}[x^{m}]\Bigg\{\sum_{I\in \mathcal{A}_{m}([r])} \frac{\prod_{k=1}^{l}(x+\varepsilon_{i,k})\prod_{P_{j}\in\Nb(P_{i})}(x+\alpha_{i,j})}
{x^{m}\prod_{k=1}^{r}\omega_{k}}\prod_{s=1}^{m}\Big(\frac{1}{x}+\sum_{j\in I_{(s)}}\frac{1}{\omega_{j}}\Big)^{|I_{(s)}|-2}\nn\\
&&\cdot\Bigg(\sum_{J^{\prime}\subset J}(-1)^{|J^{\prime}|}
\Big(\frac{m}{x}\Big)^{|J^{\prime}|}\Big(\frac{m}{x}+\sum_{j=1}^{r}\frac{1}{\omega_{j}}\Big)^{|J|-|J^{\prime}|}
\Bigg)\Bigg\}\nn\\
&=&\Big(\sum_{j=1}^{r}\frac{1}{\omega_{j}}\Big)^{|J|}\frac{(-1)^{m}}{24m}[x^{m}]\Bigg\{\sum_{I\in \mathcal{A}_{m}([r])} \frac{\prod_{k=1}^{l}(x+\varepsilon_{i,k})\prod_{P_{j}\in\Nb(P_{i})}(x+\alpha_{i,j})}
{x^{m}\prod_{k=1}^{r}\omega_{k}}\nn\\
&&\prod_{s=1}^{m}\Big(\frac{1}{x}+\sum_{j\in I_{(s)}}\frac{1}{\omega_{j}}\Big)^{|I_{(s)}|-2}\Bigg\}.
\eea
Comparing with (\ref{96}), by proposition \ref{49}, we obtain (\ref{45}).

\subsection{An alternative form of LSvR}
In this section, we give another form of LSvR, which corresponds to \cite[theorem 1B]{ZingerSvR}.
For this, we need to define the classes $\tilde{\eta}_{p}$ over $\Mbar_{(m,J)}(\mathbb{P}^{n-1},d)$. Let us first recall the definition of $\tilde{\psi}$-classes (\cite{ZingerSvR}). For $j\in J$, let $\tilde{\psi}_{j}\in H^{2}\big(\Mbar_{g,J}(Y,d)\big)$ be the cohomology class defined by pulling back the $\psi$-class on
$\Mbar_{g,1}(Y,d)$ via the forgetting map
\ben
\Mbar_{g,J}(Y,d)\rightarrow \Mbar_{g,1}(Y,d)
\een
which drops the marked points except the $j$-th one and then contracting the unstable components. Let
\bea\label{98}
\langle \mu_{1}\tilde{\psi}_{1}^{c_{1}},\cdots,\mu_{k}\tilde{\psi}_{k}^{c_{k}}\rangle_{g,k,d}^{Y}=
 \bigwedge_{j=1}^{k}\mu_{j}\tilde{\psi}_{j}^{c_{j}}\cap[\Mbar_{g,k}(Y,d)]^{\vir}.
\eea
The invariants of this type has the advantage that the divisor equation  takes a  simple form.
\begin{lemma}\label{191}
For $\gamma\in H^{2}(Y) $, we have
\bea\label{287}
\langle \gamma, \mu_{1}\tilde{\psi}_{1}^{c_{1}},\cdots,\mu_{k}\tilde{\psi}_{k}^{c_{k}}\rangle_{g,k+1,d}^{Y}
= (\gamma\cap d)\langle \mu_{1}\tilde{\psi}_{1}^{c_{1}},\cdots,\mu_{k}\tilde{\psi}_{k}^{c_{k}}\rangle_{g,k,d}^{Y}.
\eea
\end{lemma}
\begin{proof} The proof is similar to the usual one, see for example of \cite[page 264]{Manin}; since $\tilde{\psi}_{j}$ are defined by pulling back from the moduli spaces with less marked points, there is no additional terms with insertions of the form $\langle \cdots, (\gamma\wedge\mu_{j})\psi^{c_{j}-1},\cdots\rangle_{g,k,d}^{Y}$.
\end{proof}

We define $\tilde{\eta}_{p}\in H^{2}\big(\Mbar_{(m,J)}(Y,d)\big)$ by the generating function (restricted to each component $\Mbar_{(m;J_{1},\cdots,J_{m})}(Y,\mathbf{d})$)
\bea\label{99}
\sum_{p=0}^{\infty}x^{p}\tilde{\eta}_{p}=\prod_{s\in[m]}\frac{1}{1-x\pi_{i}^{*}\tilde{\psi}_{0_{s}}}.
\eea
 Let
\bea\label{100}
\langle \tilde{\eta}_{p}\mu_{0};\mu_{1},\cdots,\mu_{|J|}\rangle_{(m,J,d)}^{Y}=
\frac{1}{m!}\tilde{\eta}_{p}\wedge\ev_{0}^{*}\big(\mu_{0}\big)\bigwedge_{j\in J}\ev_{j}^{*}(\mu_{j})\cap [\Mbar_{(m,J)}(Y,d)]^{\vir}.
\eea

From (\ref{287}) it is straightforward to deduce
\begin{lemma}\label{192}
For $\gamma\in H^{2}(Y) $, we have
\bea\label{101}
\langle \tilde{\eta}_{p}\mu_{0};\gamma,\mu_{1},\cdots,\mu_{|J|}\rangle_{(m,J\sqcup\{1\},d)}^{Y}
= (\gamma\cap d)\langle \tilde{\eta}_{p}\mu_{0};\mu_{1},\cdots,\mu_{|J|}\rangle_{(m,J,d)}^{Y}.
\eea
\end{lemma}
For $X=\mathrm{Tot}(E\rightarrow Y)$ where $E$ is a concave vector bundle over $Y$, let
\bea\label{102}
&&\langle \tilde{\eta}_{p}\mu_{0};\mu_{1},\cdots,\mu_{|J|}\rangle_{(m,J,d)}^{X}\nn\\
&:= &\frac{1}{m!}e(E)^{m-1}\bigwedge_{i=1}^{m}\pi_{i}^{*}(\mathcal{U}_{0})\wedge \tilde{\eta}_{p}\wedge\ev_{0}^{*}\big(\mu_{0}\big)\bigwedge_{j\in J}\ev_{j}^{*}(\mu_{j})\cap [\Mbar_{(m,J)}(Y,d)]^{\vir}.
\eea
The divisor equation (\ref{101}) still holds.\\

For finite sets $I$ and $J$ with $I\cap J=\emptyset$ and $|I|\geq 3$, let
\bea\label{103}
\pi_{I}:\Mbar_{0,I\sqcup J}\rightarrow \Mbar_{0,I}
\eea
be the map which drops the marked points labelled by elements of $J$. By the proof the usual string equation, we have
\bea\label{104}
&&\int_{\Mbar_{0,I\sqcup J}}\frac{1}{\prod_{i\in I}(w_{i}-\pi_{I}^{*}\psi_{i})\prod_{j\in J}(w_{j}-\psi_{j})}\nn\\
&=&\Big(\sum_{j\in J}\frac{1}{w_{j}}\Big)^{|J|}\cdot \frac{1}{\prod_{i}w_{i}\prod_{j}w_{j}}\Big(\sum_{i\in I}\frac{1}{w_{i}}+\sum_{j\in J}\frac{1}{w_{j}}\Big)^{|I|-3}.
\eea
In the localization contribution, we formally extend this identity  to $|I|\geq 1$.\\

Now let $Y$ be an  algebraic GKM manifold and $X=\mathrm{Tot}(E\rightarrow Y)$ is the total space of a concave equivariant vector bundle $E$.
Let $\mu_{1},\cdots,\mu_{|J|}\in H_{\mathbb{T}}^{*}(Y)$. For $\Gamma^{\mathfrak{c}}=(\Gamma,I,K)\in\CDRT_{J}^{d}$, the localization contribution of $\Gamma^{\mathfrak{c}}$ to $\langle\tilde{\eta}_{p}\mathbf{c}_{q}(TX);\mu_{1},\cdots,\mu_{|J|}\rangle_{(m,J,d)}^{X}$ can be written as
 \begin{multline}\label{105}
 \Cont_{\Gamma^{\mathfrak{c}}}(\langle\tilde{\eta}_{p}\mathbf{c}_{q}(TX);\mu_{1},\cdots,\mu_{|J|}\rangle_{(m,J,d)}^{X})
 =\frac{1}{m!}\frac{1}{|\Aut(\Gamma^{\mathfrak{c}})|}\\
 \cdot\prod_{v\in \Ver} \Cont_{\Gamma^{\mathfrak{c}};v}\big(\langle\tilde{\eta}_{p}\mathbf{c}_{q}(TX);\mu_{1},\cdots,\mu_{|J|}\rangle_{(m,J,d)}^{X}\big)\\
\cdot \prod_{e\in \Edg} \Cont_{\Gamma^{\mathfrak{c}};e}\big(\langle\tilde{\eta}_{p}\mathbf{c}_{q}(TX);\mu_{1},\cdots,\mu_{|J|}\rangle_{(m,J,d)}^{X}\big).
 \end{multline}
For an edge $e$ and a vertex $v\in \Ver\backslash\{v_{0}\}$, the contribution is still the same as (\ref{11})  respectively. Suppose $\mathfrak{m}(v_{0})=i$, then the contribution of $v_{0}$ is
\begin{multline}\label{106}
\Cont_{\Gamma^{\mathfrak{c}};v_{0}}\big(\langle\tilde{\eta}_{p}\mathbf{c}_{q}(TX);\mu_{1},\cdots,\mu_{|J|}\rangle_{(m,J,d)}^{X}\big)
=\prod_{j\in \mathfrak{e}^{-1}(v_{0})}\mu_{j}\big|_{P_{i}}\\
\Big(\prod_{k=1}^{l}\varepsilon_{i,k}\prod_{P_{j}\in\Nb(P_{i})}\alpha_{i,j}\Big)^{|\Edg(v_{0})|-1}
\cdot[x^{q}]\Big(\prod_{k=1}^{l}(1+\varepsilon_{i,k}x)\prod_{j\in [n]\backslash\{i\}}(1+\alpha_{i,j}x)\Big)\\
\cdot\int_{\Mbar_{0,I_{(1)}\sqcup K_{(1)}\sqcup \{0_{1}\}}\times\cdots\times\Mbar_{0,I_{(m)}\sqcup K_{(m)}\sqcup \{0_{m}\}}}\\
[x^{p}]\Bigg(\frac{1}{\prod_{e\in\Edg(v_{0})}(\frac{\alpha_{v_{0},e}}{d(e)}-\psi_{(v_{0},e)}) \prod_{s=1}^{m}(1-x\pi_{I_{(s)\sqcup \{0_{s}\}}}^{*}\psi_{0_{s}})}\Bigg).
\end{multline}

Now we write the invariant $\langle\tilde{\eta}_{p}\mathbf{c}_{q}(TX);\mu_{1},\cdots,\mu_{|J|}\rangle_{(m,J,d)}^{X}$ as summing over the $\DRT_{J}^{d}$, i.e.,
\bea\label{107}
&&\langle\tilde{\eta}_{p}\mathbf{c}_{q}(TX);\mu_{1},\cdots,\mu_{|J|}\rangle_{(m,J,d)}^{X}\nn\\
&&=\sum_{\Gamma\in \DRT_{J}^{d}}\Cont_{\Gamma}\big(\langle\tilde{\eta}_{p}\mathbf{c}_{q}(TX);\mu_{1},\cdots,\mu_{|J|}\rangle_{(m,J,d)}^{X}\big),
\eea
where
\bea\label{108}
&&\Cont_{\Gamma}\big(\langle\tilde{\eta}_{p}\mathbf{c}_{q}(TX);\mu_{1},\cdots,\mu_{|J|}\rangle_{(m,J,d)}^{X}\big)\nn\\
&=&\frac{1}{m!}\frac{1}{|\Aut(\Gamma)|}
\sum_{\Gamma^{\mathfrak{c}}\in\pi_{\CDRT}^{-1}(\Gamma)}\prod_{v\in \Ver} \Cont_{\Gamma^{\mathfrak{c}};v}\big(\langle\tilde{\eta}_{p}\mathbf{c}_{q}(TX);\mu_{1},\cdots,\mu_{|J|}\rangle_{(m,J,d)}^{X}\big)\nn\\
&&\cdot\prod_{e\in \Edg} \Cont_{\Gamma^{\mathfrak{c}};e}\big(\langle\tilde{\eta}_{p}\mathbf{c}_{q}(TX);\mu_{1},\cdots,\mu_{|J|}\rangle_{(m,J,d)}^{X}\big).
 \eea

Now we state another form of the LSvR.
\begin{theorem}\label{109}
Let $\mu_{1},\cdots,\mu_{|J|}\in H_{\mathbb{T}}^{*}(Y)$.
For every decorated rooted tree $\Gamma\in \DRT_{J}^{d}$, we have the \emph{LSvR}
\bea\label{110}
&&\Cont_{\Gamma}(\langle\mu_{1},\cdots,\mu_{|J|}\rangle_{1,J,d}^{X})
=\Cont_{\Gamma}(\langle\mu_{1},\cdots,\mu_{|J|}\rangle_{1,J,d}^{0;X})\nn\\
&&+\frac{1}{24}\sum_{m\geq 1}(-1)^{m}(m-1)!
\sum_{p=0}^{n+l-1-2m}\Cont_{\Gamma}\big(\langle \tilde{\eta}_{p}c_{n+l-1-2m-p}(TX);\mu_{1},\cdots,\mu_{|J|}\rangle_{(m,J,d)}^{X}\big).\nn\\
\eea
In particular, when $X$ is a local Calabi-Yau space, for every $\Gamma\in\DRT_{\emptyset}^{d}$ we have
\begin{multline}\label{111}
\Cont_{\Gamma}(N_{1,d}^{X})=\Cont_{\Gamma}(N_{1,d}^{0;X})+\frac{1}{24}\sum_{m\geq 1}\Big[(-1)^{m}(m-1)!\\
\sum_{p=0}^{n+l-1-2m}\Cont_{\Gamma}\big(\langle\tilde{\eta}_{p}c_{n+l-1-2m-p}(TX);\rangle_{(m,\emptyset,d)}^{X}\big)\Big].
\end{multline}
\end{theorem}
\emph{Sketch of the proof}: By theorem \ref{32}, it suffices to show
\bea\label{116}
&&\sum_{p=0}^{n+l-1-2m}\Cont_{\Gamma}\big(\langle\tilde{\eta}_{p}c_{n+l-1-2m-p}(TX);\rangle_{(m,\emptyset,d)}^{X}\big)\nn\\
&=&\sum_{J^{\prime}\subset J}(-1)^{|J^{\prime}|}m^{|J^{\prime}|}
\sum_{p=0}^{n+l-1-|J^{\prime}|-2m}
\Cont_{\Gamma}\Big(\langle \eta_{p}c_{n+l-1-|J^{\prime}|-2m-p}(TX);\mu_{1},\cdots,\mu_{|J|}\rangle_{(m,J-J^{\prime},d)}^{X}\Big).\nn\\
\eea
As before, we need only to show (\ref{116}) for decorated stars. Let $\Gamma=\Gamma_{i;\vec{\omega};J}$. Define
\bea\label{117}
&&\overline{\overline{\Cont}}_{\Gamma}^{(m,p,q),X}:=[x^{q}]\Big(\prod_{k=1}^{l}(1+\varepsilon_{i,k}x)\prod_{j\in [n]\backslash\{i\}}(1+\alpha_{i,j}x)\Big)\nn\\
&&\cdot\prod_{\Gamma^{\mathfrak{c}\in\pi_{CDRT}^{-1}(\Gamma)}}\int_{\Mbar_{0,I_{(1)}\sqcup K_{(1)}\sqcup \{0_{1}\}}\times\cdots\times\Mbar_{0,I_{(m)}\sqcup K_{(m)}\sqcup \{0_{m}\}}}\nn\\
&&[x^{p}]\Bigg(\frac{1}{\prod_{e\in\Edg(v_{0})}(\frac{\alpha_{v_{0},e}}{d(e)}-\psi_{(v_{0},e)}) \prod_{s=1}^{m}(1-x\pi_{I_{(s)\sqcup \{0_{s}\}}}^{*}\psi_{0_{s}})}\Bigg).
\eea\\
By eliminating the common factors as in section 3.1, it suffices to show
\bea\label{118}
&&\sum_{p=0}^{n+l-1-2m}\overline{\overline{\Cont}}_{\Gamma}^{(m,p,n+l-1-2m-p),X}\nn\\
&=&
\sum_{J^{\prime}\subset J}(-1)^{|J^{\prime}|}m^{|J^{\prime}|}
\sum_{p=0}^{n+l-1-|J^{\prime}|-2m}\VC_{\Gamma,J-J^{\prime}}^{(m,p,n+l-1-|J^{\prime}|-2m-p),X}.
\eea
But by (\ref{104}) and
\bea\label{97}
&&\sum_{|J^{\prime}|\subset J}\sum_{K\in \mathcal{A}_{m}^{0}(J-J^{\prime})}\mathrm{Cont}_{(I,K)}
\Big((-1)^{m+|J^{\prime}|}m^{|J^{\prime}|}
\sum_{p=0}^{n+l-1-2m-|J^{\prime}|}\VC_{\Gamma_{i;J-J^{\prime};\vec{\omega}}}^{(m,p,n+l-1-2m-|J^{\prime}|-p),X}\Big)\nn\\
&=&\sum_{J^{\prime}\subset J}\sum_{K\in \mathcal{A}_{m}^{0}(J-J^{\prime})}(-1)^{|J^{\prime}|}m^{|J^{\prime}|}\nn\\
&&[x^{n+l-1-2m-|J^{\prime}|}]\Bigg(\frac{\prod_{t=1}^{l}(1+\varepsilon_{i,t}x)\prod_{P_{j}\in\Nb(P_{i})}(1+\alpha_{i,j}x)}
{\prod_{k=1}^{r}\omega_{k}}
\cdot\prod_{s=1}^{m}\Big(x+\sum_{j\in I_{(s)}}\frac{1}{\omega_{j}}\Big)^{|I_{(s)}|+|K_{(s)}|-2}\Bigg)\nn\\
&=&[x^{n+l-1-2m}]\Bigg\{\frac{\prod_{t=1}^{l}(1+\varepsilon_{i,t}x)\prod_{P_{j}\in\Nb(P_{i})}(1+\alpha_{i,j}x)}
{\prod_{k=1}^{r}\omega_{k}}\nn\\
&&\cdot\sum_{J^{\prime}\subset J}\sum_{K\in \mathcal{A}_{m}^{0}(J-J^{\prime})}(-xm)^{|J^{\prime}|}
\cdot\prod_{s=1}^{m}\Big(x+\sum_{j\in I_{(s)}}\frac{1}{\omega_{j}}\Big)^{|I_{(s)}|+|K_{(s)}|-2}\Bigg\}\nn\\
&=&\Big(\sum_{j=1}^{r}\frac{1}{\omega_{j}}\Big)^{|J|}[x^{n+l-1-2m}]\Bigg\{\frac{\prod_{t=1}^{l}(1+\varepsilon_{i,t}x)\prod_{P_{j}\in\Nb(P_{i})}(1+\alpha_{i,j}x)}
{\prod_{k=1}^{r}\omega_{k}}\prod_{s=1}^{m}\Big(x+\sum_{j\in I_{(s)}}\frac{1}{\omega_{j}}\Big)^{|I_{(s)}|-2}\Bigg\},\nn\\
\eea
it is easy to deduce (\ref{118}).\pqed\\

\begin{corollary}\label{112}
\bea\label{113}
&&\langle\mu_{1},\cdots,\mu_{|J|}\rangle_{1,J,d}^{X}
=\langle\mu_{1},\cdots,\mu_{|J|}\rangle_{1,J,d}^{0;X}\nn\\
&&+\frac{1}{24}\sum_{m\geq 1}(-1)^{m}(m-1)!
\sum_{p=0}^{n+l-1-2m}\langle \tilde{\eta}_{p}\mathbf{c}_{n+l-1-2m-p}(TX);\mu_{1},\cdots,\mu_{|J|}\rangle_{(m,J,d)}^{X}.\nn\\
\eea
In particular, when $X$ is a local Calabi-Yau space,
\bea\label{114}
N_{1,d}^{0;X}=N_{1,d}^{0;X}+\frac{1}{24}\sum_{m\geq 1}(-1)^{m}(m-1)!
\sum_{p=0}^{n+l-1-2m}\langle\tilde{\eta}_{p}c_{n+l-1-2m-p}(TX);\rangle_{(m,\emptyset,d)}^{X}.
\eea
\end{corollary}

\subsection{Modifying the virtual localization in genus one for Calabi-Yau complete intersections}
Let $E$ be a direct sum of  equivariant \emph{ample} line bundles over an  algebraic GKM manifold $Y$. The weights of $E$ at the fixed point $P_{i}$ are still denoted by $\varepsilon_{i,1},\cdots,\varepsilon_{i,l}$. We  assume that for every $1\leq k\leq l$, and every $P_{j}\in\Nb(P_{i})$, $\varepsilon_{i,k}$ is linear independent to $\alpha_{i,j}$. For $X=\mathrm{Tot}(E\rightarrow Y)$, it is natural to define the invariants
\bea\label{175}
\langle \eta_{p}\mu_{0};\mu_{1},\cdots,\mu_{|J|}\rangle_{(m,J-J^{\prime},d)}^{X}=\langle \eta_{p}\mu_{0};\mu_{1},\cdots,\mu_{|J|}\rangle_{(m,J-J^{\prime},d)}^{W},
\eea
where $W$ is a generic smooth section of $E$. For $(\Gamma,I,K)\in m\CDRT_{J}^{d}(Y)$, the localization contribution of an edge or of a vertex other than the root is as the usual genus zero Gromov-Witten of complete intersections. The contribution of the root $v_{0}$ is replacing (\ref{17}) by
\bea\label{176}
&&\Cont_{\Gamma^{\mathfrak{c}};v_{0}}\big(\langle\eta_{p}\mathbf{c}_{q}(TX);\mu_{1},\cdots,\mu_{|J|}\rangle_{(m,J^{\prime},d)}^{X}\big)
=\prod_{j\in \mathfrak{e}^{-1}(v_{0})}\mu_{j}\big|_{P_{i}}\nn\\
&&\Big(\frac{\prod_{P_{j}\in\Nb(P_{i})}\alpha_{i,j}}{\prod_{k=1}^{l}\varepsilon_{i,k}}\Big)^{|\Edg(v_{0})|-1}
\cdot[x^{q}]\Big(\frac{\prod_{j\in [n]\backslash\{i\}}(1+\alpha_{i,j}x)}{\prod_{k=1}^{l}(1+\varepsilon_{i,k}x)}\Big)\nn\\
&&\cdot\int_{\Mbar_{0,I_{(1)}\sqcup K_{(1)}\sqcup \{0_{1}\}}\times\cdots\times\Mbar_{0,I_{(m)}\sqcup K_{(m)}\sqcup \{0_{m}\}}}
[x^{p}]\Bigg(\frac{1}{\prod_{e\in\Edg(v_{0})}(\frac{\alpha_{v_{0},e}}{d(e)}-\psi_{(v_{0},e)}) \prod_{s=1}^{m}(1-x\psi_{0_{s}})}\Bigg).\nn\\
\eea
We define the localization contribution of $\Gamma\in\DRT_{J}^{d}$ to $\langle\eta_{p}\mathbf{c}_{q}(TX);\mu_{1},\cdots,\mu_{|J|}\rangle_{(m,J^{\prime},d)}^{X}$ as in section 2.2.\\

To define the formal reduced genus one Gromov-Witten invariants of $X$, it is natural to replace (\ref{24}) by
\bea\label{177}
\mathbf{e}(\mathcal{U}_{1}^{\prime})= \mathbf{e}(\mathcal{U}_{1})\cdot \frac{1}{\mathbf{e} (E)\big|_{P_{\mathfrak{m}(v_{0})}}}
:= \prod_{e\in\Edg(v_{0})}\pi_{e}^{*}\mathbf{e}(\mathcal{U}_{0}^{\prime})
\cdot \frac{1}{\mathbf{e}(L_{\widetilde{\Gamma}}^{*}\otimes E_{\mu(v_{0})}\otimes \gamma^{*})},
\eea
where
\bea
\mathbf{e}(\mathcal{U}_{0}^{\prime})=\mathbf{e}(\pi_{*}f^{*}E)/\mathbf{e}(E).
\eea
Thus when $E=\bigoplus_{k=1}^{l}\mathcal{O}(a_{k})\rightarrow \mathbb{P}^{n-1}$, where $a_{k}>0$ and $\sum_{k=1}^{l}a_{k}=n$, the \emph{formal reduced genus one Gromov-Witten invariants} of $X$ is no other than the reduced  genus one Gromov-Witten invariants of $W$, where $W$  is a complete intersection of $\mathbb{P}^{n-1}$ with multiple degree $(a_{1},\cdots,a_{l})$, by the localization contributions given in \cite{Zinger1} and \cite{Popa1}. \\

On the other hand, we cannot compute the genus one Gromov-Witten invariants of $W$ by the virtual localization, so it does not make sense to say whether the LSvR holds for $W$ (or for $X$). But since the \emph{global} SvR holds (\cite{ZingerSvR}), it is  reasonable to formally write
the  genus one Gromov-Witten invariants of $W$ as summing over $\DOL$ and $\DRT$, \emph{such that
 the LSvR holds}. For simplicity, we consider the case $E=\mathcal{O}_{\mathbb{P}^{n-1}}(n)$. Then $\alpha_{i,j}=\alpha_{i}-\alpha_{j}$, and we can linearize $E$ such that the weight of $E$ at $P_{i}$ is $n\alpha_{i}$. The localization contribution of $\Gamma\in\DOL$ to $N_{1,d}^{W}$ is as in the usual virtual localization. For $\Gamma\in\DRT$, the problem arise only at the root $v_{0}$. As in the concave cases, we can put aside the common factors, and focus on $\VC_{\Gamma}^{W}$, where $\Gamma$ is a decorated star. Let $\Gamma_{i;\vec{\omega}}$ be a simply decorated star, where $\omega_{s}=\alpha_{i}-\alpha_{j_{s}}$ for $1\leq s\leq r$.  When $\Gamma$ is a simply decorated star, we have $\VC_{\Gamma_{i;\vec{\omega}}}^{0;X}=0$, so we \emph{define}
\bea\label{178}
\VC_{\Gamma_{i;\vec{\omega}}}^{W}&:=&\sum_{m\geq 1}\frac{(-1)^{m}}{24m}\sum_{p=0}^{n-2-2m}\VC_{\Gamma_{i;\vec{\omega}}}^{(m,p,n-2-2m-p),W}\nn\\
&=&\frac{(-1)^{m}}{24m}\sum_{I\in \mathcal{A}_{m}([r]) }
\int_{\Mbar_{0,I_{(1)}\sqcup \{0_{1}\}}\times\cdots\times\Mbar_{0,I_{(m)}\sqcup \{0_{m}\}}}\nn\\
&&[x^{n-2-2m}]\Bigg(\frac{\prod_{j\in[n]\backslash\{i\}}(1+x(\alpha_{i}-\alpha_{j}))}
{(1+n\alpha_{i}x)\prod_{k=1}^{r}(\alpha_{i}-\alpha_{j_{k}}-\psi_{k})\cdot \prod_{s=1}^{m}(1-x\psi_{0_{s}})}\Bigg).
\eea
In contrast to the concave cases, the rational function of $x$ in the big brackets has a pole at $-\frac{1}{n\alpha_{i}}$. It is not hard to see
\bea\label{179}
&&\VC_{\Gamma_{i;\vec{\omega}}}^{W}\nn\\
&=&\sum_{m\geq 1}\Bigg[\frac{(-1)^{m}}{24m}\sum_{I\in \mathcal{A}_{m}([r]) }[x^{m}]\Bigg(\frac{\prod_{j\in[n]\backslash\{i\}}(x+\alpha_{i}-\alpha_{j})}
{(x+n\alpha_{i})x^{m}\prod_{k=1}^{r}(\alpha_{i}-\alpha_{j_{k}})}
\cdot\prod_{s=1}^{m}\Big(\frac{1}{x}+\sum_{j\in I_{(s)}}\frac{1}{\alpha_{i}-\alpha_{j}}\Big)^{|I_{(s)}|-2}\Bigg)\nn\\
&&+\frac{\prod_{j\in[n]\backslash\{i\}}(-n\alpha_{i}+\alpha_{i}-\alpha_{j})}
{(-n\alpha_{i})^{r+1}\prod_{k=1}^{r}(\alpha_{i}-\alpha_{j_{k}})}\cdot
\frac{(-1)^{m}}{24m}\sum_{I\in \mathcal{A}_{m}([r]) }
\prod_{s=1}^{m}\Big(1-\sum_{j\in I_{(s)}}\frac{n\alpha_{i}}{\alpha_{i}-\alpha_{j}}\Big)^{|I_{(s)}|-2}\Bigg].
\eea
By the lemma \ref{181} we have
\bea\label{180}
&&\sum_{m\geq 1}\frac{(-1)^{m}}{24m}\sum_{I\in \mathcal{A}_{m}([r]) }[x^{m}]\Bigg(\frac{\prod_{j\in[n]\backslash\{i\}}(x+\alpha_{i}-\alpha_{j})}
{(x+n\alpha_{i})x^{m}\prod_{k=1}^{r}(\alpha_{i}-\alpha_{j_{k}})}
\cdot\prod_{s=1}^{m}\Big(\frac{1}{x}+\sum_{j\in I_{(s)}}\frac{1}{\alpha_{i}-\alpha_{j}}\Big)^{|I_{(s)}|-2}\Bigg)\nn\\
&=&\frac{\prod_{j\in[n]\backslash\{i\}}(\alpha_{i}-\alpha_{j})}{n\alpha_{i}}\int_{\Mbar_{1,r}}\frac{1}
{\prod_{k=1}^{r}(\omega_{k}-\psi_{k})}\nn\\
&&+[x]\Big(\frac{\prod_{j\in[n]\backslash\{i\}}(\alpha_{i}-\alpha_{j}-x)}{n\alpha_{i}-x}\Big)\cdot
\frac{1}
{24\prod_{k=1}^{r}\omega_{k}}\Big(\sum_{k=1}^{r}\frac{1}{\omega_{k}}\Big)^{r-1}.
\eea
Thus  we have
\bea\label{182}
\VC_{\Gamma_{i;\vec{\omega}}}^{W}&=&
\int_{\Mbar_{1,r}}\frac{\prod_{j\in[n]\backslash\{i\}}
\Lambda_{1}^{\vee}(\alpha_{i}-\alpha_{j})}
{\Lambda_{1}^{\vee}(n\alpha_{i})\prod_{k=1}^{r}(\omega_{k}-\psi_{k})}\nn\\
&&+\frac{\prod_{j\in[n]\backslash\{i\}}(-n\alpha_{i}+\alpha_{i}-\alpha_{j})}
{(-n\alpha_{i})^{r+1}\prod_{k=1}^{r}\omega_{k}}\cdot
H_{r}\Big(-\frac{n\alpha_{i}}{\omega_{1}},\cdots,-\frac{n\alpha_{i}}{\omega_{r}}\Big),
\eea
where the function $H_{r}$ is defined by (\ref{183}).
For general decorated stars $\Gamma_{i;\vec{\omega}}$, taking into account the contribution $\VC_{\Gamma_{i;\vec{\omega}}}^{0;W}$ as in the preceding sections, we naturally define $\VC_{\Gamma_{i;\vec{\omega}}}$  still by (\ref{182}). For a finite set of variables $S=\{w_{1},\cdots,w_{r}\}$, let
\bea\label{187}
H(S)=H_{r}(w_{1},\cdots,w_{r}).
\eea
For $\Gamma\in\DRT_{\emptyset}^{d}$, we define
\bea\label{184}
 &&\Cont_{\Gamma}(N_{1,d}^{W})=\frac{1}{|\Aut(\Gamma)|}\prod_{v\in \Ver}\Cont_{\Gamma;v}(N_{1,d}^{W})\prod_{e\in \Edg} \Cont_{\Gamma;e}(N_{1,d}^{W}),\nn\\
 \eea
 where
 \bea\label{185}
\Cont_{\Gamma;v_{0}}(N_{1,d}^{W})&=&\Big(\frac{\prod_{j\in [n]\backslash\{i\}}(\alpha_{i}-\alpha_{j})}{n\alpha_{i}}
\Big)^{|\Edg(v_{0})|-1}\nn\\
&&\cdot\Bigg[\int_{\Mbar_{1,\val(v_{0})}}\frac{\prod_{j\in [n]\backslash\{i\}}\Lambda_{1}^{\vee}(\alpha_{i}-\alpha_{j})}
{\Lambda_{1}^{\vee}(n\alpha_{i})\prod_{e\in \Edg(v_{0})}\Big(\frac{\alpha_{v_{0},e}}{d(e)}-\psi_{(v_{0},e)}\Big)}\nn\\
&&+\frac{\prod_{j\in[n]\backslash\{i\}}(-n\alpha_{i}+\alpha_{i}-\alpha_{j})}
{(-n\alpha_{i})^{r+1}\prod_{k=1}^{r}\omega_{k}}\cdot
H\Big(\big\{-\frac{n\alpha_{i}\cdot d(e)}{\alpha_{v_{0},e}}\big\}_{e\in \Edg(v_{0})}\Big)\Bigg],
\eea
and $\Cont_{\Gamma;v}(N_{1,d}^{W})$, $ \Cont_{\Gamma;e}(N_{1,d}^{W})$, as well as the localization contribution of $\Gamma\in\DOL_{\emptyset}^{d}$, are defined as in the genus zero virtual localization. By the SvR for $W$ (\cite{ZingerSvR}), we have
\bea\label{186}
N_{1,d}^{W}=\sum_{\Gamma\in \DOL_{\emptyset}^{d}}\Cont_{\Gamma}(N_{1,d}^{W})
+\sum_{\Gamma\in \DRT_{\emptyset}^{d}}\Cont_{\Gamma}(N_{1,d}^{W}).
\eea
 We call  (\ref{185}) the \emph{modified} contribution of the genus one vertex $v_{0}$. With the third row of (\ref{185}) dropped, we call the resulted contribution the \emph{naive contribution}. Then we can summarize the above as
 \begin{theorem}\label{188}
 Assigning the  contribution of the genus one vertex to be the modified contribution, we can compute the genus one Gromov-Witten invariants of the Calabi-Yau hypersurface $W$ in $\mathbb{P}^{n-1}$ by (the modified) virtual localization.
 \end{theorem}

 It is routine to deduce parallel results for the complete intersections in $\mathbb{P}^{n-1}$. It is reasonable to make the following conjecture.
 \begin{Conjecture}\label{189}
An analog of theorem \ref{188} holds for complete intersections in  algebraic GKM manifolds.
 \end{Conjecture}
 We leave the precise formulation of this conjecture to the reader.

\begin{remark}\label{190}
It would be fascinating if one could find a direct geometric interpretation for (\ref{182}).
Note that by the properties of $H(S)$ shown in the proof of lemma \ref{181}, it is straightforward to see the sum of the terms in the square bracket of (\ref{179}) has no factors $n\alpha_{i}$ in the denominator. In higher genera, we conjecture that there are also some natural correction terms which cancel the negative powers of the weight $E|_{P_{i}}$, such that the corresponding modified virtual localization gives the correct Gromov-Witten invariants.
\end{remark}

\section{The difference between standard and formal reduced genus one Gromov-Witten invariants of $X=\mathrm{Tot}\big(\bigoplus_{k=1}^{l}\mathcal{O}(-a_{k})\rightarrow \mathbb{P}^{n-1}\big)$}
\subsection{}
We follow the description on the equivariant cohomology of $\mathbb{P}^{n-1}$ in \cite[section 1.1]{Zinger1}. We recollect the facts and terminology of \cite{Zinger1} that we need in the following.
\begin{itemize}
  \item{} The equivariant cohomology ring of $\mathbb{P}^{n-1}$ is
  \bea
  H_{\mathbb{T}}^{*}\big(\mathbb{P}^{n-1}\big)=\mathbb{Q}[\mathbf{x},\alpha_{1},\cdots,\alpha_{n}]/\prod_{i=1}^{n}(\mathbf{x}-\alpha_{i}).
  \eea
  \item{}The restriction map on the equivariant cohomology induced by $P_{i}\hookrightarrow \mathbb{P}^{n-1}$ is given by
  \bea
  \mathbb{Q}[\mathbf{x},\alpha_{1},\cdots,\alpha_{n}]/\prod_{i=1}^{n}(\mathbf{x}-\alpha_{i})\rightarrow \mathbb{Q}[\alpha_{1},\cdots,\alpha_{n}],
  & \mathbf{x}\mapsto\alpha_{i},
  \eea
  and the localized equivariant cohomology ring is
  \bea
  \mathcal{H}_{\mathbb{T}}^{*}\big(\mathbb{P}^{n-1}\big)=\mathbb{Q}_{\alpha}[\mathbf{x}]/\prod_{i=1}^{n}(\mathbf{x}-\alpha_{i}),
  \eea
  where  $\mathbb{Q}_{\alpha}=\mathbb{Q}(\alpha_{1},\cdots,\alpha_{n})$ is the field of fractions of $\mathbb{Q}[\alpha_{1},\cdots,\alpha_{n}]$.
  \item{} The tautological line bundle $\gamma_{n-1}\big(\cong \mathcal{O}_{\mathbb{P}^{n-1}}(-1)\big)$ is linearized such that the equivariant Euler class restricted to the fixed points are
      \bea
      \mathbf{e}(\gamma_{n-1})\big|_{P_{i}}=-\alpha_{i}
      \eea
      for $1\leq i\leq n$.
  \item{}   The equivariant Poincar\'{e} dual of $P_{i}$ is
  \bea
  \phi_{i}=\prod_{j\in [n]\backslash\{i\}}(\mathbf{x}-\alpha_{j}).
  \eea
  The tangent bundle $T\mathbb{P}^{n-1}$ is linearized such that the equivariant Euler class restricted to the fixed points are
      \bea
      \mathbf{e}(T\mathbb{P}^{n-1})|_{P_{i}}=\prod_{j\in[n]\backslash\{i\}}(\alpha_{i}-\alpha_{j})=\phi_{i}|_{P_{i}}.
      \eea
      for $1\leq i\leq n$.
\end{itemize}
This $(\mathbb{C}^{*})^{n}$-action can be lifted to the vector bundle $E=\prod_{k=1}^{l}\mathcal{O}(-a_{k})$, and there are many choices of  liftings. In this article we fix the lifting\footnote{As shown in \cite{GP}, \cite{KP}, \cite{PZ} and \cite{Hu}, for certain extremal cases of $X$, other choices of the liftings will make the localization computations extremely simple.}, such that as an equivariant vector bundle we have $E\cong \bigotimes_{k=1}^{l}\gamma_{n-1}^{\otimes a_{k}}$. Thus the equivariant Euler class of $E$ restricted to the fixed points are
\bea
\mathbf{e}(E)|_{P_{i}}=\prod_{k=1}^{l}(-a_{k}\alpha_{i}).
\eea

The localization contributions have been described as in section 2, except that
the contribution of an edge $e=\{v_{1},v_{2}\}$ with $\mathfrak{m}(v_{1})=i$ and $\mathfrak{m}(v_{2})=j$ is
\bea\label{12}
&&\Cont_{\Gamma;e}\big(\langle\mu_{1},\cdots,\mu_{|J|}\rangle_{1,J,d}^{X}\big)=\nn\\
&&\frac{\prod_{k=1}^{l}\prod_{a=1}^{a_{k}d(e)-1}(-a_{k}\alpha_{j}+a\frac{\alpha_{j}-\alpha_{i}}{d(e)})}
{\mathfrak{d}(e)\cdot
(\frac{d(e)!}
{d(e)^{d}})^{2}(\alpha_{i}-\alpha_{j})^{d(e)}(\alpha_{j}-\alpha_{i})^{d(e)}\prod_{k\neq i,j}\prod_{a=0}^{d(e)}(\alpha_{i}-\alpha_{k}+a\frac{\alpha_{j}-\alpha_{i}}{d(e)})}.
\eea

\subsection{}
As in \cite{Zinger1}, let
\bea\label{210}
\mathfrak{F}(\alpha,\mathbf{x},Q)=\sum_{d=1}^{\infty}Q^{d}\big(\ev_{*}\mathbf{e}(\mathcal{U}_{1})\big)\in \big(H_{\mathbb{T}}^{n-2}(\mathbb{P}^{n-1})\big)[[Q]],
\eea
we have
\bea\label{211}
\mathfrak{F}(\alpha,\mathbf{x},Q)=\mathfrak{F}_{0}(Q)\mathbf{x}^{n-2}+\mathfrak{F}_{1}(\alpha,Q)\mathbf{x}^{n-3}
+\cdots+\mathfrak{F}_{n-2}(\alpha,Q),
\eea
where $\mathfrak{F}_{p}(\alpha,Q)\in \mathbb{Q}[[Q]][\alpha_{1},\cdots,\alpha_{n}]$ is of degree $p$ and symmetric in $\alpha_{1},\cdots,\alpha_{n}$. By the divisor equation and a degree counting we have
\bea\label{212}
\frac{d}{dT}\sum_{d=1}^{\infty}e^{dT}N_{1,d}^{X}=\mathfrak{F}_{0}(e^{T}).
\eea
To determine $\mathfrak{F}(\alpha,\mathbf{x},Q)$ (and thus $\mathfrak{F}_{0}(Q)$), it suffices to compute $\mathfrak{F}(\alpha,\alpha_{i},Q)$. By the Atiyah-Bott localization theorem (\cite{AB}) on $\mathbb{P}^{n-1}$,
\bea\label{293}
\mathfrak{F}(\alpha,\alpha_{i},Q)=\sum_{d=1}^{\infty}Q^{d}\langle \phi_{i}\rangle_{1,1,d}^{X}.
\eea
By corollary \ref{112},
\bea\label{289}
\langle \phi_{i}\rangle_{1,1,d}^{X}=\langle \phi_{i}\rangle_{1,1,d}^{0;X}
+\frac{1}{24}\sum_{m\geq 1}(-1)^{m}(m-1)!\sum_{p=0}^{n+l-1-2m}\langle\tilde{\eta}_{p}\mathbf{c}_{n+l-1-2m-p}(TX);\phi_{i}\rangle_{(m,1,d)}^{X}.
\eea
Let
\bea\label{213}
\mathcal{C}_{i}(Q)=\sum_{d=1}^{\infty}Q^{d}\frac{1}{24}\sum_{m\geq 1}(-1)^{m}(m-1)!\sum_{p=0}^{n+l-1-2m}\langle\tilde{\eta}_{p}\mathbf{c}_{n+l-1-2m-p}(TX);\phi_{i}\rangle_{(m,1,d)}^{X}.
\eea
The same reasoning shows that there exist $\mathfrak{C}(\alpha,\mathbf{x},Q)\in \big(H_{\mathbb{T}}^{n-2}(\mathbb{P}^{n-1})\big)[[Q]]$ such that
\bea\label{290}
\mathfrak{C}(\alpha,\mathbf{x},Q)=\mathfrak{C}_{0}(Q)\mathbf{x}^{n-2}+\mathfrak{C}_{1}(\alpha,Q)\mathbf{x}^{n-3}+\cdots+\mathfrak{C}_{n-2}(\alpha,Q),
\eea
where $\mathfrak{C}_{p}(\alpha,Q)$ is of degree $p$ and symmetric in $\alpha_{1},\cdots,\alpha_{n}$, and
\bea\label{295}
\mathcal{C}_{i}(Q)=\mathfrak{C}(\alpha,\alpha_{i},Q).
\eea
 By the divisor equation (\ref{101}), we have
\bea\label{291}
\mathfrak{C}_{0}(e^{dT})=\frac{d}{dT}\Bigg(\sum_{d=1}^{\infty}Q^{d}\frac{1}{24}\sum_{m\geq 1}(-1)^{m}(m-1)!\sum_{p=0}^{n+l-1-2m}\langle\eta_{p}c_{n+l-1-2m-p}(TX)\rangle_{(m,\emptyset,d)}^{X}\Bigg).
\eea
By (\ref{211}), (\ref{293}), (\ref{289}),(\ref{290}) and (\ref{295}),  there exist $\mathfrak{F}^{0}(\alpha,\mathbf{x},Q)\in \big(H_{\mathbb{T}}^{n-2}(\mathbb{P}^{n-1})\big)[[Q]]$ such that
\bea\label{292}
\mathfrak{F}^{0}(\alpha,\mathbf{x},Q)=\mathfrak{F}_{0}^{0}(Q)\mathbf{x}^{n-2}+\mathfrak{F}_{1}^{0}(\alpha,Q)\mathbf{x}^{n-3}
+\cdots+\mathfrak{F}_{n-2}^{0}(\alpha,Q),
\eea
where $\mathfrak{F}_{p}^{0}(\alpha,Q)$ is of degree $p$ and symmetric in $\alpha_{1},\cdots,\alpha_{n}$, and
\bea\label{296}
\sum_{d=1}^{\infty}Q^{d}\langle \phi_{i}\rangle_{1,1,d}^{0;X}=\mathfrak{F}_{0}^{0}(\alpha,\alpha_{i},Q).
\eea
By (\ref{212}), (\ref{291}) and the LSvR for $N_{1,d}^{X}$, we deduce the \emph{divisor equation} for the reduced genus one Gromov-Witten invarians of $X$:
\bea\label{294}
\mathfrak{F}_{0}^{0}(e^{T})=\mathfrak{F}_{0}(e^{T})-\mathfrak{C}_{0}^{0}(e^{T})=\frac{d}{dT}\sum_{d=1}^{\infty}e^{dT}N_{1,d}^{0;X}.
\eea

In subsection 4.1, we use the known result on the generating function of one-point genus zero Gromov-Witten invariants for $X$ to obtain a formula for
\bea\label{275}
\sum_{d=1}^{\infty}Q^{d}\frac{1}{24}\sum_{m\geq 1}(-1)^{m}(m-1)!\sum_{p=0}^{n+l-1-2m}\langle\eta_{p}c_{n+l-1-2m-p}(TX)\rangle_{(m,\varnothing,d)}^{X}.
\eea
In section 5, we compute the righthand-side of (\ref{294}). By theorem \ref{270} and theorem \ref{268}, we have
\begin{theorem}\label{297}
\begin{multline}\label{298}
\sum_{d=1}^{\infty}Q^{d}N_{1,d}^{X}=\frac{n}{48}\Big(n-1-2\sum_{k=1}^{l}\frac{1}{a_{k}}\Big)\big(T-t\big)\\
-\left\{\begin{array}{lll}
\frac{n+l}{48}\log (1-\prod_{k=1}^{l}(-a_{k})^{a_{k}}q)+\sum_{p=l}^{\frac{n+l-2}{2}}\frac{(n+l-2p)^{2}}{8}\log I_{p}(q), & \mathrm{if}& 2\mid(n+l);\\
\frac{n+l-3}{48}\log (1-\prod_{k=1}^{l}(-a_{k})^{a_{k}}q)+\sum_{p=l}^{\frac{n+l-3}{2}}\frac{(n+l-2p)^{2}-1}{8}\log I_{p}(q), & \mathrm{if}& 2\nmid(n+l).\\
\end{array}\right.
\end{multline}
\end{theorem}
\subsection{}

Let
\bea\label{273}
Z_{r}(Q)&=&\prod_{k=1}^{l}(-a_{k})\sum_{d=1}^{\infty}Q^{d}\langle \psi^{r}\ev^{*}H^{n+l-3-r}\mathbf{e}(\mathcal{U}_{0}),[\Mbar_{0,1}(\mathbb{P}^{n-1},d)]\rangle.
\eea
for $0\leq r\leq n+l-3$. We would like to formally write
\bea\label{285}
Z_{r}(Q)&=&\sum_{d=1}^{\infty}Q^{d}\langle \psi^{r}\ev^{*}H^{n-3-r}\mathbf{e}(\mathcal{U}_{0}^{\prime}),[\Mbar_{0,1}(\mathbb{P}^{n-1},d)]\rangle.
\eea
Givental gave a mirror formula for the generating function of $Z_{r}(Q)$ .
\begin{theorem}(A. Givental)
\bea\label{274}
e^{Tw}\Big(1+\sum_{r=0}^{n+l-3}Z_{r}(Q)w^{r+2}\Big)=R(w,t) \mod(w^{n+l}).
\eea
\end{theorem}
\begin{proof} Let\footnote{In  the definition of $J(\hbar,\alpha,Q)$ in \cite{Givental} the $\mathbf{e}_{T}(\mathcal{U}_{0}^{\prime})$ is replaced by
$\mathbf{e}_{T}(\mathcal{U}_{0})$, while the equivariant Poincar\'{e} pairing is a twisted one. One can also see \cite[theorem 4.6]{Popa2}.}
\bea\label{279}
J(\hbar,\alpha,Q)=1+\frac{1}{\hbar}\sum_{d=1}^{\infty}Q^{d}\ev_{*}\Big(\frac{\mathbf{e}_{T}(\mathcal{U}_{0}^{\prime})}{\hbar-\psi}\Big),
\eea
where $\ev_{*}$ is the push-forward map in the equivariant cohomology£¬ and
\bea\label{280}
I(\hbar,\alpha,q)=\sum_{d=0}^{\infty}q^{d}\frac{\prod_{k=1}^{l}\prod_{s=0}^{a_{k}d-1}(-a_{k}x-s\hbar)}
{\prod_{i=1}^{n}\prod_{s=1}^{d}(x-\alpha_{i}+s\hbar)}.
\eea
The genus zero mirror theorem (\cite[theorem 4.2]{Givental}) says
\bea\label{281}
J(\hbar,\alpha,Q)=e^{-\frac{xf(q)}{\hbar}}I(\hbar,\alpha,q).
\eea
Since $x$ is not a zero divisor in $H_{\mathbb{T}}^{*}(\mathbb{P}^{n-1})$, from (\ref{281}) we have
\bea\label{284}
\frac{\prod_{k=1}^{l}(-a_{k})}{\hbar}\sum_{d=1}^{\infty}Q^{d}\ev_{*}\Big(\frac{\mathbf{e}_{T}(\mathcal{U}_{0})}{\hbar-\psi}\Big)
=\Bigg(e^{-xf(q)}\sum_{d=0}^{\infty}q^{d}\frac{\prod_{k=1}^{l}\prod_{s=0}^{a_{k}d-1}(-a_{k}x-s\hbar)}
{\prod_{i=1}^{n}\prod_{s=1}^{d}(x-\alpha_{i}+s\hbar)}-1\Bigg)\Bigg/x^{l}.\nn\\
\eea
The lefthand-side of (\ref{284}) lies in $\in H_{\mathbb{T}}^{*}(\mathbb{P}^{n-1})$; when $l=1$ this is obvious, and when $l=2$, $f(q)=0$.
Taking the nonequivariant limit, it is easy to deduce (\ref{274}) from (\ref{284}).
\end{proof}

In the following we use (\ref{274}) to compute (\ref{275}).

\begin{theorem}\label{270}
\bea\label{271}
&&\sum_{d=1}^{\infty}Q^{d}\sum_{m\geq 1}(-1)^{m}(m-1)!\sum_{p=0}^{n+l-1-2m}\langle\eta_{p}c_{n+l-1-2m-p}(TX)\rangle_{(m,\varnothing,d)}^{X}.
\nn\\
&=&-\frac{1}{\prod_{k=1}^{l}(-a_{k})}\Res_{w=0}
\Bigg\{\frac{\prod_{k=1}^{l}(1-a_{k}w)\cdot\big((1+w)^{n}-w^{n}\big)}{w^{n+l}} \Big( -Tw+\ln R(w,t)\Big)\Bigg\}.\nn\\
\eea
\end{theorem}
\begin{proof} The proof is parallel to the proof of  \cite[lemma 2.2]{ZingerSvR}.
\bea\label{276}
&&\sum_{m\geq 1}(-1)^{m}(m-1)!\sum_{p=0}^{n+l-1-2m}\langle\eta_{p}c_{n+l-1-2m-p}(TX)\rangle_{(m,\varnothing,d)}^{X}
\nn\\
&=&\sum_{p=2}^{n+l-1}\sum_{m=1}^{2m\leq p}(-1)^{m}(m-1)!
\langle\eta_{p-2m}c_{n+l-1-p}(TX)\rangle_{(m,\varnothing,d)}^{X}\nn\\
&=&\sum_{p=2}^{n+l-1}[w^{n+l-1-p}]\Big((1+w)^{n}\prod_{k=1}^{l}(1-a_{k}w)\Big)\sum_{m=1}^{2m\leq p}(-1)^{m}(m-1)!
\langle\eta_{p-2m}H^{n+l-1-p}\rangle_{(m,\varnothing,d)}^{X}.\nn\\
\eea
By  the normalization sequence,
\bea\label{277}
&&\sum_{m=1}^{2m\leq p}(-1)^{m}(m-1)!
\langle\eta_{p-2m}H^{n+l-1-p}\rangle_{(m,\varnothing,d)}^{X}\nn\\
&=&(-1)^{m}\sum_{m=1}^{2m\leq p}\frac{1}{m}\langle\ev_{0}^{*}H^{n+l-1-p}\Big(\prod_{k=1}^{l}(-a_{k}H)\Big)^{m-1}
\prod_{i=1}^{m}\pi_{i}^{*}\frac{\mathbf{e}(\mathcal{U}_{0})}{1-\psi_{0}},[\Mbar_{(m,\varnothing)}(\mathbb{P}^{n-1},d)]\rangle\nn\\
&=&(-1)^{m}\sum_{m=1}^{2m\leq p}\frac{\Big(\prod_{k=1}^{l}(-a_{k})\Big)^{m-1}}{m}\langle\ev_{0}^{*}H^{n+lm-1-p}
\prod_{i=1}^{m}\pi_{i}^{*}\frac{\mathbf{e}(\mathcal{U}_{0})}{1-\psi_{0}},[\Mbar_{(m,\varnothing)}(\mathbb{P}^{n-1},d)]\rangle.
\eea

By the decomposition of the diagonal in $\big(\mathbb{P}^{n-1}\big)^{m}$,
\bea\label{278}
&&\sum_{d=1}^{\infty}Q^{d}\sum_{m=1}^{2m\leq p}\frac{\Big(-\prod_{k=1}^{l}(-a_{k})\Big)^{m}}{m}\langle\ev_{0}^{*}H^{n+lm-1-p}
\prod_{i=1}^{m}\pi_{i}^{*}\frac{\mathbf{e}(\mathcal{U}_{0})}{1-\psi_{0}},[\Mbar_{(m,\varnothing)}(\mathbb{P}^{n-1},d)]\rangle\nn\\
&=&\sum_{m=1}^{2m\leq p}\frac{\Big(-\prod_{k=1}^{l}(-a_{k})\Big)^{m}}{m}\sum_{d=1}^{\infty}Q^{d}\sum_{\stackrel{\sum_{i=1}^{m}d_{i}=d}{d_{i}>0}}
\sum_{\stackrel{\sum_{i=1}^{m}p_{i}=p-lm}{p_{i}\geq 0}}
\prod_{i=1}^{m}\langle \ev_{0}^{*}H^{n-1-p_{i}}\frac{\mathbf{e}(\mathcal{U}_{0})}{1-\psi_{0}},[\Mbar_{(1,\varnothing)}(\mathbb{P}^{n-1},d_{i})]\rangle\nn\\
&=&\sum_{m=1}^{2m\leq p}\frac{(-1)^{m}}{m}\sum_{d=1}^{\infty}Q^{d}\sum_{\stackrel{\sum_{i=1}^{m}d_{i}=d}{d_{i}>0}}
\sum_{\stackrel{\sum_{i=1}^{m}p_{i}=p}{p_{i}\geq l}}
\prod_{i=1}^{m}\langle \ev_{0}^{*}H^{n-1-p_{i}}\frac{\mathbf{e}(\mathcal{U}_{0}^{\prime})}{1-\psi_{0}},[\Mbar_{(1,\varnothing)}(\mathbb{P}^{n-1},d_{i})]\rangle\nn\\
&=&\sum_{m=1}^{2m\leq p}\frac{(-1)^{m}}{m}\sum_{d=1}^{\infty}Q^{d}\sum_{\stackrel{\sum_{i=1}^{m}d_{i}=d}{d_{i}>0}}
\sum_{\stackrel{\sum_{i=1}^{m}p_{i}=p}{p_{i}\geq 2}}
\prod_{i=1}^{m}\langle \psi^{p_{i}-2}\ev^{*}H^{n-1-p_{i}}\mathbf{e}(\mathcal{U}_{0}^{\prime}),[\Mbar_{0,1}(\mathbb{P}^{n-1},d_{i})]\rangle\nn\\
&=&\sum_{m=1}^{2m\leq p}\frac{(-1)^{m}}{m}\sum_{\stackrel{\sum_{i=1}^{m}p_{i}=p}{p_{i}\geq 2}}
Z_{p_{i}-2}(Q)\nn\\
&=&-[w^{p}]\Bigg(\ln \Big(1+\sum_{r=0}^{n+l-3}Z_{r}(Q)w^{r+2}\Big)\Bigg).
\eea
By (\ref{276}), (\ref{277}), (\ref{278}) and (\ref{274}), noting that $-Tw+\ln R(w,t)\equiv 0\mod (w^{l})$, we obtain (\ref{271}).
\end{proof}

\section{Formal reduced genus one Gromov-Witten invariants of\\
$X=\mathrm{Tot}\big(\bigoplus_{k=1}^{l}\mathcal{O}(-a_{k})\rightarrow \mathbb{P}^{n-1}\big)$}
Let $A_{i}$, $\tilde{A}_{ij}$, $B_{i}$, $\tilde{B}_{ij}$ be the four types of decorated one loop graph and refined decorated rooted trees defined in \cite{Zinger1}. We write $\langle\phi_{i}\rangle_{1,1,d}^{0;X}$ as the sum of localization contributions of these graphs.
\bea\label{214}
\langle\phi_{i}\rangle_{1,1,d}^{0;X}=\mathcal{A}_{i}(Q)+\sum_{j=1}^{n}\tilde{\mathcal{A}}_{ij}(Q)
+\mathcal{B}_{i}(Q)+\sum_{j=1}^{n}\tilde{\mathcal{B}}_{ij}(Q).
\eea
Let $\mathbb{Q}[\alpha]^{S_{n}}$, $\mathcal{I}\subset \mathbb{Q}[\alpha]^{S_{n}}$, $\tilde{\mathbb{Q}}[\alpha]^{S_{n}}$,  $\tilde{\mathbb{Q}}[\alpha]^{S_{n-1}}$ and $\mathcal{K}_{i}$ be defined as in \cite{Popa1}. We will frequently make use of \cite[lemma 5.1, lemma 5.2]{Popa1}. We adopt the notation $\equiv_{i}$ of \cite{Popa1}; $F\equiv_{i}G$ means $F-G\in \mathcal{K}_{i}$. We call $\mathcal{A}_{i}(Q)+\sum_{j=1}^{n}\tilde{\mathcal{A}}_{ij}(Q)$ the type A contributions, and $\mathcal{B}_{i}(Q)+\sum_{j=1}^{n}\tilde{\mathcal{B}}_{ij}(Q)$ the type B contributions. In the following of this section, after some preparation on some properties of  hypergeometric series, we will compute the two types of contributions modulo $\mathcal{K}_{i}$ separately.\\

In the computation of the type A contributions we can  make use of Popa's results in \cite{Popa1} in a \emph{dual} way to simplify our proof. Although Popa's computation is for $\bigoplus_{k=1}^{l}\mathcal{O}(a_{k})$ where $\sum_{k=1}^{l}a_{k}=n$ and $a_{k}\geq 2$ for $1\leq k\leq l$, the last assumption $a_{k}\geq 2$ occurs only because of the factors of the form $a_{k}\alpha_{i}+\hbar$ in the denominators of some rational functions. In the local cases, we don't have factors of this form in the denominators, so we are able to apply, e.g.,  \cite[lemma 5.4]{Popa1} to our space $X$. For the computation of the genus 0 two-point functions we also refer the reader to [13].\\

By proposition \ref{249} and \ref{266}, we obtain
\begin{theorem}\label{268}
\begin{multline}\label{265}
\sum_{d=1}^{\infty}e^{dT}N_{1,d}^{0;X}=\frac{n}{48}\Big(n-1-2\sum_{k=1}^{l}\frac{1}{a_{k}}\Big)\big(T-t\big)\\
-\left\{\begin{array}{lll}
\frac{n+l}{48}\log (1-\prod_{k=1}^{l}(-a_{k})^{a_{k}}e^{t})+\sum_{p=l}^{\frac{n+l-2}{2}}\frac{(n+l-2p)^{2}}{8}\log I_{p}(e^{t}) & \mathrm{if}& 2\mid(n+l)\\
\frac{n+l-3}{48}\log (1-\prod_{k=1}^{l}(-a_{k})^{a_{k}}e^{t})+\sum_{p=l}^{\frac{n+l-3}{2}}\frac{(n+l-2p)^{2}-1}{8}\log I_{p}(e^{t}) & \mathrm{if}& 2\nmid(n+l)\\
\end{array}\right.\\
+\frac{1}{24\prod_{k=1}^{l}(-a_{k})}\Res_{w=0}
\Bigg\{\frac{\prod_{k=1}^{l}(1-a_{k}w)\cdot\big((1+w)^{n}-w^{n}\big)}{w^{n+l}} \Big( -Tw+\ln R(w,t)\Big)\Bigg\}.
\end{multline}
\end{theorem}

\subsection{Some properties of certain hypergeometric series}
Following \cite{ZagierZinger}, we denote by
\ben
\mathcal{P}\subset 1+q\mathbb{Q}(w)[[q]]
\een
the subgroup of power series in $q$ with constant term $1$ whose coefficients are rational functions in $w$ which are regular at $w=0$, and define a map $\mathbf{M}:\mathcal{P}\rightarrow \mathcal{P}$ by
\ben
\mathbf{M}F(w,q)=\Big(1+\frac{q}{w}\frac{\partial}{\partial q}\Big)\frac{F(w,q)}{F(0,q)}.
\een
Note that for $F(w,q)\in \mathcal{P}$, we have a well-defined series $\log F(w,q)\in \mathbb{Q}(w)[[q]]$. Let $\mathcal{P}_{1}$ be the subset of $\mathcal{P}$ such that $F(w,q)\in\mathcal{P}_{1}$ if and only if every coefficient of the power series $\log F(w,q)$ is $O(w)$ as $w\rightarrow \infty$. We recall the following lemma from \cite{ZagierZinger}
\begin{lemma}\label{227}
If $\mathcal{F}\in\mathcal{P}$ and $\mathbf{M}^{k}F=F$ for some $k>0$, then $F\in\mathcal{P}_{1}$.
\end{lemma}
The proof of this lemma in \cite{ZagierZinger} shows also
\begin{lemma}\label{228}
If $\mathcal{F}\in\mathcal{P}$, then $\mathbf{M}F\in\mathcal{P}_{1}$ if and only if If $F\in\mathcal{P}_{1}$.
\end{lemma}
Now following \cite{Popa1}, we define
\bea\label{225}
\mathcal{F}_{-l}(w,q)&=&\sum_{d=0}^{\infty}q^{d}\frac{\prod_{k=1}^{l}\prod_{r=0}^{a_{k}d-1}(a_{k}w+r)}
{\prod_{r=0}^{d}\big((w+r)^{n}-w^{n}\big)},
\eea
and for $p>-1$,
\bea\label{226}
\mathcal{F}_{p}&=&\mathbf{M}^{l+p}\mathcal{F}_{-l}.
\eea
We have
\ben
\mathcal{F}_{0}(w,q)&=&\sum_{d=0}^{\infty}q^{d}\frac{\prod_{k=1}^{l}\prod_{r=1}^{a_{k}d}(a_{k}w+r)}
{\prod_{r=0}^{d}\big((w+r)^{n}-w^{n}\big)},\\
\een
 and by \cite[lemma 4.1]{Popa1}, we have
\ben
\mathbf{M}^{n}\mathcal{F}_{0}=\mathcal{F}_{0}.
\een
Thus by lemma \ref{227} and lemma \ref{228} we have $\mathcal{F}_{p}\in \mathcal{P}_{1}$ for $p\geq -l$. By the definition of $\mathcal{P}_{1}$, as $w\rightarrow \infty$, $\mathcal{F}_{p}(w,q)$ has an asymptotic expansion of the form
\bea\label{229}
\mathcal{F}_{p}(w,q)\sim e^{\mu_{p}(q)w}\sum_{s=0}^{\infty}\Phi_{p,s}(q)w^{-s}.
\eea
Let
\bea
L(q)=(1-\prod_{i=1}^{l}a_{i}^{a_{i}}q)^{-\frac{1}{n}}.
\eea
\begin{proposition}
The power series $\mu_{-l}(q)$ and $\Phi_{-l,0}(q)$ defined in (\ref{229}) are given by
\bea\label{230}
\mu_{-l}(q)=\int_{0}^{q}\frac{L(u)-1}{u}du,&
\Phi_{-l,0}(q)=L^{\frac{1-l}{2}}.
\eea
\end{proposition}
\begin{proof} Since $\mathcal{F}_{p}(0,q)=0$ for $-l\leq p\leq -1$, we have
\ben
\mathcal{F}_{p}(w,q)=\Big(1+\frac{q}{w}\frac{\partial}{\partial q}\Big)^{l+p}\mathcal{F}_{-l}
\een
for $-l\leq p\leq 0$. Thus let $1+\frac{q}{w}\frac{\partial}{\partial q}$ operates on both sides of (\ref{229}), we obtain
\bea\label{234}
\mu_{p+1}(q)=\mu_{p}(q),&
\Phi_{p+1,0}(q)=\big(1+q\mu_{p}^{\prime}(q)\big)\Phi_{p,0}(q)
\eea
for $-l\leq p\leq -1$. By \cite[proposition 4.3]{Popa1} we have
\ben
\mu_{0}(q)=\int_{0}^{q}\frac{L(u)-1}{u}du,&
\Phi_{0,0}(q)=L^{\frac{1+l}{2}}.
\een
from which (\ref{230}) follows.
\end{proof}

From now on, we denote $\mu(q)=\mu_{-l}(p)=\cdots=\mu_{0}(q)$. For convenience to consult results from \cite{Popa1}, we define
\bea\label{263}
\tilde{I}_{p}(q)=\mathbf{M}^{p}\mathcal{F}_{0}(w,q)\big|_{w=0}.
\eea
Thus
\bea\label{264}
I_{p+l}(q)=\tilde{I}_{p}\big((-1)^{n}q\big),
\eea
for $0\leq p\leq n-1$.

\subsection{Some properties of genus zero generating functions}
Let
\bea\label{201}
\mathcal{Z}_{i}^{*}(\hbar, Q)=\sum_{d=1}^{\infty}Q^{d}\int_{\Mbar_{0,2}(\mathbb{P}^{n-1},d)}\frac{\mathbf{e}(\mathcal{U}_{0}^{\prime})}{\hbar-\psi_{1}}\ev_{1}^{*}\phi_{i},
\eea
\bea\label{202}
\mathcal{Z}_{ij}^{*}(\hbar, Q)=\hbar^{-1}\sum_{d=1}^{\infty}Q^{d}\int_{\Mbar_{0,2}(\mathbb{P}^{n-1},d)}
\frac{\mathbf{e}(\mathcal{U}_{0}^{\prime})}{\hbar-\psi_{1}}\ev_{1}^{*}\phi_{i}\ev_{2}^{*}\phi_{j},
\eea
\bea\label{203}
\widetilde{\mathcal{Z}}_{ij}^{*}(\hbar_{1},\hbar_{2}, Q)=\frac{1}{2\hbar_{1}\hbar_{2}}\sum_{d=1}^{\infty}Q^{d}\int_{\Mbar_{0,2}(\mathbb{P}^{n-1},d)}
\frac{\mathbf{e}(\mathcal{U}_{0}^{\prime})}{(\hbar-\psi_{1})(\hbar-\psi_{2})}\ev_{1}^{*}\phi_{i}\ev_{2}^{*}\phi_{j}.
\eea
\bea\label{252}
\mathcal{Z}_{i}^{\prime *}(\hbar, Q)=\sum_{d=1}^{\infty}Q^{d}\int_{\Mbar_{0,1}(\mathbb{P}^{n-1},d)}
\frac{\mathbf{e}(\mathcal{U}_{0}^{\prime})}{\hbar-\psi_{1}}\ev_{1}^{*}\phi_{i}
\eea
By the string equation we have
\bea\label{204}
\mathcal{Z}_{i}^{*}(\hbar, Q)=\hbar^{-1}\mathcal{Z}_{i}^{\prime *}(\hbar, Q).
\eea
Define
\bea\label{205}
\eta_{i}(Q)=\Res_{\hbar=0}\Big\{\log\big(1+\mathcal{Z}_{ij}^{*}(\hbar, Q)\big)\Big\},
\eea
and
\bea\label{206}
\Phi_{0}(\alpha_{i},Q)=\Res_{\hbar=0}\Big\{\hbar^{-1}e^{-\eta_{i}(Q)/\hbar}\big(1+\mathcal{Z}_{i}^{*}(\hbar, Q)\big)\Big\}.
\eea
\begin{lemma}\label{207}
$e^{-\eta_{i}(Q)/\hbar}\big(1+\mathcal{Z}_{ij}^{*}(\hbar, Q)\big)$ is holomorphic at $\hbar=0$, and thus
\bea\label{208}
\Phi_{0}(\alpha_{i},Q)=e^{-\eta_{i}(Q)/\hbar}\big(1+\mathcal{Z}_{i}^{*}(\hbar, Q)\big)\big|_{\hbar=0}.
\eea
\end{lemma}
\begin{proof} The proof is the same as the proof of  \cite[lemma 2.3]{Zinger1}. The only change is to replace $\mathbf{e}(\mathcal{V}_{0}^{\prime})$ there by $\mathbf{e}(\mathcal{U}_{0}^{\prime})$, and note that as in (2.10) of that proof, by the normalization sequence we still have
\bea\label{209}
\mathbf{e}(\mathcal{U}_{0}^{\prime})=\prod_{e\in \Edg(v_{0})}\pi_{e}^{*}\mathbf{e}(\mathcal{U}_{0}^{\prime}).
\eea
\end{proof}

Let
\bea\label{220}
f(q)=\left\{\begin{array}{lll}
\sum_{d=1}^{\infty}q^{d}\frac{(-1)^{nd}}{d}\frac{(nd)!}{(d!)^{n}}, & \mathrm{if} & l=1;\\
0 & \mathrm{if} & l\geq 2.\end{array}\right.
\eea
In any case the mirror map is
\bea\label{221}
T=t+f(e^{t}).
\eea

\begin{lemma}\label{224}
\bea\label{222}
1+\mathcal{Z}_{i}^{*}(\hbar, Q)-e^{-f(q)\frac{\alpha_{i}}{\hbar}}\mathcal{F}_{-l}\big(\frac{\alpha_{i}}{\hbar},(-1)^{n}q\big)\in \mathcal{I}\cdot \tilde{\mathbb{Q}}_{i}[\alpha]^{S_{n-1}}[[q],
\eea
\bea\label{223}
\alpha_{i}^{n-2}\alpha_{j}+\hbar\mathcal{Z}_{ji}^{*}(\hbar, Q)-\alpha_{i}^{n-2}\alpha_{j}e^{-f(q)\frac{\alpha_{j}}{\hbar}}
\frac{\mathbf{M}\mathcal{F}_{-l}\big(\frac{\alpha_{j}}{\hbar},(-1)^{n}q\big)}{I_{1,1}(q)}\in \mathcal{K}_{i}[[q]].
\eea
\end{lemma}
\begin{proof} This follows from \cite[theorem 4.6]{Popa2}; the argument is almost the same as in the proof of the first part of \cite[lemma 5.3]{Popa1}, so we omit it.
\end{proof}

\begin{lemma}\label{238}
\bea\label{239}
n\alpha_{i}^{n-1}+2(\hbar_{1}+\hbar_{2})\hbar_{1}\hbar_{2}\widetilde{\mathcal{Z}}_{ii}^{*}(\hbar_{1},\hbar_{2}, Q)&\nn\\
-\alpha_{i}^{n-1}e^{-f(q)\alpha_{i}\big(\frac{1}{\hbar_{1}}+\frac{1}{\hbar_{2}}\big)}
\mathbb{F}\big(\frac{\alpha_{i}}{\hbar_{1}},\frac{\alpha_{i}}{\hbar_{2}},(-1)^{n}q\big)&
\in \mathcal{I}\cdot \tilde{\mathbb{Q}}_{i}[\alpha]^{S_{n-1}}[[q]],\nn\\
\eea
where
\bea\label{240}
\mathbb{F}(w_{1},w_{2},q)&=&\sum_{p=0}^{n-1-l}\frac{\mathbf{M}^{p}\mathcal{F}_{0}(w_{1},q)}{\tilde{I}_{p,p}(q)}
\frac{\mathbf{M}^{n-1-l-p}\mathcal{F}_{0}(w_{2},q)}{\tilde{I}_{n-1-l-p,n-1-l-p}(q)}\nn\\
&&+\sum_{p=1}^{l}\frac{\mathbf{M}^{n-1-l+p}\mathcal{F}_{0}(w_{1},q)}{\tilde{I}_{n-1-l+p,n-1-l+p}(q)}
\frac{\mathbf{M}^{n-p}\mathcal{F}_{0}(w_{2},q)}{\tilde{I}_{n-p,n-p}(q)}.
\eea
\end{lemma}
\begin{proof} By \cite[corollary 4.6, theorem 4.7, remark 4.4]{Popa2}, it is not hard to find that
\bea\label{243}
&&n\alpha_{i}^{n-1}+2(\hbar_{1}+\hbar_{2})\hbar_{1}\hbar_{2}\widetilde{\mathcal{Z}}_{ii}^{*}(\hbar_{1},\hbar_{2}, Q)
-\alpha_{i}^{n-1}e^{-f(q)\alpha_{i}\big(\frac{1}{\hbar_{1}}+\frac{1}{\hbar_{2}}\big)}\nn\\
&&\cdot\Bigg(\sum_{p=0}^{n-1-l}\frac{\mathbf{M}^{p}\mathcal{F}_{0}(w_{1},(-1)^{n}q)}{\tilde{I}_{p,p}((-1)^{n}q)}
\frac{\mathbf{M}^{n-1-l-p}\mathcal{F}_{0}(w_{2},(-1)^{n}q)}{\tilde{I}_{n-1-l-p,n-1-l-p}((-1)^{n}q)}\nn\\
&&+\sum_{p=1}^{l}\frac{\mathbf{M}^{n-1-l+p}\mathcal{F}_{0}(w_{1},(-1)^{n}q)}{\tilde{I}_{n-1-l+p,n-1-l+p}((-1)^{n}q)}
\mathcal{F}_{-p}\big(w_{2},(-1)^{n}q\big)
\Bigg)\in \mathcal{I}\cdot \tilde{\mathbb{Q}}_{i}[\alpha]^{S_{n-1}}[[q].\nn\\
\eea
Moreover, in the proof of \cite[lemma 4.5]{Popa1} we see that
\bea\label{241}
\frac{\mathcal{F}_{n-l}(w,q)}{\tilde{I}_{n-l,n-l}(q)}=\mathcal{F}_{-l}(w,q).
\eea
Furthermore by  \cite[(4.7)]{Popa1} we see $\tilde{I}_{p,p}(q)=1$ for $n-l+1\leq p\leq n-1$, therefore by (\ref{241}) we have
\bea\label{242}
\mathcal{F}_{-p}(w,q)=\mathcal{F}_{n-p}(w,q)=\frac{\mathcal{F}_{n-p}(w,q)}{\tilde{I}_{n-p,n-p}(q)}
\eea
for $1\leq p\leq l-1$. Substituting (\ref{241}) and (\ref{242}) into (\ref{243}) we obtain (\ref{239}).
\end{proof}

By (\ref{229}) and (\ref{222}) we have

\begin{corollary}\label{231}
\bea\label{232}
\eta_{i}(Q)-\big(\mu\big((-1)^{n}q\big)-f(q)\big)\alpha_{i}\in\mathcal{I}\cdot \tilde{\mathbb{Q}}_{i}[\alpha]^{S_{n-1}}[[q],
\eea
\bea\label{233}
\Phi_{0}(\alpha_{i},Q)-\Phi_{-l,0}\big((-1)^{n}q\big)\in\mathcal{I}\cdot \tilde{\mathbb{Q}}_{i}[\alpha]^{S_{n-1}}[[q].
\eea
\end{corollary}

\begin{lemma}\label{235}
\bea\label{236}
\Res_{\hbar=0}\Big\{\hbar^{-1}e^{-\mu(q)\frac{\alpha_{i}}{\hbar}}\mathbf{M}\mathcal{F}_{-l}\big(\frac{\alpha_{i}}{\hbar},q\big)\Big\}
=L(q)\Phi_{-l,0}(q),
\eea
\end{lemma}
\begin{proof} The $p=-l$ case of (\ref{234}) gives (\ref{236}).
\end{proof}
For later use, we record \cite[lemma 5.4]{Popa1} as follows.
\begin{lemma}\label{244}
\bea\label{237}
\Res_{\hbar_{1}=0}\Res_{\hbar_{2}=0}\Big\{\frac{e^{-\mu(q)\alpha_{i}\big(\frac{1}{\hbar_{1}}+\frac{1}{\hbar_{2}}\big)}}
{\hbar_{1}\hbar_{2}(\hbar_{1}+\hbar_{2})}\mathbb{F}\big(\frac{\alpha_{i}}{\hbar_{1}},\frac{\alpha_{i}}{\hbar_{2}},q\big)\Big\}
=\frac{2}{\alpha_{i}L(q)}\cdot q\frac{d\widetilde{A}(q)}{dq},
\eea
where
\bea\label{245}
\widetilde{A}(q)&=&\frac{n}{48}\Big(n-1-2\sum_{k=1}^{l}\frac{1}{a_{k}}\Big)\mu(q)\nn\\
&&-\left\{\begin{array}{lll}
\frac{n+1}{48}\log (1-\prod_{k=1}^{l}a_{k}^{a_{k}}q)+\sum_{p=0}^{\frac{n-2-l}{2}}\frac{(n-l-2p)^{2}}{8}\log \tilde{I}_{p,p}(q), & \mathrm{if}& 2\mid(n-l);\nn\\
\frac{n-2}{48}\log (1-\prod_{k=1}^{l}a_{k}^{a_{k}}q)+\sum_{p=0}^{\frac{n-3-l}{2}}\frac{(n-l-2p)^{2}-1}{8}\log \tilde{I}_{p,p}(q), & \mathrm{if}& 2\nmid(n-l).\nn\\
\end{array}\right.\\
\eea
\end{lemma}

\subsection{Summing the type A contributions}

The same argument as in the proof  \cite[proposition 1.1]{Zinger1} shows
\begin{proposition}\label{215}
\bea\label{216}
\mathcal{A}_{i}(Q)=\frac{1}{\Phi_{0}(\alpha_{i},Q)}\Res_{\hbar_{1}=0}\Big\{\Res_{\hbar_{2}=0}\Big\{
e^{-\eta_{i}(Q)/\hbar_{1}}e^{-\eta_{i}(Q)/\hbar_{2}}\widetilde{\mathcal{Z}}_{ii}^{*}(\hbar_{1},\hbar_{2}, Q)\Big\}\Big\},
\eea
\bea\label{217}
\tilde{\mathcal{A}}_{ij}(Q)=\frac{\mathcal{A}_{j}(Q)}{\prod_{k\in [n]\backslash\{j\}}(\alpha_{j}-\alpha_{k})}
\Res_{\hbar=0}\Big\{e^{-\eta_{j}(Q)/\hbar}\big(1+\mathcal{Z}_{ji}^{*}(\hbar, Q)\big)\Big\}.
\eea
\end{proposition}

Since
\ben
\Res_{\hbar_{1}=0}\Big\{\Res_{\hbar_{2}}\Big\{\frac{e^{-\frac{\eta_{i}(Q)}{\hbar_{1}}-\frac{\eta_{i}(Q)}{\hbar_{2}}}}
{\hbar_{1}\hbar_{2}(\hbar_{1}+\hbar_{2})}\Big\}\Big\}=0,
\een
 by (\ref{216}), (\ref{239}), (\ref{237}) (\ref{232}) and (\ref{233}) we have
\bea\label{246}
\mathcal{A}_{i}(q)-\alpha_{i}^{n-2}\frac{1}{L\big((-1)^{n}q\big)\Phi_{-l,0}\big((-1)^{n}q\big)}\cdot q\frac{d\widetilde{A}\big((-1)^{n}q\big)}{dq}
\in \mathcal{I}\cdot \tilde{\mathbb{Q}}_{i}[\alpha]^{S_{n-1}}[[q].
\eea
By (\ref{217}), (\ref{223}), (\ref{232}) and (\ref{233}) we have
\bea\label{247}
\sum_{j=1}^{n}\tilde{A}_{ij}(Q)-\sum_{j=1}^{n}\Bigg(
\frac{\alpha_{j}^{n-2}}{\prod_{k\in[n]\backslash\{j\}}(\alpha_{j}-\alpha_{k})}
\frac{}{L\big((-1)^{n}q\big)\Phi_{-l,0}\big((-1)^{n}q\big)}\cdot q\frac{d\widetilde{A}\big((-1)^{n}q\big)}{dq}&\nn\\
\cdot\Big[\frac{L\big((-1)^{n}q\big)\Phi_{-l,0}\big((-1)^{n}q\big)}{I_{1,1}(q)}-1\Big]\alpha_{i}^{n-2}\alpha_{j}\Bigg)&
\in \mathcal{K}_{i}[[q]].
\eea
Combining (\ref{246}) and (\ref{247}) we obtain
\bea\label{248}
\mathcal{A}_{i}(q)+\sum_{j=1}^{n}\tilde{A}_{ij}(Q)-\alpha_{i}^{n-2}\frac{1}{I_{1,1}(q)}\cdot
q\frac{d\widetilde{\mathcal{A}}\big((-1)^{n}q\big)}{dq}\in \mathcal{K}_{i}[[q]].
\eea
Since $I_{1,1}(q)=q\frac{dQ}{dq}$, we obtain
\begin{proposition}{\label{249}}
\begin{multline}\label{250}
\mathcal{A}_{i}(q)+\sum_{j=1}^{n}\tilde{A}_{ij}(Q)\equiv_{i}\alpha_{i}^{n-2}Q\frac{d}{dQ}\Bigg[
\frac{n}{48}\Big(n-1-2\sum_{k=1}^{l}\frac{1}{a_{k}}\Big)\mu\big((-1)^{n}q\big)\\
-\left\{\begin{array}{lll}
\frac{n+1}{48}\log (1-\prod_{k=1}^{l}(-a_{k})^{a_{k}}q)+\sum_{p=l}^{\frac{n+l-2}{2}}\frac{(n+l-2p)^{2}}{8}\log I_{p,p}(q), & \mathrm{if}& 2\mid(n-l)\\
\frac{n-2}{48}\log (1-\prod_{k=1}^{l}(-a_{k})^{a_{k}}q)+\sum_{p=l}^{\frac{n+l-3}{2}}\frac{(n+l-2p)^{2}-1}{8}\log I_{p,p}(q), & \mathrm{if}& 2\nmid(n-l)\\
\end{array}\right.
\Bigg].
\end{multline}
\end{proposition}

\subsection{Summing the type B contributions}
\begin{proposition}\label{174}
\bea\label{253}
\mathcal{B}_{i}(Q)=\frac{1}{24\alpha_{i}^{l}\prod_{k=1}^{l}(-a_{k})}\Res_{\hbar=0,\infty}
\Big\{\frac{\prod_{k=1}^{l}(\hbar-a_{k}\alpha_{i})\prod_{j=1}^{n} (\alpha_{i}-\alpha_{j}+\hbar)}{\hbar^{3}} \frac{\mathcal{Z}_{i}^{*}(\hbar,Q)}{1+\mathcal{Z}_{i}^{*}(\hbar,Q)}\Big\}.
\eea
\bea\label{254}
\tilde{B}_{ij}(Q)&=&
-\frac{1}{24\prod_{k=1}^{l}(-a_{k})\cdot\alpha_{j}^{l}\prod_{k\in[n]\backslash\{j\} }(\alpha_{j}-\alpha_{k})}\nn\\
&&\cdot\Res_{\hbar=0,\infty}
\Big\{\frac{\prod_{k=1}^{l}(\hbar-a_{k}\alpha_{j})\prod_{s=1}^{n} (\alpha_{j}-\alpha_{s}+\hbar)}{\hbar^{2}} \frac{\mathcal{Z}_{ji}^{*}(\hbar,Q)}{1+\mathcal{Z}_{j}^{*}(\hbar,Q)}
\Big\}.\nn\\
\eea
\end{proposition}
\begin{proof} Let
\ben
\Psi_{i}(\hbar,\tilde{\psi})=-\frac{\prod_{k=1}^{l}(-a_{k}\alpha_{i}+\hbar)\prod_{j\in[n]\backslash\{i\}} (\alpha_{i}-\alpha_{j}+\hbar)}{\hbar+\tilde{\psi}}.
\een

Then
\ben
\int_{\widetilde{\mathcal{Z}_{\Gamma}}}\frac{\mathbf{e}(\widetilde{\mathcal{U}}_{1})\ev_{1}^{*}\phi_{i}}
{\mathbf{e}(\mathcal{N}\widetilde{\mathcal{Z}_{\Gamma}})}&=&
\int_{\Mtilde_{1,|\mathrm{val}(v_{0})|}\times \mathbb{P}^{m_{+}-1}}\Big\{
\Psi_{i}(\hbar,\tilde{\psi})\prod_{e\in\mathrm{Edg_{+}}}\Big(
\int_{\mathcal{Z}_{\Gamma_{e}}}\frac{\mathbf{e}(\mathcal{U}_{0}^{\prime})\ev_{1}^{*}\phi_{i}}
{\mathbf{e}(\mathcal{N}\mathcal{Z}_{\Gamma_{e}})}\Big)\\
&&\times\prod_{e\in\mathrm{Edg_{-}}}\Big(
\int_{\mathcal{Z}_{\Gamma_{e}}}\frac{\mathbf{e}(\mathcal{U}_{0}^{\prime})\ev_{1}^{*}\phi_{i}}
{(\hbar-\psi_{e})\mathbf{e}(\mathcal{N}\mathcal{Z}_{\Gamma_{e}})}\Big)
\Big\}_{\hbar=\psi_{\Gamma}+\lambda}.
\een

The sum of above terms with only $m\equiv |\mathrm{Edg}(v_{0})| $ and $(\mu(\Gamma),\mathfrak{d}(\Gamma))$ fixed is
\ben
\int_{\Mtilde_{1,|\mathrm{val}(v_{0})|}}\Res_{z=\psi_{\Gamma}}\Big(\Psi_{i}(z,\tilde{\psi})\prod_{e\in
\mathrm{Edg}(v_{0})}\mathcal{Z}_{i}^{\prime*}(z,u)\Big).
\een

Since $|\mathrm{val}(v_{0})|=m+1$,  we have
\ben
&&\int_{\Mtilde_{1,|\mathrm{val}(v_{0})|}}\Res_{z=\psi_{\Gamma}}\Big(\Psi_{i}(z,\tilde{\psi})\prod_{e\in
\mathrm{Edg}(v_{0})}\mathcal{Z}_{i}^{\prime*}(z,u)\Big)\\
&=&\frac{(-1)^{m}m!}{24}\Res_{z=\psi_{\Gamma}}\Big\{
z^{-(m+2)}\prod_{k=1}^{l}(z-a_{k}\alpha_{i})\prod_{j\neq i} (\alpha_{i}-\alpha_{j}+z)\cdot (\mathcal{Z}_{i}^{\prime*}(z,u))^{m} \Big\}.
\een

Therefore summing over $m\geq 1$ and all possible $(\mu(\Gamma),\mathfrak{d}(\Gamma))$ and taking into the symmetric group $S_{m}$ we have
\ben
\mathcal{B}_{i}(u)&=&\frac{1}{24\prod_{k=1}^{l}(-a_{k})\cdot\alpha_{i}^{l}}\sum_{r\in[n]\backslash\{i\}}
\sum_{d=1}^{\infty}\sum_{m=1}^{\infty}(-1)^{m}\nn\\
&&\cdot\Res_{z=\frac{\alpha_{r}-\alpha_{i}}{d}}
\Big\{\frac{\prod_{k=1}^{l}(z-a_{k}\alpha_{i})\prod_{j\in[n]\backslash\{i\}} (\alpha_{i}-\alpha_{j}+z)}{z^{2}} (\mathcal{Z}_{i}^{*}(z,u))^{m}\Big\}\\
&=&-\frac{1}{24\prod_{k=1}^{l}(-a_{k})\cdot\alpha_{i}^{l}}\sum_{r\in[n]\backslash\{i\}}
\sum_{d=1}^{\infty}\Res_{z=\frac{\alpha_{r}-\alpha_{i}}{d}}\\
&&
\Big\{\frac{\prod_{k=1}^{l}(z-a_{k}\alpha_{i})\prod_{j=1}^{n} (\alpha_{i}-\alpha_{j}+z)}{z^{3}} \frac{\mathcal{Z}_{i}^{*}(z,u)}{1+\mathcal{Z}_{i}^{*}(z,u)}\Big\}.
\een
Then by the residue theorem we obtain (\ref{253}).\\

Besides the groups of symmetries, the contribution of type $\tilde{B}_{ij}$ with only  $ |\mathrm{Edg}(v_{0})|=m+1 $ and $(\mu(\Gamma),\mathfrak{d}(\Gamma))$ fixed is
\ben
\frac{1}{\prod_{k=1}^{l}(-a_{k})\cdot\alpha_{j}^{l}\prod_{k\in[n]\backslash\{j\}}(\alpha_{j}-\alpha_{k})}\int_{\Mtilde_{1,|\mathrm{val}(v_{0})|}}
\Res_{z=\frac{\alpha_{\mu_{\Gamma}}-\alpha_{j}}{\mathfrak{d}(\Gamma)}}\Big(\Psi_{j}(z,\tilde{\psi})(\mathcal{Z}_{j}^{\prime*}(z,Q))^{m}\cdot
z\mathcal{Z}_{ji}^{*}(z,Q)\Big).
\een

We have
\ben
&&\int_{\Mtilde_{1,|\mathrm{val}(v_{0})|}}\Res_{z=\psi_{\Gamma}}\Big(\Psi_{j}(z,\tilde{\psi})(\mathcal{Z}_{j}^{\prime*}(z,Q))^{m}\cdot
z\mathcal{Z}_{ji}^{*}(z,Q)\Big)\\
&=&\frac{(-1)^{m}m!}{24}\Res_{z=\psi_{\Gamma}}\Big\{
z^{-(m+2)}\prod_{k=1}^{l}(z-a_{k}\alpha_{i})\prod_{s\in[n]\backslash\{j\}} (\alpha_{j}-\alpha_{s}+z)\cdot (\mathcal{Z}_{j}^{\prime*}(z,Q))^{m}\cdot
z\mathcal{Z}_{ji}^{*}(z,Q) \Big\}\\
\een

Summing over $m\geq 1$ and all possible $(\mu(\Gamma),\mathfrak{d}(\Gamma))$ and taking into the account of the group of symmetries $S_{m}$,
we have
\ben
\tilde{B}_{ij}(Q)
&=&\frac{1}{24\prod_{k=1}^{l}(-a_{k})\cdot\alpha_{j}^{l}\prod_{k\in[n]\backslash\{j\}}(\alpha_{j}-\alpha_{k})}\sum_{r\in[n]\backslash\{j\}}
\sum_{d=1}^{\infty}\\
&&\Res_{z=\frac{\alpha_{r}-\alpha_{j}}{d}}
\Big\{\frac{\prod_{k=1}^{l}(z-a_{k}\alpha_{j})\prod_{s=1}^{n} (\alpha_{j}-\alpha_{s}+z)}{z^{2}}\frac{1}{1+\mathcal{Z}_{j}^{*}(z,Q)}\cdot
\mathcal{Z}_{ji}^{*}(z,Q)\Big\}.
\een
Again by the residue theorem  we obtain (\ref{254}).
\end{proof}

\begin{proposition}\label{266}
\bea\label{262}
B_{i}(Q)+\sum_{j=1}^{n}\tilde{B}_{ij}(Q)
&\equiv_{i}&\alpha_{i}^{n-2}Q\frac{d}{dQ}\Bigg[-\frac{n}{48}\Big(n-1-2\sum_{k=1}^{l}\frac{1}{a_{k}}\Big)\Big(-f(q)+\mu\big((-1)^{n}q\big)\Big)\nn\\
&&+\frac{1-l}{48}\ln\big(1-\prod_{k=1}^{l}(-a_{k})^{a_{k}}q\big)+\frac{1}{24\prod_{k=1}^{l}(-a_{k})}\nn\\
&&\cdot\Res_{w=0}
\Bigg\{\frac{\prod_{k=1}^{l}(1-a_{k}w)\cdot\big((1+w)^{n}-w^{n}\big)}{w^{n+l}} \Big( -Tw+\ln R(w,t)\Big)\Bigg\}\Bigg].\nn\\
\eea
\end{proposition}
\begin{proof} By (\ref{253}) and (\ref{222})
\bea\label{255}
\mathcal{B}_{i}(Q)
&=&\frac{1}{24\alpha_{i}^{l}\prod_{k=1}^{l}(-a_{k})}\Res_{\hbar=0,\infty}
\Big\{\frac{\prod_{k=1}^{l}(\hbar-a_{k}\alpha_{i})\prod_{k=1}^{n} (\alpha_{i}-\alpha_{k}+\hbar)}{\hbar^{3}} \frac{\mathcal{Z}_{i}^{*}(\hbar,Q)}{1+\mathcal{Z}_{i}^{*}(\hbar,Q)}\Big\}\nn\\
&\equiv_{i}& \frac{1}{24\alpha_{i}^{l}\prod_{k=1}^{l}(-a_{k})}\Res_{\hbar=0,\infty}
\Big\{\frac{\prod_{k=1}^{l}(\hbar-a_{k}\alpha_{i})\big((\alpha_{i}+\hbar)^{n}-\alpha_{i}^{n}\big)}{\hbar^{3}}\\ &&\cdot\frac{e^{-f(q)\frac{\alpha_{i}}{\hbar}}\mathcal{F}_{-l}\big(\frac{\alpha_{i}}{\hbar},(-1)^{n}q\big)-1}
{e^{-f(q)\frac{\alpha_{i}}{\hbar}}\mathcal{F}_{-l}\big(\frac{\alpha_{i}}{\hbar},(-1)^{n}q\big)}\Big\}\nn\\
&\equiv_{i}& \frac{1}{24\alpha_{i}^{l}\prod_{k=1}^{l}(-a_{k})}\Res_{\hbar=0,\infty}
\Big\{\alpha_{i}^{n-3+l}\frac{\prod_{k=1}^{l}(\frac{\hbar}{\alpha_{i}}-a_{k})\big((1+\frac{\hbar}{\alpha_{i}})^{n}-1\big)}
{(\frac{\hbar}{\alpha_{i}})^{3}}\\ &&\cdot \frac{e^{-f(q)\frac{\alpha_{i}}{\hbar}}\mathcal{F}_{-l}\big(\frac{\alpha_{i}}{\hbar},(-1)^{n}q\big)-1}
{e^{-f(q)\frac{\alpha_{i}}{\hbar}}\mathcal{F}_{-l}\big(\frac{\alpha_{i}}{\hbar},(-1)^{n}q\big)}\Big\}\nn\\
&\equiv_{i}& \frac{\alpha_{i}^{n-2}}{24\prod_{k=1}^{l}(-a_{k})}\Res_{\hbar=0,\infty}
\Big\{\frac{\prod_{k=1}^{l}(\hbar-a_{k})\big((1+\hbar)^{n}-1\big)}{\hbar^{3}} \frac{e^{-\frac{f(q)}{\hbar}}\mathcal{F}_{-l}\big(\frac{1}{\hbar},(-1)^{n}q\big)-1}
{e^{-\frac{f(q)}{\hbar}}\mathcal{F}_{-l}\big(\frac{1}{\hbar},(-1)^{n}q\big)}\Big\}.\nn\\
\eea

By (\ref{254}) and (\ref{222}), (\ref{223}),
\bea\label{256}
\tilde{B}_{ij}(Q)
&=&-\frac{1}{24\prod_{k=1}^{l}(-a_{k})\cdot\alpha_{j}^{l}\prod_{k\in[n]\backslash\{j\}}(\alpha_{j}-\alpha_{k})}\nn\\
&&\cdot\Res_{\hbar=0,\infty}
\Big\{\frac{\prod_{k=1}^{l}(\hbar-a_{k}\alpha_{j})\prod_{k=1}^{n} (\alpha_{j}-\alpha_{k}+\hbar)}{\hbar^{2}} \frac{\mathcal{Z}_{ji}^{*}(\hbar,Q)}{1+\mathcal{Z}_{i}^{*}(\hbar,Q)}\Big\}\nn\\
&\equiv_{i}&-\frac{\alpha_{i}^{n-2}\alpha_{j}^{n-1}}{24\prod_{k=1}^{l}(-a_{k})\cdot\prod_{k\in[n]\backslash\{j\}}
(\alpha_{j}-\alpha_{k})}\nn\\
&&\cdot\Res_{\hbar=0,\infty}
\Big\{\frac{\prod_{k=1}^{l}(\hbar-a_{k})\big((1+\hbar)^{n}-1\big)}{\hbar^{3}} \frac{(1+\hbar Q\frac{d}{dQ})\Big[e^{-\frac{f(q)}{\hbar}}\mathcal{F}_{-l}\big(\frac{1}{\hbar},(-1)^{n}q\big)\Big]-1}
{e^{-\frac{f(q)}{\hbar}}\mathcal{F}_{-l}\big(\frac{1}{\hbar},(-1)^{n}q\big)}\Big\},\nn\\
\eea
where we have used
\ben
&&(1+\hbar Q\frac{d}{dQ})\Big[e^{-\frac{f(q)}{\hbar}}\mathcal{F}_{-l}\big(\frac{1}{\hbar},(-1)^{n}q\big)\Big]\\
&=&e^{-\frac{f(q)}{\hbar}}\mathcal{F}_{-l}\big(\frac{1}{\hbar},(-1)^{n}q\big)
+\Big(\hbar Q\frac{d}{dQ}e^{-\frac{f(q)}{\hbar}}\Big)\mathcal{F}_{-l}\big(\frac{1}{\hbar},(-1)^{n}q\big)
+e^{-\frac{f(q)}{\hbar}}\Big(\hbar Q\frac{d}{dQ}\mathcal{F}_{-l}\big(\frac{1}{\hbar},(-1)^{n}q\big)\Big)\\
&=&e^{-\frac{f(q)}{\hbar}}\mathcal{F}_{-l}\big(\frac{1}{\hbar},(-1)^{n}q\big)
+\Big(-\frac{I_{1,1}(q)-1}{I_{1,1}(q)}\Big)e^{-\frac{f(q)}{\hbar}}\mathcal{F}_{-l}\big(\frac{1}{\hbar},(-1)^{n}q\big)\\
&&+e^{-\frac{f(q)}{\hbar}}\cdot\frac{1}{I_{1,1}(q)}q\frac{d}{dq}\mathcal{F}_{-l}\big(\frac{1}{\hbar},(-1)^{n}q\big)\\
&=&e^{-\frac{f(q)}{\hbar}}\frac{\mathbf{M}\mathcal{F}_{-l}\big(\frac{1}{\hbar},(-1)^{n}q\big)}{I_{1,1}(q)}.
\een
Summing over $j$ we have
\bea\label{257}
\sum_{j=1}^{n}\tilde{B}_{ij}(Q)
&\equiv_{i}&-\frac{\alpha_{i}^{n-2}}{24\prod_{k=1}^{l}(-a_{k})}\Res_{\hbar=0,\infty}
\Big\{\frac{\prod_{k=1}^{l}(\hbar-a_{k})\big((1+\hbar)^{n}-1\big)}{\hbar^{3}} \nn\\
&&\cdot\frac{(1+\hbar Q\frac{d}{dQ})\Big[e^{-\frac{f(q)}{\hbar}}\mathcal{F}_{-l}\big(\frac{1}{\hbar},(-1)^{n}q\big)\Big]-1}
{e^{-\frac{f(q)}{\hbar}}\mathcal{F}_{-l}\big(\frac{1}{\hbar},(-1)^{n}q\big)}\Big\}.
\eea
Therefore combining (\ref{255}) and  (\ref{257}) we obtain
\bea\label{258}
B_{i}(Q)+\sum_{j=1}^{n}\tilde{B}_{ij}(Q)
&\equiv_{i}&-\frac{\alpha_{i}^{n-2}}{24\prod_{k=1}^{l}(-a_{k})}\Res_{\hbar=0,\infty}
\Big\{\frac{\prod_{k=1}^{l}(\hbar-a_{k})\big((1+\hbar)^{n}-1\big)}{\hbar^{2}} \nn\\
&&\cdot\frac{Q\frac{d}{dQ}\Big[e^{-\frac{f(q)}{\hbar}}\mathcal{F}_{-l}\big(\frac{1}{\hbar},(-1)^{n}q\big)\Big]}
{e^{-\frac{f(q)}{\hbar}}\mathcal{F}_{-l}\big(\frac{1}{\hbar},(-1)^{n}q\big)}\Big\}\nn\\
&=&\alpha_{i}^{n-2}Q\frac{d}{dQ}\Big(\tilde{B}_{0}(q)+\tilde{B}_{\infty}(q)\Big),
\eea
where
\bea\label{259}
\tilde{B}_{w}(q)&=&-\frac{1}{24\prod_{k=1}^{l}(-a_{k})}\Res_{\hbar=0,\infty}
\Big\{\frac{\prod_{k=1}^{l}(\hbar-a_{k})\big((1+\hbar)^{n}-1\big)}{\hbar^{2}} \nn\\
&&\cdot \log\Big(e^{-\frac{f(q)}{\hbar}}\mathcal{F}_{-l}\big(\frac{1}{\hbar},(-1)^{n}q\big)\Big)\Big\},
\eea
for $w=0$ or $\infty$. Now we use results in section 5.1 to compute $\tilde{B}_{0}(q)$ and $\tilde{B}_{\infty}(q)$. By lemma \ref{228} and (\ref{230}), we have
\bea\label{260}
\tilde{B}_{0}(q)
&=&-\frac{1}{24\prod_{k=1}^{l}(-a_{k})}\Bigg[
\prod_{k=1}^{l}(-a_{k})\Big(
\frac{n(n-1)}{2}-n\sum_{k=1}^{l}\frac{1}{a_{k}}\Big)\Big(-f(q)+\mu\big((-1)^{n}q\big)\Big)\nn\\
&&+n\prod_{k=1}^{l}(-a_{k})\ln\Phi_{-l,0}\big((-1)^{n}q\big)
\Bigg]\nn\\
&=&
-\frac{n}{48}\Big(n-1-2\sum_{k=1}^{l}\frac{1}{a_{k}}\Big)\Big(-f(q)+\mu\big((-1)^{n}q\big)\Big)
-\frac{n}{24}\cdot\frac{1-l}{-2n}\ln\big(1-\prod_{k=1}^{l}(-a_{k})^{a_{k}}q\big)\nn\\
&=&
-\frac{n}{48}\Big(n-1-2\sum_{k=1}^{l}\frac{1}{a_{k}}\Big)\Big(-f(q)+\mu\big((-1)^{n}q\big)\Big)
+\frac{1-l}{48}\ln\big(1-\prod_{k=1}^{l}(-a_{k})^{a_{k}}q\big),\nn\\
\eea
where we have used
\ben
\frac{\prod_{k=1}^{l}(z-a_{k})\cdot\big((1+z)^{n}-1\big)}{z^{2}}
=\frac{n\prod_{k=1}^{l}(-a_{k})}{z}+\prod_{i=k}^{l}(-a_{k})\Big(
\frac{n(n-1)}{2}-n\sum_{k=1}^{l}\frac{1}{a_{k}}\Big)+O(z).
\een
On the other hand,
\bea\label{261}
\tilde{B}_{\infty}(q)
&=&\frac{1}{24\prod_{k=1}^{l}(-a_{k})}\Res_{w=0}
\Bigg\{\frac{\prod_{k=1}^{l}(1-a_{k}w)\cdot\big((1+w)^{n}-w^{n}\big)}{w^{n+l}}\nn\\
&&\cdot\ln\Big( e^{-Tw}e^{tw}\mathcal{F}_{-l}\big(w,(-1)^{n}q\big)\Big)\Bigg\}\nn\\
&=&\frac{1}{24\prod_{k=1}^{l}(-a_{k})}\Res_{w=0}
\Bigg\{\frac{\prod_{k=1}^{l}(1-a_{k}w)\cdot\big((1+w)^{n}-w^{n}\big)}{w^{n+l}} \Big( -Tw+\ln R(w,t)\Big)\Bigg\},\nn\\
\eea
where we have used the fact $e^{tw}\mathcal{F}_{-l}\big(w,(-1)^{n}q\big)\equiv R(w,t)\mod w^{n+l}$.
\end{proof}

 \titleformat{\chapter}[display]
           {\normalfont\Large\bfseries}{Appendix~\Alph{chapter}}{11pt}{\Large}

\begin{appendices}

\section{The modularity of the genus one Gromov-Witten potential for the local $\mathbb{P}^{2}$}
For $X=K_{\mathbb{P}^{2}}$, the mirror map is $Q=qe^{f(q)}$, where
\bea\label{159}
f(q)=\sum_{d=1}^{\infty}q^{d}\frac{(-1)^{d}3\cdot(3d-1)!}{(d!)^{3}}.
\eea

Let $\psi=-\frac{1}{3}q^{-\frac{1}{3}}$, the genus one free energy given in \cite{ABK} is
\bea\label{151}
\mathcal{F}_{1}&=&-\frac{1}{2}\log(\frac{dT}{d\psi})-\frac{1}{12}\log(1-\psi^{3}).
\eea
Up to a constant, we have
\bea\label{163}
\mathcal{F}_{1}&=&-\frac{1}{12}\log q-\frac{1}{2}\log(I_{1,1})-\frac{1}{12}\log(1+27q),
\eea
where
\bea\label{152}
I_{1,1}(q)=1+qf^{\prime}(q)=\sum_{d=0}^{\infty}q^{d}\frac{(-1)^{d}(3d)!}{(d!)^{3}}=\leftidx{_2}F_{1}(1/3,2/3;1;-27q).
\eea
Note that the meaning of the physicists' genus one free energy is, to get the generating series $\sum_{d=1}^{\infty}Q^{d}N_{1,d}^{X}$ of the Gromov-Witten invariants, one needs to add $\frac{1}{12}\log Q$ to cancel the log-term in (\ref{163}). After doing this, we see that (\ref{163}) coincides with (\ref{265}).\\

Now we recall the definition of the modular coordinate in \cite{ABK} and then show that $\mathcal{F}_{1}$ is a modular form in this coordinate.
Let $\mathfrak{q}=\exp(2\pi i\tau)$ be the coordinate on the modular curve of $\Gamma(3)$. From the mathematical viewpoint, \cite{ABK} gives two ways to relate $\mathfrak{q}$ to $q$. One way is through
\bea\label{158}
j(\tau)=-\frac{27\psi^{3}(8+\psi^{3})^{3}}{(1-\psi^{3})^{3}}=\frac{(216q-1)^{3}}{q(27q+1)^{3}}.
\eea
Let us first take a look at another way. First we have
\bea\label{160}
\sum_{d=1}^{\infty}dN_{0,d}e^{dT}&=&-\frac{1}{3}[\frac{x^{2}}{\hbar^{2}}]\Big(e^{-\frac{xf(q)}{\hbar}}Y(x,\hbar,q)\Big)\nn\\
&=&-\frac{1}{3}[x^{2}]\Big(e^{-xf(q)}
\sum_{d=0}^{\infty}q^{d}\frac{\prod_{s=0}^{3d-1}(-3x-s)}{\prod_{s=1}^{d}(x+s)^3}\Big)\nn\\
&=&\frac{f(q)^{2}}{6}-\frac{1}{3}\sum_{d=1}^{\infty}q^{d}\frac{(-1)^{d}9\cdot(3d-1)!}{(d!)^{3}}\Big(\sum_{s=d+1}^{3d-1}\frac{1}{s}\Big),
\eea
and thus
\bea\label{161}
\frac{d^{2}}{dT^{2}}\sum_{d=1}^{\infty}N_{0,d}e^{dT}&=&
\frac{f(q)}{3}-\frac{f(q)
+\sum_{d=1}^{\infty}q^{d}\frac{(-1)^{d}3\cdot(3d)!}{(d!)^{3}}\Big(\sum_{s=d+1}^{3d-1}\frac{1}{s}\Big)}{3I_{1,1}(q)}\nn\\
&=&\frac{f(q)}{3}-\frac{\sum_{d=1}^{\infty}q^{d}\frac{(-1)^{d}\cdot(3d)!}{(d!)^{3}}\Big(\sum_{s=d+1}^{3d}\frac{1}{s}\Big)}{I_{1,1}(q)}.
\eea
The second way to relate $\mathfrak{q}$ to $q$ is via
\bea\label{162}
\tau &=&-3\cdot2\pi i\frac{d}{dT}\Bigg(-\frac{1}{6}\Big(\frac{T}{2\pi i}\Big)^{2}+\frac{1}{6}\Big(\frac{T}{2\pi i}\Big)\Bigg)\nn\\
&&-3\cdot\frac{1}{2\pi i}\cdot\Big(\frac{f(q)}{3}-
\frac{\sum_{d=1}^{\infty}q^{d}\frac{(-1)^{d}\cdot(3d)!}{(d!)^{3}}\Big(\sum_{s=d+1}^{3d}\frac{1}{s}\Big)}{I_{1,1}(q)}\Big)\nn\\
&=&-\frac{1}{2}+\frac{\log q}{2\pi i}+
\frac{3\sum_{d=1}^{\infty}q^{d}\frac{(-1)^{d}\cdot(3d)!}{(d!)^{3}}\Big(\sum_{s=d+1}^{3d}\frac{1}{s}\Big)}{2\pi i\cdot I_{1,1}(q)}\Big).
\eea
So
\bea\label{153}
\mathfrak{q}&=&-q\exp\Big(\frac{3\sum_{d=1}^{\infty}q^{d}(-1)^{d}\frac{(3d)!}{(d!)^{3}}\sum_{s=d+1}^{3d}\frac{1}{s}}
{\sum_{d=0}^{\infty}q^{d}(-1)^{d}\frac{(3d)!}{(d!)^{3}}}\Big)\nn\\
&=&-q\exp\Big(\frac{3\sum_{d=1}^{\infty}q^{d}(-1)^{d}\frac{(3d)!}{(d!)^{3}}\big(3\Psi(3d)-\Psi(d)-2\Psi(d+1)\big)}
{\sum_{d=0}^{\infty}q^{d}(-1)^{d}\frac{(3d)!}{(d!)^{3}}}\Big),
\eea
where
\bea
\Psi(z)=\frac{\Gamma^{\prime}(z)}{\Gamma(z)}=-\gamma+\sum_{k=1}^{\infty}\big(\frac{1}{k}-\frac{1}{k+z-1}\big).
\eea

By \cite[Chap. 11, entry 26, example 2 of entry 27]{BRa}, we obtain
\bea\label{154}
\mathfrak{q}=\exp\Big(-\frac{2\pi}{\sqrt{3}}\frac{\leftidx{_2}F_{1}(1/3,2/3;1;1+27q)}{\leftidx{_2}F_{1}(1/3,2/3;1;-27q)}\Big).
\eea

The equivalence of the two ways is insured by the following lemma (by abuse of notation we write $j$ as a function of $\mathfrak{q}$), 
which is  \cite[(2.8)]{BC}.
\begin{lemma}
\ben
j\Bigg(\exp\Big(-\frac{2\pi}{\sqrt{3}}\frac{\leftidx{_2}F_{1}(1/3,2/3;1;1+27q)}{\leftidx{_2}F_{1}(1/3,2/3;1;-27q)}\Big)\Bigg)
=\frac{(216q-1)^{3}}{q(1+27q)^{3}}.
\een
\end{lemma}

The modularity of $\mathcal{F}_{1}$ in the coordinate $\mathfrak{q}$ is a consequence of \emph{Ramanujan's cubic transformation}.
First we write $\mathcal{F}_{1}$ as
\bea\label{155}
\mathcal{F}_{1}&=&-\frac{1}{24}\log\Big(\big(\frac{\partial T}{\partial \psi}\big)^{12}(1-\psi^{3})^{3}\Big)
+\frac{1}{24}\log(1-\psi^{3})\nn\\
&=&-\frac{1}{24}\log\big(q(1+27q)^{3}I_{1,1}^{12}\big)+\frac{1}{24}\log\big(\frac{1+27q}{27q}\big)+\mathrm{Const}.
\eea

By (\ref{152}), (\ref{154}), and \cite[(2.1), (2.5)]{BC}, it is not hard to show
\begin{lemma}
\bea\label{156}
\Delta(\mathfrak{q})
=-q(1+27q)^{3}I_{1,1}^{12},
\eea
where $\Delta(\mathfrak{q})=\mathfrak{q}\prod_{n=1}^{\infty}(1-\mathfrak{q}^{n})^{24}$ is the Ramanujan's tau function.
\end{lemma}

Following \cite{ABK}, let
\ben
d(\mathfrak{q})=\Big(\theta_{2}\big(\frac{\pi}{6},\mathfrak{q}^{\frac{1}{2}}\big)\Big)^{3},
\een
Then by the infinite product formula for Jacobi theta functions (see for example \cite{WW})
\ben
\theta_{2}\big(\frac{\pi}{6},\mathfrak{q}^{\frac{1}{2}}\big)=\sqrt{3}\mathfrak{q}^{\frac{1}{8}}\prod_{n=1}^{\infty}(1-\mathfrak{q}^{3n}),
\een
and the infinite product formula for the Dedekind's eta function
\ben
\eta(\mathfrak{q})=\mathfrak{q}^{\frac{1}{24}}\prod_{n=1}^{\infty}(1-\mathfrak{q}^{n}),
\een
together with  \cite[(2.7)]{Chan}, we obtain
\bea\label{164}
\frac{1+27q}{27q}=-\frac{27\eta^{12}(\mathfrak{q})}{d^{4}(\mathfrak{q})}.
\eea
So by (\ref{153}), (\ref{156}) and (\ref{164}) we have
\begin{theorem}
Up to a constant,
\bea\label{157}
\mathcal{F}_{1}=-\frac{1}{6}\log\big(d(\mathfrak{q})\eta^{3}(\mathfrak{q})\big).
\eea
\end{theorem}

 \end{appendices}

\hspace{1cm}Beijing international center for mathematical research, Beijing, 100871, China\\

\hspace{1cm}\emph{E-mail address}: huxw06@gmail.com

\end{document}